\documentclass[A4paper,11pt]{amsart}

\usepackage[english]{babel}
\usepackage{amsmath}
\usepackage{amssymb} 
\usepackage{amsfonts}
\usepackage[all]{xy}
\usepackage{hyperref}
\usepackage{color}
\usepackage{lscape}

 \newtheorem{proposition}{Proposition}[section] 
 \newtheorem{lemma}[proposition]{Lemma}
 \newtheorem*{lemma*}{Lemma}
 \newtheorem{theorem}[proposition]{Theorem}
 
 \theoremstyle{definition}
 \newtheorem{definition}[proposition]{Definition}
 \newtheorem*{definition*}{Definition}
 \newtheorem{example}[proposition]{Example}
 \newtheorem*{example*}{Example}
 \newtheorem{remark}[proposition]{Remark}
 
\numberwithin{equation}{section}

\def\ox{\otimes}
\def\M{\mathbb M}

\def\op{\mathsf{op}}
\def\bl{}
\def\br{}
\def\act{}

\begin{document}

\title{Comodules over weak multiplier bialgebras} 
\author{Gabriella B\"ohm}
\address{Wigner Research Centre for Physics, H-1525 Budapest 114,
P.O.B.\ 49, Hungary} 
\email{bohm.gabriella@wigner.mta.hu}
\thanks{I am grateful to Jos\'e G\'omez-Torrecillas and Alfons Van Daele for
 highly enlightening discussions; and to the referee for several helpful
 comments. My research was supported by the Hungarian Scientific
 Research Fund OTKA, grant K108384.}  
\begin{abstract}
This is a sequel paper of \cite{BoGTLC:wmba} in which we study the 
comodules over a regular weak multiplier bialgebra over a field, with a full
comultiplication. Replacing the usual notion of coassociative coaction over a
(weak) bialgebra, a comodule is defined via a pair of compatible linear
maps. Both the total algebra and the base (co)algebra of a regular weak
multiplier bialgebra with a full comultiplication are shown to carry comodule
structures. Kahng and Van Daele's integrals \cite{VDa:int} are interpreted as
comodule maps from the total to the base algebra. 
Generalizing the counitality of a comodule to the multiplier setting, we
consider the particular class of so-called full comodules. 
They are shown to carry bi(co)module structures over the base (co)algebra and 
constitute a monoidal category via the (co)module tensor
product over the base (co)algebra. If a regular weak multiplier bialgebra 
with a full comultiplication possesses an antipode, then finite dimensional
full comodules are shown to possess duals in the monoidal category of full
comodules. Hopf modules are introduced over regular weak multiplier bialgebras
with a full comultiplication. Whenever there is an antipode, the Fundamental
Theorem of Hopf Modules is proven. It asserts that the category of Hopf
modules is equivalent to the category of firm modules over the base algebra. 
\end{abstract} 
\date{Nov 2013}
\maketitle

\section*{Introduction}

The categories of modules and comodules over a weak bialgebra (say, over a
field), are well-studied in the literature. The most important features are
the following.
To any weak bialgebra $A$ over a field, there is an associated separable
Frobenius algebra $R$ (known as the `base (co)algebra'), see \cite{Nill, WHAI,
BoCaJa}. It is a subalgebra and a quotient coalgebra of $A$. Both the category
$M_A$ of $A$-modules and the category $M^A$ of $A$-comodules are monoidal,
admitting strict monoidal forgetful functors to the category ${}_R M_R$ of
$R$-bimodules \cite{WHAII,BoCaJa}. If $A$ is a weak Hopf algebra, then finite
dimensional modules and comodules possess duals in the appropriate monoidal
category \cite{WHAII}, see also \cite{Scha:DuDoQgp, PhHH}.

The aim of this paper is to study analogous questions about {\em weak
multiplier bialgebras} in \cite{BoGTLC:wmba}. These are generalizations of weak
bialgebras in the spirit of \cite{VDae:MHA,VDaWa:Banach,VDaWa}. That is, the
algebra underlying a weak multiplier bialgebra $A$, is not required to possess
a unit. Instead, the multiplication is assumed to be non-degenerate and
surjective, so that the notion of multiplier algebra $\M(A)$
\cite{Dauns:Multiplier} is available. The comultiplication is a
multiplicative map from $A$ to the multiplier algebra of $A\otimes A$. It is
subject to certain compatibility axioms in \cite[Definition
2.1]{BoGTLC:wmba}. A central role is played by a canonical idempotent element
$E$ in the multiplier algebra of $A\ox A$. Whenever the comultiplication is
regular and full (in the sense discussed in \cite[Definition 2.3 and Theorem
3.13]{BoGTLC:wmba}, respectively), there is a co-separable co-Frobenius
coalgebra $R$ associated to $A$, see \cite[Theorem 4.6]{BoGTLC:wmba} (it is a
non-unital subalgebra of the multiplier algebra of $A$). As shown in
\cite{BoVe}, $R$ possesses then a firm algebra structure and the monoidal
category of $R$-bicomodules is isomorphic to the monoidal category ${}_R M_R$
of firm $R$-bimodules. As a crucial difference from usual weak bialgebras, a
weak multiplier bialgebra $A$ is not known to induce any monad or
comonad. Still, whenever the comultiplication is regular and full, it was
shown in \cite[Section 5]{BoGTLC:wmba} that the category $M_{(A)}$, of
non-degenerate $A$-modules with surjective action, is monoidal and it
possesses a strict monoidal (forgetful) functor $U_{(A)}:M_{(A)} \to {}_R M_R$. 

In this paper we continue the study started in \cite{BoGTLC:wmba} by analyzing
the category of comodules.
Comodules over a weak multiplier bialgebra $A$ need to be defined in such a way
that in particular $A$ itself is a comodule via the comultiplication. Since in
this particular case the comultiplication $\Delta$ lands in the multiplier
algebra of $A\ox A$ (rather than in $A\ox A$), the existence of a coaction
from any comodule $V$ to $V\ox A$ would be too much to expect. However, we can
also not work with a `multiplier valued' coaction, what would be the most
naive generalization of the comultiplication $\Delta$, as the notion of
multiplier on $V\ox A$ is not available for an arbitrary vector space
$V$. Instead, one can observe that the comultiplication of a regular weak
multiplier bialgebra $A$ is uniquely encoded in the maps $T_1(a\ox
b):=\Delta(a)(1_{\M(A)}\ox b)$ and $T_3(a\ox b):=(1_{\M(A)}\ox b)\Delta(a)$ from
$A\ox A$ to $A\ox A$. In a similar spirit, in the same way as in
\cite{VDaZha:corep_I,VDaZha:corep_II} in the non-weak case, we define a right 
$A$-comodule $V$ in terms of a compatible pair of maps $V\ox A\to V\ox A$ 
(instead of a map $V\to V\ox A$). As a generalized coassociativity condition,
we require the same pentagonal equation in \cite{VDaZha:corep_I}, but we
supplement it with a normalization condition in terms of the canonical
idempotent element $E$. The counitality condition is replaced by the property
of `fullness': the `first leg' of either coaction map $V\ox A \to V\ox A$ is
required to generate the vector space $V$. 

Throughout the paper, we deal with regular weak multiplier bialgebras 
over a field. Their definition and some of the most important properties are
recalled from \cite{BoGTLC:wmba} in Section \ref{sec:prelims}, where also some
new technical lemmata are proven. The axioms of weak multiplier bialgebra are
presented in a somewhat new but equivalent form. In Section \ref{sec:comod} we
spell out the notions of comodule and its morphisms. We show that any regular
weak multiplier bialgebra $A$ is a right comodule over itself. 
Whenever in addition the comultiplication is right full, there is another
distinguished right comodule on the base algebra $R$. In this case, in
Section \ref{sec:int} we describe the integrals in \cite{VDa:int} as comodule
homomorphisms from $A$ to $R$. 
Under the same assumption of right full comultiplication, in Section
\ref{sec:bim} we equip any full right $A$-comodule with a firm $R$-bimodule
structure.
In Section \ref{sec:mon_cat}, the full subcategory of full right $A$-comodules
is shown to be monoidal, in such a way that there is a strict monoidal forgetful
functor to the category of firm $R$-bimodules. A similar result for multiplier
Hopf algebras was obtained in \cite{VDaZha:corep_I}. In Section \ref{sec:dual}
we show that the dual vector space of any finite dimensional right
$A$-comodule admits a left $A$-comodule structure. The dual of a finite
dimensional full right comodule is proven to be an (obviously finite
dimensional) full left comodule. If $A$ possesses an antipode, then it 
is shown to induce a (full) right $A$-comodule structure on any (full) left
$A$-comodule; in particular it induces a right comodule structure on
the dual vector space of any finite dimensional right comodule. A finite
dimensional full right comodule, and the corresponding full right comodule on
its dual vector space are shown to be duals in the monoidal category of full
right comodules. In Section \ref{sec:Hopf_mod} we introduce Hopf modules over a
regular weak multiplier bialgebra $A$. They are both non-degenerate
$A$-modules with a surjective action, and full $A$-comodules, in a compatible
way. In particular, whenever the comultiplication is right full, $A$ itself
is an $A$-Hopf module. Assuming that the comultiplication is both left and
right full and there exists an antipode, for any $A$-Hopf module $V$ we
discuss a distinguished subspace of the vector space of Hopf module
homomorphisms $A\to V$. It plays the role of the space of $A$-coinvariants in
$V$ and it can be smaller indeed than the vector space of Hopf module
homomorphisms $A\to V$. For example, the space of coinvariants of the $A$-Hopf
module $A$ is isomorphic to the base algebra of $A$ (which is a non-unital
subalgebra of the unital algebra of Hopf module endomorphisms of $A$). For a
regular weak multiplier bialgebra $A$ with left and right full
comultiplication and possessing an antipode, we prove the Fundamental Theorem 
of Hopf Modules; that is, an equivalence of the category of $A$-Hopf modules
and the category of firm modules over the base algebra of $A$. 

\section{Preliminaries}\label{sec:prelims}

In this section we recall from \cite{BoGTLC:wmba} the definition and the basic
properties of weak multiplier bialgebra. We also prove some new technical
lemmata for use in the later sections.

\subsection{Notation}\label{sec:notation}

Throughout the paper, for any vector space $V$, we denote the identity map
$V\to V$ also by $V$. On elements $v\in V$, also the notation $v\mapsto 1v$
or $v\mapsto v1$ is used, meaning multiplication with the unit $1$ of the base
field. For any vector spaces $V$ and $W$, the space of linear maps $V\to W$ is
denoted by $\mathsf{Lin}(V,W)$. For any subset $I$ of a vector space $V$,
$\langle v\ |\ v\in I\rangle$ denotes the subspace generated by $I$ in
$V$. For linear maps $f:V\to W$ and $g:W\to Z$, the composite is denoted by
juxtaposition $gf$. The tensor product of vector spaces is denoted by
$\ox$. For vector spaces $V$ and $W$, we denote by $\mathsf{tw}$ the flip map 
$$ 
\mathsf{tw}:V\ox W\to W\ox V,\qquad v\ox w\mapsto w\ox v.
$$
For any linear map $f:V\ox W\to V'\ox W'$, we consider the associated linear
maps $f^{21}:=\mathsf{tw}f\mathsf{tw}:W\ox V\to W'\ox V'$, 
$f^{13}:=(V'\ox \mathsf{tw})(f\ox Z)(V\ox \mathsf{tw}):V\ox Z\ox W \to V'\ox
Z\ox W'$, $f^{31}:=(W'\ox \mathsf{tw})(f^{21}\ox Z)(W\ox \mathsf{tw}):W\ox
Z\ox V \to W'\ox Z\ox V'$ and so on, for any vector space $Z$. 

\subsection{Multiplier algebra}\label{sec:multiplier}

Let $A$ be an associative algebra over a field, with multiplication $\mu:A\ox
A \to A$, $a\ox b \mapsto ab$, possibly possessing no unit. It is said to be
{\em idempotent} if $\mu$ is surjective; that is, $A$ is spanned by elements
of the form $ab$ for $a,b\in A$. It is termed {\em non-degenerate} if any of
the conditions $(ab=0\ \forall b\in A)$ and $(ba=0\ \forall b\in A)$ implies
$a=0$. A {\em multiplier} \cite{Dauns:Multiplier} on an idempotent
non-degenerate algebra $A$ is a pair $(\lambda,\varrho)$ of linear maps $A\to
A$ such that 
\begin{equation}\label{eq:mp}
a\lambda(b)=\varrho(a)b,\qquad \forall a,b\in A.
\end{equation}
For any multiplier $(\lambda,\varrho)$ on $A$, $\lambda$ is a morphism of right
$A$-modules and $\varrho$ is a morphism of left $A$-modules. Multipliers on $A$
constitute a unital associative algebra, the so-called {\em multiplier
algebra} $\M(A)$ with multiplication $(\lambda,\varrho)(\lambda',\varrho')=
(\lambda\lambda',\varrho'\varrho)$ and unit $1=(A,A)$. The multiplication in
$\M(A)$ will be denoted by $\mu$ too. Any element $a$ of $A$ can be regarded
as a multiplier via the embedding 
$$
A\to \M(A),\qquad a\mapsto (a(-),(-)a).
$$
This makes $A$ a two-sided ideal in $\M(A)$. Indeed, 
$$
a(\lambda,\varrho)=\varrho(a)\qquad \textrm{and}\qquad
(\lambda,\varrho)a=\lambda(a).
$$
The ideal $A$ is dense in $\M(A)$ in the sense that for $\nu\in \M(A)$, any
of the conditions $(\nu b=0\ \forall b\in A)$ and $(b\nu=0\ \forall b\in
A)$ implies $\nu=0$.

The opposite $A^{\mathsf{op}}$ of an idempotent non-degenerate algebra $A$ is
the same vector space $A$ equipped with the multiplication
$\mu^{\mathsf{op}}:a\ox b\mapsto ba$. It is again an idempotent non-degenerate
algebra and $\M(A^{\mathsf{op}})\cong \M(A)^{\mathsf{op}}$.

The tensor product $A\ox B$ of idempotent non-degenerate algebras $A$ and $B$
is again an idempotent non-degenerate algebra via the factorwise
multiplication. We extend the leg numbering notation of linear maps in Section
\ref{sec:notation} to multipliers on $A\ox B$ by putting
$(\lambda,\varrho)^{ij}:=(\lambda^{ij},\varrho^{ij})$, for any labels $i,j$.

The following result, about the extension of maps to the multiplier algebra, is
due to Van Daele and Wang \cite[Proposition A.3]{VDaWa:Banach}: Let $A$ and
$B$ be idempotent non-degenerate algebras. Let $\phi:A\to \M(B)$ be a
multiplicative linear map and $e$ be an idempotent element of $B$ (i.e. such
that $e^2=e$). If 
$$
\langle \phi(a)b\ |\ a\in A,b\in B\rangle=\langle eb\ | \ b\in B\rangle
\quad \textrm{and}\quad
\langle b\phi(a)\ |\ a\in A,b\in B\rangle=\langle be\ | \ b\in B\rangle
$$
then there is a unique multiplicative map $\overline \phi:\M(A)\to \M(B)$ such
that $\overline \phi(1)=e$ and $\overline \phi(a)=\phi(a)$ for any $a\in A$.

\subsection{Weak multiplier bialgebra}\label{sec:wmba}

We present here the definition of weak multiplier bialgebra in
\cite[Definition 2.1]{BoGTLC:wmba} in a slightly different, but equivalent
form. Recall from Section \ref{sec:notation} that for any vector space $A$,
we denote both as $a\mapsto 1a$ and as $a\mapsto a1$ the action of the
identity map $A$ on elements $a\in A$, where $1$ means the unit of the base
field.

\begin{definition}\label{def:wmba}
A {\em weak multiplier bialgebra} over a field $k$ is given by an idempotent
non-degenerate $k$-algebra $A$ with multiplication $\mu:A\ox A \to A$, and
linear maps $E_1,E_2,T_1,T_2:A\ox A \to A\ox A$ and a linear map
$\epsilon:A\to k$ --- called the {\em counit} --- subject to the following
axioms. 
\begin{itemize}
\item[{(i)}] For all $a,b,c,d\in A$, $(a\ox b)E_1(c\ox d)=E_2(a\ox b)(c\ox d)$.
\item[{(ii)}] $E_1^2=E_1$, equivalently, $E_2^2=E_2$.
\item[{(iii)}] $(T_2\ox A)(A\ox T_1)=(A\ox T_1)(T_2\ox A)$.
\item[{(iv)}] $(\epsilon \ox A)T_1=\mu=(A\ox \epsilon)T_2$.
\item[{(v)}] $(\mu \ox A)T_1^{13}(A\ox T_1)=T_1(\mu\ox A)$,
equivalently,  
$(A\ox \mu)T_2^{13}(T_2\ox A)=T_2(A\ox \mu)$.
\item[{(vi)}] $\langle T_1(a\ox c)(b\ox 1)\ |\ a,b,c\in A\rangle =
\langle E_1(b\ox c)\ |\ b,c\in A\rangle$ and \\
$\langle (1\ox c)T_2(b\ox a)\ |\ a,b,c\in A\rangle =
\langle E_2(b\ox c)\ |\ b,c\in A\rangle$.
\item[{(vii)}] $(E_1\ox A)(A\ox T_1)=(A\ox T_1)(E_1\ox A)$ 
(equivalently, \\
$(A\ox E_2)(T_2\ox A)=(T_2\ox A)(A\ox E_2)$) and \\
$(E_2\ox A)(A\ox T_2)=(A\ox T_2)E_2^{13}$
(equivalently, $(A\ox E_1)(T_1\ox A)=(T_1\ox A)E_1^{13}$). 
\item[{(viii)}] For all $a,b,c\in A$, \\
$(\epsilon \ox A)((1\ox a)E_1(b\ox c))=
(\epsilon \ox A)(T_1(a\ox c)(b\ox 1))$ and \\
$(\epsilon \ox A)(E_2(a\ox b)(1\ox c))=
(\epsilon \ox A)((1\ox b)T_2(a\ox c))$.
\end{itemize}
The weak multiplier bialgebra $A$ is said to be {\em regular} if there exist
further two maps $T_3,T_4:A\ox A \to A\ox A$ such that for all $a,b,c\in A$,
\begin{itemize}
\item[{(ix)}] $(1\ox a)T_1(c\ox b)=T_3(c\ox a)(1\ox b)$ and 
$T_2(a\ox c)(b\ox 1)=(a\ox 1)T_4(b\ox c)$.
\end{itemize}
\end{definition}

Throughout the paper, we shall often use the index notation $T_i(a\ox
b)=:a^i\ox b^i$, for $i\in \{1,2,3,4\}$, where implicit summation is
understood. 

It follows from the axioms in Definition \ref{def:wmba} that
\begin{eqnarray*}
(\mu \ox A)(A\ox T_1)&\stackrel{\textrm{(iv)}}=&
(A\ox \epsilon \ox A)(T_2\ox A)(A\ox T_1)\\
&\stackrel{\textrm{(iii)}}=&
(A\ox \epsilon \ox A)(A\ox T_1)(T_2\ox A)\stackrel{\textrm{(iv)}}=
(A\ox \mu)(T_2\ox A).
\end{eqnarray*}
That is, for all $a,b,c\in A$, 
\begin{equation}\label{eq:iii}
(a\ox 1)T_1(c\ox b)=T_2(a\ox c)(1\ox b).
\end{equation}

\begin{theorem}\label{thm:def_eq}
The definition of (regular) weak multiplier bialgebra in Definition
\ref{def:wmba} above, is equivalent to that in 
\cite[Definition 2.1 (and Definition 2.3)]{BoGTLC:wmba}. 
\end{theorem}

\begin{proof}
Axiom (i) in Definition \ref{def:wmba} is equivalent to saying that there is a
multiplier $E=(E_1,E_2)$ on $A\ox A$ and axiom (ii) is equivalent to $E$ being
idempotent. Identity \eqref{eq:iii} is equivalent to the existence of a linear
map $\Delta:A\to \M(A\ox A)$ --- the so-called {\em comultiplication} ---
defined by 
$$
\Delta(c)(a\ox b):=T_1(c\ox b)(a\ox 1)
\qquad \textrm{and}\qquad 
(a\ox b)\Delta(c):=(1\ox b)T_2(a\ox c)
$$
(so that $T_1(c\ox b)=\Delta(c)(1\ox b)$ and $T_2(a\ox c)=(a\ox 1)\Delta(c)$).
By the non-degeneracy of the right $A$-module $A\ox A$ (via multiplication in
the second factor), and by the non-degeneracy of the left $A$-module $A\ox A$
(via multiplication in the first factor), either equality in axiom (v) is
equivalent to the multiplicativity of the map $\Delta$. Axioms 
(iii), (iv), (vi) and (viii) are literally the same as in \cite[Definition
2.1]{BoGTLC:wmba}. Let us investigate axiom (vii). Its first equality is
equivalent to 
\begin{eqnarray*}
((E_1\ox A)(A\ox T_1)(a\ox b\ox c))(1\ox d\ox 1)&=&
((A\ox T_1)(E_1\ox A)(a\ox b\ox c))(1\ox d\ox 1)\\ 
\Leftrightarrow\qquad
(E\ox 1)(a\ox \Delta(b)(d\ox c))&=&(A\ox \Delta)(E(a\ox b))(1\ox d\ox c)
\end{eqnarray*}
holding true, for any $a,b,c,d\in A$. By (i), (ii), \eqref{eq:iii}, (v)
and (vi), $A\ox \Delta$ extends to a unique multiplicative
map $\overline{A\ox  \Delta}:\M(A\ox A)\to \M(A\ox A\ox A)$ such that
$(\overline{A\ox\Delta})(1)=1\ox E$. Hence for any $w\in \M(A\ox A)$,  
$(\overline{A\ox \Delta})(w)=
(\overline{A\ox \Delta})(w1)= 
(\overline{A\ox \Delta})(w)(\overline{A\ox \Delta})(1)= 
(\overline{A\ox \Delta})(w)(1\ox E)$. 
In terms of this extended map, the first equality in (vii) is further
equivalent to
$$
(E\ox 1)(a\ox \Delta(b)(d\ox c))=
(\overline{A\ox \Delta})(E)(a\ox \Delta(b)(d\ox c))
$$
holding true for any $a,b,c,d\in A$. By (vi), this is equivalent to the
validity of
\begin{eqnarray}\label{eq:E_coproduct}
(E\ox 1)(1\ox E)(a\ox d\ox c)&=&
(\overline{A\ox \Delta})(E)(1\ox E)(a\ox d\ox c)\quad \forall a,d,c\in A
\qquad\Leftrightarrow\nonumber\\ 
(E\ox 1)(1\ox E)&=&(\overline{A\ox \Delta})(E),
\end{eqnarray}
which is one of the axioms in \cite[Definition 2.1]{BoGTLC:wmba}~(v). One
proves symmetrically the equivalence of the second equality in (vii) and of
the other axiom in \cite[Definition 2.1]{BoGTLC:wmba}~(v).

If $A$ is a regular weak multiplier bialgebra in the sense of 
\cite[Definition 2.3]{BoGTLC:wmba}, then $T_3(a\ox b):=(1\ox b)\Delta(a)$ and
$T_4(a\ox b):= \Delta(b)(a\ox 1)$ clearly obey (ix). In the opposite
direction, we need to show that the maps $T_3$ and $T_4$ satisfying (ix) must
be of this form. Using the first equality in (ix), it follows for any
$a,b,c,d\in A$ that
$$
(a\ox 1)T_3(c\ox b)(1\ox d)=
(a\ox b)T_1(c\ox d)=
(a\ox b)\Delta(c)(1\ox d).
$$
Using the non-degeneracy of $A$ and simplifying by $a$ and $d$, we conclude
that $T_3$ is of the desired form. The claim about $T_4$ follows symmetrically
by the second equality in (ix). 
\end{proof}

If the algebra $A$ in Definition \ref{def:wmba} is in addition unital, then we
obtain an equivalent definition of usual weak bialgebra in \cite{WHAI}, in
light of \cite[Theorem 2.10]{BoGTLC:wmba}. On the other hand, if we add any of
the equivalent conditions 
\begin{itemize}
\item $E_1=A\ox A$
\item $E_2=A\ox A$
\item $\epsilon\mu=\epsilon \ox \epsilon$
\end{itemize}
in Definition \ref{def:wmba}, then we obtain an equivalent definition of
non-weak multiplier bialgebra over a field in \cite[Theorem
2.11]{BoGTLC:wmba}.  

Note that if $A$ is a regular weak multiplier bialgebra, then its
comultiplication $\Delta:A\to \M(A\ox A)$ can be defined equivalently by the
prescriptions
$$
\Delta(c)(a\ox b):=T_1(c\ox b)(a\ox 1)
\qquad \textrm{and}\qquad 
(a\ox b)\Delta(c):=(a\ox 1)T_3(c\ox b).
$$
Using the first identity in axiom (ix) in the first equality and (iv) in the
second one, for all $a,b,c\in A$ 
$$
((\epsilon \ox A)T_3(b\ox a))c=
a((\epsilon \ox A)T_1(b\ox c))=
abc.
$$
Symmetrically, by the second equality in (ix) and by (iv), 
$$
c((A\ox \epsilon)T_4(b\ox a))=
((A\ox \epsilon)T_2(c\ox a))b=
cab.
$$
Thus by the non-degeneracy of $A$,
\begin{equation}\label{eq:reg_counit}
(\epsilon \ox A)T_3=
\mu^\op=
(A\ox \epsilon)T_4.
\end{equation}
A symmetric reasoning shows that \eqref{eq:reg_counit} is in fact an
equivalent form of the counitality axiom (iv) in Definition \ref{def:wmba}. 
The axioms in Definition \ref{def:wmba} imply the further identities
\begin{eqnarray}
\label{eq:T_14}
T_1(c\ox b)(a\ox 1)=&\Delta(c)(a\ox b)&=T_4(a\ox c)(1\ox b)\\
\label{eq:T_23}
(1\ox b)T_2(a\ox c)=&(a \ox b)\Delta(c)&=(a\ox 1)T_3(c\ox b)\\
\label{eq:T_34}
T_3(c\ox b)(a\ox 1)=&(1\ox b)\Delta(c)(a\ox 1)&=(1\ox b)T_4(a\ox c)
\end{eqnarray}
for any regular weak multiplier bialgebra $A$ and all $a,b,c\in A$. It is
immediate from axioms (ii) and (vi) that $E_1T_1=T_1$ and $E_2T_2=T_2$ for any
weak multiplier bialgebra; and by (\ref{eq:T_14}-\ref{eq:T_23}), also
$E_2T_3=T_3$ and $E_1T_4=T_4$ in the regular case. 

By \cite[Proposition 2.4 and Proposition 2.6]{BoGTLC:wmba}, for any weak
multiplier bialgebra $A$ there exist linear maps $\overline \sqcap^L,
\overline \sqcap^R:A\to \M(A)$ defined by the prescriptions 
\begin{eqnarray}\label{eq:Pibar}
&\overline \sqcap^L(a)b:=(\epsilon \ox A)T_2(a\ox b)
\qquad \textrm{and}\qquad 
&b\overline \sqcap^L(a):=(\epsilon \ox A)E_2(a\ox b)\\
&\overline \sqcap^R(a)b:=(A\ox \epsilon)E_1(b\ox a)
\qquad \textrm{and}\qquad 
&b\overline \sqcap^R(a):=(A\ox \epsilon)T_1(b\ox a).
\nonumber
\end{eqnarray}
They obey similar properties to the analogous maps in a usual, unital weak
bialgebra in \cite{WHAI}, see \cite{BoGTLC:wmba}. If $A$ is regular, then
there are further two similar linear maps $\sqcap^L,\sqcap^R:A\to \M(A)$
defined as in (3.1), (3.2) and (3.3) in \cite{BoGTLC:wmba}, by the
prescriptions 
\begin{eqnarray}\label{eq:Pi}
&\sqcap^L(a)b:=(\epsilon \ox A)E_1(a\ox b)
\qquad \textrm{and}\qquad 
&b \sqcap^L(a):=(\epsilon \ox A)T_4(a\ox b)\\
&\sqcap^R(a)b:=(A\ox \epsilon)T_3(b\ox a)
\qquad \textrm{and}\qquad 
&b\sqcap^R(a):=(A\ox \epsilon)E_2(b\ox a).
\nonumber
\end{eqnarray}

\begin{lemma}\label{lem:Pi_on_E}
For any weak multiplier bialgebra $A$ and any $a,b\in A$, the
following identities hold.
\begin{itemize}
\item[{(1)}] $(\overline\sqcap^R\ox A)E_1(a\ox b)=
E(\overline\sqcap^R(a)\ox b)$.
\item[{(2)}] $(A\ox \overline\sqcap^L)E_2(a\ox b)=
(a\ox \overline\sqcap^L(b))E$.
\end{itemize}
If $A$ is regular, then also the following hold for any $a,b\in A$.
\begin{itemize}
\item[{(3)}] $(\sqcap^R\ox A)E_2(a\ox b)=
(\sqcap^R(a)\ox b)E$.
\item[{(4)}] $(A\ox \sqcap^L)E_1(a\ox b)=
E(a\ox \sqcap^L(b))$.
\end{itemize}
\end{lemma}

\begin{proof}
We only prove part (1), all other parts are proven analogously.
For any $b,c\in A$ introduce the notation $T_1(b\ox c)=b^1\ox c^1$ where
implicit summation is understood. Then for any $a\in A$, 
$$
(\overline\sqcap^R\ox A)(E(a\ox bc))=
\overline\sqcap^R(\overline\sqcap^R(b^1)a)\ox c^1=
\overline\sqcap^R(b^1)\overline\sqcap^R(a)\ox c^1=
E(\overline\sqcap^R(a)\ox bc).
$$
From this we conclude by the idempotency of $A$. The first and the last
equalities follow by identity (2.3) in \cite{BoGTLC:wmba} and the middle one
follows by \cite[Lemma 3.4]{BoGTLC:wmba}.
\end{proof}

\begin{lemma}\label{lem:prep_d}
For any regular weak multiplier bialgebra $A$ and any $a,b\in A$, the
following identities hold.
\begin{itemize}
\item[{(1)}] $(A\ox \sqcap^L)T_1(a\ox b)=E(a\overline \sqcap^R(b)\ox 1)$.
\item[{(2)}] $(\overline\sqcap^R\ox \sqcap^L)T_1(a\ox b)=
E(\overline \sqcap^R(ab)\ox 1)$.
\end{itemize}
\end{lemma}

\begin{proof}
(1) For any $a,b,c\in A$,
\begin{eqnarray*}
((A\ox \sqcap^L)T_1(a\ox b))(c\ox 1)&=&
(A\ox \sqcap^L)(\Delta(a)(c\ox b))=
(A\ox \sqcap^L)(\Delta(a)(c\ox \overline \sqcap^R(b)))\\
&=&(A\ox \sqcap^L)T_4(c\ox a\overline\sqcap^R(b))=
E(a\overline \sqcap^R(b)c\ox 1),
\end{eqnarray*}
from which we conclude by the non-degeneracy of $A$. In the second equality we
used that by \eqref{eq:Pi} and \cite[Lemma 3.1]{BoGTLC:wmba}, for any
$a,b,c\in A$, 
$$
\sqcap^L(a\overline\sqcap^R(b))c=
(\epsilon \ox A)(E(a\overline\sqcap^R(b)\ox c))=
(\epsilon \ox A)(E(ab\ox c))=
\sqcap^L(ab)c,
$$
hence $\sqcap^L(ab)=\sqcap^L(a\overline\sqcap^R(b))$.
The third equality follows by \cite[Lemma 3.3]{BoGTLC:wmba} and the
multiplicativity of the comultiplication. The last equality follows by (3.4)
in \cite{BoGTLC:wmba}.

(2) For any $a,b\in A$,
$$
(\overline\sqcap^R\ox \sqcap^L)T_1(a\ox b)\stackrel{(1)}=
(\overline\sqcap^R\ox A)(E(a\overline\sqcap^R(b)\ox 1))=
E(\overline\sqcap^R(a\overline\sqcap^R(b))\ox 1)=
E(\overline\sqcap^R(ab)\ox 1),
$$
where the second equality follows by Lemma \ref{lem:Pi_on_E}~(1) and the 
last equality follows by \cite[Lemma 3.2]{BoGTLC:wmba}.
\end{proof}

By Lemma 3.4 and identities (3.8) in \cite{BoGTLC:wmba}, the ranges of the
maps \eqref{eq:Pibar} and \eqref{eq:Pi} are (non-unital) subalgebras of
$\M(A)$. In particularly nice cases they carry even more structure: Let $A$ be
a regular weak multiplier bialgebra over a field $k$. As in \cite[Theorem
3.13]{BoGTLC:wmba}, we say that its comultiplication is {\em right full} if
$$
\langle (A\ox \omega)T_1(a\ox b)\ |\ a,b\in A, \omega \in \mathsf{Lin}(A,k)
\rangle=A
$$
equivalently, 
$$
\langle (A\ox \omega)T_3(a\ox b)\ |\ a,b\in A, \omega \in \mathsf{Lin}(A,k)
\rangle=A.
$$
These conditions hold if and only if the ranges of the maps $\overline
\sqcap^R$ and $\sqcap^R$ coincide. Let us denote this coinciding range by
$R$. It carries a coalgebra structure as follows. By \cite[Proposition
4.3~(1)]{BoGTLC:wmba}, there is a multiplier $F$ on $A\ox A$ defined by the
prescriptions 
\begin{equation}\label{eq:F}
F(a\ox bc):=((\overline \sqcap^R\ox A)\mathsf{tw}T_4(c\ox b))(a\ox 1)
\qquad \textrm{and}\qquad 
(ab\ox c)F:=(1\ox c)((A\ox \sqcap^R)T_2(a\ox b)).
\end{equation}
By \cite[Theorem 4.4]{BoGTLC:wmba}, it induces a coassociative
comultiplication 
\begin{equation}\label{eq:delta}
\delta:R\to R\ox R, \qquad 
\sqcap^R(a)\mapsto (\sqcap^R(a)\ox 1)F=F(1\ox \sqcap^R(a))
\end{equation}
with the counit 
\begin{equation}\label{eq:varepsilon}
\varepsilon:R\to k,\qquad \sqcap^R(a)\mapsto \epsilon(a).
\end{equation}
The comultiplication $\delta$ is an $R$-bimodule map and it provides a section
of the multiplication $R\ox R \to R$. That is to say, $R$ is a co-separable
co-Frobenius coalgebra, see \cite[Theorem 4.6]{BoGTLC:wmba}. The algebra $R$
is not unital but it has (idempotent) local units. There is a unique
automorphism, the so-called Nakayama automorphism $\vartheta$ of $R$, for which
$\varepsilon(rs)=\varepsilon (\vartheta(s)r)$, for any $s,r\in R$. 

As a particular consequence of the above properties, $R$ is a firm subalgebra
of $\M(A)$, meaning that the multiplication $R\ox R \to R$ projects to an
isomorphism $R\ox_R R \to R$ (where $\ox_R$ stands for the $R$-module tensor
product). There are evident notions of (associative but not unital) left and
right $R$-modules, and hence of $R$-bimodules. A (say, right) $R$-module $M$
is said to be firm if the action $M\ox R \to M$ projects to an isomorphism
$M\ox_R R\to M$. Since $R$ is a coseparable coalgebra, it follows by
\cite[Proposition 2.17]{BoVe} that the category of firm modules over the firm
algebra $R$ is isomorphic to the category of comodules over the coalgebra $R$.
A bimodule is said to be firm if it is firm as a left and as a right
$R$-module. Firm $R$-bimodules constitute a monoidal category ${}_R M_R$
(which is isomorphic to the category of $R$-bicomodules). The monoidal product
is the module tensor product $\ox_R$ over the firm algebra $R$ or, what is
isomorphic to it, the comodule tensor product over the coalgebra $R$. It is a
linear retract of the tensor product of vector spaces, see \cite[Proposition
2.17]{BoVe}. The monoidal unit is $R$ (with (co)actions provided by the
(co)multiplication). 

One defines symmetrically the {\em left full} property of the comultiplication
of a regular weak multiplier bialgebra. It implies analogous properties of the
coinciding image $L$ of the maps $\overline \sqcap^L$ and $\sqcap^L:A\to
\M(A)$, see \cite[Theorem 3.13]{BoGTLC:wmba}. 

If $A$ is a regular weak multiplier bialgebra with a right full
comultiplication, then by \cite[Lemma 4.8]{BoGTLC:wmba} there are 
anti-multiplicative maps 
\begin{eqnarray*}
&\tau:\sqcap^R(A)=\overline \sqcap^R(A)\to \overline \sqcap^L(A),\qquad
&\sqcap^R(a) \mapsto \overline \sqcap^L(a),\\
&\overline\tau:\sqcap^R(A)=\overline\sqcap^R(A)\to \sqcap^L(A),\qquad
&\overline\sqcap^R(a) \mapsto \sqcap^L(a).
\end{eqnarray*}
Symmetrically, if the comultiplication is left full then there are 
anti-multiplicative maps 
\begin{eqnarray*}
&\sigma:\sqcap^L(A)=\overline \sqcap^L(A)\to \sqcap^R(A),\qquad
&\overline \sqcap^L(a) \mapsto \sqcap^R(a),\\
&\overline\sigma:\sqcap^L(A)=\overline\sqcap^L(A)\to \overline\sqcap^R(A),\qquad
&\sqcap^L(a) \mapsto \overline\sqcap^R(a).
\end{eqnarray*}
If the comultiplication is both right and left full, then $\tau=\sigma^{-1}$
and $\overline\tau=\overline\sigma^{\,-1}$; by \cite[Proposition
4.9]{BoGTLC:wmba} they are anti-coalgebra maps; and the Nakayama automorphism
of $R$ is equal to $\sigma\overline\sigma\, ^{-1}$ and the Nakayama
automorphism of $L$ is equal to $\overline\sigma\, ^{-1}\sigma$.

\begin{lemma}\label{lem:Pi_spanned}
For any regular weak multiplier bialgebra $A$, 
the following hold.
\begin{itemize}
\item[{(1)}] The vector space $A$ is spanned by elements of the form
$a\overline \sqcap^L(b)$, for $a,b\in A$. 
\item[{(2)}] The vector space $A$ is spanned by elements of the form
$\overline \sqcap^R(b)a$, for $a,b\in A$.
\item[{(3)}] The vector space $A$ is spanned by elements of the form
$\sqcap^L(b)a$, for $a,b\in A$.
\item[{(4)}] The vector space $A$ is spanned by elements of the form
$a\sqcap^R(b)$, for $a,b\in A$.
\end{itemize}
If the comultiplication of $A$ is right full, then also the following
hold. 
\begin{itemize}
\item[{(5)}] The vector space $A$ is spanned by elements of the form
$a\overline \sqcap^R(b)$, for $a,b\in A$.
\item[{(6)}] The vector space $A$ is spanned by elements of the form
$\sqcap^R(b)a$, for $a,b\in A$.
\end{itemize}
If the comultiplication of $A$ is left full, then also the following
hold. 
\begin{itemize}
\item[{(7)}] The vector space $A$ is spanned by elements of the form
$\overline \sqcap^L(b)a$, for $a,b\in A$. 
\item[{(8)}] The vector space $A$ is spanned by elements of the form
$a\sqcap^L(b)$, for $a,b\in A$.
\end{itemize}
\end{lemma}

\begin{proof}
Assertions (1)-(4) follow by the idempotency of $A$ and 
\cite[Lemma 3.7]{BoGTLC:wmba}. 
Assertions (5)-(8) follow from assertions (1) to (4) noting that by
\cite[Theorem 3.13]{BoGTLC:wmba} the ranges of $\sqcap^R$ and $\overline
\sqcap^R$ coincide whenever the comultiplication is right full; and the ranges
of $\sqcap^L$ and $\overline \sqcap^L$ coincide whenever the comultiplication
is left full. 
\end{proof}

\subsection{The antipode}\label{sec:antipode}

For a regular weak multiplier bialgebra $A$, consider the idempotent map 
\begin{equation}\label{eq:G_1}
G_1:A\ox A \to A\ox A, \qquad
a\ox bc\mapsto (a\ox 1)F(1\ox bc)=
(a\ox 1)((\overline \sqcap^R \ox A)\mathsf{tw}T_4(c\ox b))
\end{equation}
in (6.5) of \cite{BoGTLC:wmba} and its symmetric counterpart 
\begin{equation}\label{eq:G_2}
G_2:A\ox A \to A\ox A, \qquad
ab\ox c\mapsto 
((A\ox \overline \sqcap^L)\mathsf{tw}T_3(b\ox a))(1\ox c)
\end{equation}
in (6.11) of \cite{BoGTLC:wmba}.
By \cite[Proposition 6.3]{BoGTLC:wmba} and its symmetric counterpart, they
obey $E_iT_i=T_i=T_iG_i$, for $i\in \{1,2\}$. By \cite[Theorem
6.8]{BoGTLC:wmba}, the maps $T_i$ are `weakly invertible' --- in the sense
that there are linear maps $R_1,R_2:A\ox A\to A\ox A$ obeying $R_iT_i=G_i$,
$T_iR_i=E_i$ and $R_iT_iR_i=R_i$ for $i\in \{1,2\}$ --- if and only if there
is a linear map $S:A \to \M(A)$ --- the so-called {\em antipode} --- obeying
the equalities in part (2) of \cite[Theorem 6.8]{BoGTLC:wmba}.
The resulting structure is in generality between arbitrary, and regular 
weak multiplier Hopf algebra in \cite{VDaWa}.
The conditions on $S$ in part (2) of \cite[Theorem 6.8]{BoGTLC:wmba} imply 
\begin{equation}\label{eq:S_id}
\begin{array}{ll}
\mu(S\ox A)T_1=\mu(\sqcap^R\ox A),\qquad
&\mu(A\ox S)T_2=\mu(A\ox \sqcap^L),\\
\mu(S\ox A)E_1=\mu(S\ox A),\qquad
&\mu(A\ox S)E_2=\mu(A\ox S),
\end{array}
\end{equation}
see (6.14) in \cite{BoGTLC:wmba}, but do not seem be equivalent to them. The
antipode is anti-multiplicative by \cite[Theorem 6.12]{BoGTLC:wmba}, see also
\cite[Proposition 3.5]{VDaWa}, and anti-comultiplicative (in the multiplier
sense) by \cite[Corollary 6.16]{BoGTLC:wmba}.

\begin{lemma}\label{lem:S_nd}
Let $A$ be a regular weak multiplier bialgebra possessing an antipode $S$. If
the comultiplication is right full, then for any element $a\in A$ the
following are equivalent.
\begin{itemize}
\item[{(a)}] $a=0$.
\item[{(b)}] $aS(b)=0$ for all $b\in A$.
\end{itemize}
If the comultiplication is left full, then for any element $a\in A$ the
following are equivalent. 
\begin{itemize}
\item[{(a')}] $a=0$.
\item[{(b')}] $S(b)a=0$ for all $b\in A$.
\end{itemize}
\end{lemma}

\begin{proof}
The implication (a)$\Rightarrow$(b) is trivial. If (b) holds then for all
$b,c\in A$, using the implicit summation index notation $T_1(b\ox c)=b^1\ox
c^1$, 
\begin{equation}\label{eq:S_nd}
0=aS(b^1)c^1=a\sqcap^R(b)c,
\end{equation}
where the last equality follows by (6.14) in \cite{BoGTLC:wmba}. By
Lemma \ref{lem:Pi_spanned}~(6), any element of $A$ can be written as a
linear combination of elements of the form $\sqcap^R(b)c$. Thus we conclude by
the non-degeneracy of $A$ from \eqref{eq:S_nd} that (a) holds. The equivalence
(a')$\Leftrightarrow$(b') follows symmetrically. 
\end{proof}

\begin{lemma}\label{lem:T_1_S}
Consider a regular weak multiplier bialgebra $A$ possessing an antipode
$S$. For any $b,d\in A$, introduce the implicit summation index notation
$T_2(d\ox b)=d^2\ox b^2$. For any $a,b,c,d\in A$,
$$
T_1(aS(b) \ox S(d)c)=T_1(a\ox S(d^2)c)(S(b^2)\ox 1).
$$
\end{lemma}

\begin{proof}
Using the multiplicativity of $\Delta$ and hence of its extension $\overline
\Delta:\M(A)\to \M(A\ox A)$ in the second equality; the
anti-comultiplicativity of $S$ (cf. \cite[Corollary 6.16]{BoGTLC:wmba}) in the
third equality; and the anti-multiplicativity of $S$ (cf. \cite[Theorem
6.12]{BoGTLC:wmba}) in the fourth equality,
\begin{eqnarray*}
T_1(aS(b) \ox S(d)c)&=&
\Delta(aS(b))(1 \ox S(d)c)=
\Delta(a)\overline \Delta S(b)(1 \ox S(d)c)\\
&=&\Delta(a)(\overline{S\ox S})(\Delta(b)^{21})(1\ox S(d)c)=
\Delta(a)(S(b^2)\ox S(d^2))(1\ox c)\\
&=&T_1(a\ox S(d^2)c)(S(b^2)\ox 1).
\end{eqnarray*}
\end{proof}

\begin{lemma}\label{lem:T_tw_R}
For a regular weak multiplier bialgebra $A$ with a left and right full
comultiplication and possessing an antipode $S$, 
and for any elements $a,b\in A$, the following identities hold.
\begin{itemize}
\item[{(1)}] $T_3 \mathsf{tw} R_1(a\ox b)=(S(a)\ox 1)\Delta(b)$.
\item[{(2)}] $T_4 \mathsf{tw} R_2(a\ox b)=\Delta(a)(1\ox S(b))$.
\end{itemize}
\end{lemma}

\begin{proof}
We only prove part (1), part (2) follows symmetrically. For any $a,c\in A$,
introduce the notation $T_2(c\ox a)=c^2\ox a^2$ where implicit summation is
understood. For any $b,d\in A$,
\begin{eqnarray*}
(S(d)\ox c)(T_3 \mathsf{tw} R_1(a\ox b))&=&
(S(d)\ox 1)(T_3 \mathsf{tw}((c\ox 1)R_1(a\ox b)))\\&=&
(S(d)\ox 1)(T_3 \mathsf{tw}(((A\ox S)T_2(c\ox a))(1\ox b)))\\&=&
(S(d)\ox c^2)\Delta(S(a^2)b)\\&=&
(S(d)\ox c^2)((\overline{S\ox S})\Delta(a^2)^{21})\Delta(b)\\&=&
(\mathsf{tw}(\mu\ox A)(A\ox S\ox S)(A\ox T_1)(T_2\ox A)(c\ox a\ox d))
\Delta(b)\\&=&
(\mathsf{tw}(\mu\ox A)(A\ox S\ox S)(T_2\ox A)(A\ox T_1)(c\ox a\ox d))
\Delta(b)\\&=&
(\mathsf{tw}(\mu\ox A)(A\ox \sqcap^L\ox S)(A\ox T_1)(c\ox a\ox d))
\Delta(b)\\&=&
(S\mu (\sqcap^L\ox A)T_1(a\ox d)\ox c)\Delta(b)\\&=&
(S(ad)\ox c)\Delta(b)\\&=&
(S(d)\ox c)(S(a)\ox 1)\Delta(b).
\end{eqnarray*}
Simplifying by $S(d)\ox c$ (cf. Lemma \ref{lem:S_nd}), we conclude the claim. 
In the first equality we used the left $A$-module property $T_3(b\ox ca)=(1\ox
c)T_3(b\ox a)$. The second equality follows applying identity (6.3) in
\cite{BoGTLC:wmba} to $R_1$. The fourth equality follows by the
multiplicativity of $\Delta$ and the anti-comultiplicativity of $S$ (in the
multiplier sense, cf. \cite[Corollary 6.16]{BoGTLC:wmba}) and the fifth and
the last equalities follow by the anti-multiplicativity of $S$ (see
\cite[Theorem 6.12]{BoGTLC:wmba}). In the sixth equality we made use of the
coassociativity axiom (iii) of weak multiplier bialgebra in Definition
\ref{def:wmba}. The seventh equality follows by an identity in (6.14) in
\cite{BoGTLC:wmba} and the eighth one follows by \cite[Lemma 3.9 and Lemma
6.14]{BoGTLC:wmba}. The penultimate equality is a consequence of 
\cite[Lemma 3.7~(3)]{BoGTLC:wmba}. 
\end{proof}

\section{Comodules and their morphisms}\label{sec:comod}

Before we can formulate the definition of {\em comodule} over a regular weak
multiplier bialgebra, some preparation is needed. 
Let $k$ be a field, $V$ a $k$-vector space and $A$ a (not necessarily unital)
algebra over $k$. Then we can regard the tensor product vector space $V\ox A$
as an $A\cong k\ox A$-bimodule via the left and right actions
$$
(1\ox a)(v\ox b)=v\ox ab\qquad \textrm{and}\qquad 
(v\ox b)(1\ox a)=v \ox ab,
$$
where $1$ stands for the unit element of the base field.

\begin{proposition}\label{prop:prep_comod}
Let $A$ be a regular weak multiplier bialgebra, let $V$ be a vector space and
let $\lambda$ and $\varrho$ be linear maps $V\ox A \to V \ox A$ such that 
\begin{equation}\label{eq:prep}
(1\ox a)\lambda (v\ox b)=\varrho(v\ox a)(1\ox b)\qquad 
\textrm{for all } v\in V \textrm{ and } a,b\in A.
\end{equation}
Then the following statements hold.
\begin{itemize}
\item[{(1)}] The map $\lambda$ is a morphism of right $A$-modules and
 $\varrho$ is a morphism of left $A$-modules.
\item[{(2)}] The following assertions are equivalent.
\begin{itemize}
\item[{(2.a)}] $(V\ox E_1)(\lambda\ox A)\lambda^{13}=(\lambda\ox A)\lambda^{13}$.
\item[{(2.b)}] $(\varrho\ox A)\varrho^{13}(V\ox E_2)=
(\varrho \ox A)\varrho^{13}$. 
\end{itemize}
\item[{(3)}] The following assertions are equivalent.
\begin{itemize}
\item[{(3.a)}] $(\lambda\ox A)\lambda^{13}(V\ox E_1)=(\lambda\ox A)\lambda^{13}$.
\item[{(3.b)}] $(V\ox E_2)(\varrho\ox A)\varrho^{13}=
(\varrho \ox A)\varrho^{13}$. 
\end{itemize}
\item[{(4)}] The following assertions are equivalent.
\begin{itemize}
\item[{(4.a)}] The equalities in part (3) hold and 
\begin{equation}\label{eq:coass-l-prep}
(\lambda\ox A)\lambda^{13}(V\ox T_1)=(V\ox T_1)(\lambda\ox A).
\end{equation}
\item[{(4.b)}] The equalities in part (3) hold and 
\begin{equation}\label{eq:coass-l-prep-b}
(\lambda\ox A)\lambda^{13}(V\ox T_4)=(V\ox T_4)\lambda^{13}.
\end{equation}
\item[{(4.c)}] The equalities in part (3) hold and 
\begin{equation}\label{eq:coass-l-prep-c}
(\lambda\ox A)(V\ox T_3)\varrho^{13}=(V\ox T_3)(\lambda\ox A).
\end{equation}
\item[{(4.d)}] The equalities in part (2) hold and 
\begin{equation}\label{eq:coass-r-prep}
(\varrho\ox A)\varrho^{13}(V\ox T_3)=(V\ox T_3)(\varrho \ox A).
\end{equation} 
\item[{(4.e)}] The equalities in part (2) hold and 
\begin{equation}\label{eq:coass-r-prep-e}
(\varrho\ox A)\varrho^{13}(V\ox T_2)=(V\ox T_2)\varrho^{13}.
\end{equation} 
\item[{(4.f)}] The equalities in part (2) hold and 
\begin{equation}\label{eq:coass-r-prep-f}
(\varrho\ox A)(V\ox T_1)\lambda^{13}=(V\ox T_1)(\varrho \ox A).
\end{equation} 
\end{itemize}
\end{itemize}
\end{proposition}

\begin{proof}
(1) Applying twice \eqref{eq:prep},
$$
(1\ox a)\lambda(v\ox bc)=
\varrho(v\ox a)(1\ox bc)=
(1\ox a)\lambda(v\ox b)(1\ox c),
$$
for any $v\in V$ and $a,b,c\in A$. Hence by non-degeneracy of the left
$A$-module $V\ox A$ it follows that $\lambda$ is a right $A$-module map. It is
proven symmetrically that $\varrho$ is a left $A$-module map.

(2) Again, applying twice \eqref{eq:prep},
\begin{eqnarray}\label{eq:l-r-2}
((\varrho \ox A)\varrho^{13}(v\ox a\ox b))\hspace{-1.cm}&&(1\ox c\ox d)\\
&=&(1\ox a\ox 1)(\lambda\ox A)[\varrho^{13}(v\ox c\ox b)(1\ox 1\ox d)]
\nonumber\\
&=&(1\ox a\ox b)((\lambda\ox A)\lambda^{13}(v\ox c\ox d)),
\nonumber
\end{eqnarray}
for any $v\in V$ and $a,b,c,d\in A$. With this identity at hand,
\begin{eqnarray*}
((\varrho\ox A)\varrho^{13}(V\ox E_2)(v\ox a\ox b)) 
\hspace{-1.3cm}&&(1\ox c\ox d)\\
&=&((\varrho\ox A)\varrho^{13}(v\ox (a\ox b)E))(1\ox c\ox d)\\
&\stackrel{\eqref{eq:l-r-2}}=&
(1\ox (a\ox b)E)((\lambda\ox A)\lambda^{13}(v\ox c\ox d))\\
&=&(1\ox a\ox b)((V\ox E_1)(\lambda\ox A)\lambda^{13}(v\ox c\ox d)).
\end{eqnarray*}
Hence the implication (2.a)$\Rightarrow$(2.b) follows by non-degeneracy of the
right $A\ox A$-module $V\ox A \ox A$ and (2.b)$\Rightarrow$(2.a) follows by
non-degeneracy of the left $A\ox A$-module $V\ox A \ox A$.

(3) is proven analogously to (2).

(4.a)$\Leftrightarrow$(4.b) The equality in \eqref{eq:coass-l-prep} is
equivalent to
$$
((\lambda\ox A)\lambda^{13}(V\ox T_1)(v\ox a\ox b))(1\ox c\ox 1)=
((V\ox T_1)(\lambda\ox A)(v\ox a\ox b))(1\ox c\ox 1),
$$
for all $v\in V$ and $a,b,c\in A$. Using \eqref{eq:T_14} and the right
$A$-module map property of $\lambda$ from part (1), the displayed equality is
equivalent to
$$
((\lambda\ox A)\lambda^{13}(V\ox T_4)(v\ox c\ox a))(1\ox 1\ox b)=
((V\ox T_4)\lambda^{13}(v\ox c\ox a))(1\ox 1\ox b).
$$
Its validity for all $v\in V$ and $a,b,c\in A$ is equivalent to
\eqref{eq:coass-l-prep-b}. 
The equivalence (4.d)$\Leftrightarrow$(4.e) is proven symmetrically.

(4.a)$\Leftrightarrow$(4.c) The equality in \eqref{eq:coass-l-prep} is
equivalent also to
$$
(1\ox 1\ox c)((\lambda\ox A)\lambda^{13}(V\ox T_1)(v\ox a\ox b))=
(1\ox 1\ox c)((V\ox T_1)(\lambda\ox A)(v\ox a\ox b)),
$$
for all $v\in V$ and $a,b,c\in A$. Combining the first equality in axiom (ix)
in Definition \ref{def:wmba} with \eqref{eq:prep}, the displayed equality is
equivalent to  
$$
((\lambda\ox A)(V\ox T_3)\varrho^{13}(v\ox a\ox c))(1\ox 1\ox b)=
((V\ox T_3)(\lambda\ox A)(v\ox a\ox c))(1\ox 1\ox b).
$$
Its validity for all $v\in V$ and $a,b,c\in A$ is equivalent to
\eqref{eq:coass-l-prep-c}. 
The equivalence (4.d)$\Leftrightarrow$(4.f) is proven symmetrically.
 
(4.a)$\Rightarrow$(4.d)
First we show that (4.a) implies (2.a). By axiom (vi) of weak multiplier
bialgebra in Definition \ref{def:wmba}, for any elements $a,b\in A$ 
there exist finitely many elements $p^i,q^i,r^i\in A$ such that $E(a\ox
b)=\sum_i\Delta(p^i)(q^i\ox r^i)$. In terms of these elements and for any
$v\in V$, omitting for brevity the summation symbols,
\begin{eqnarray}\label{eq:4a>b}
(\lambda\ox A)\lambda^{13}(v\ox a\ox b)
&\stackrel{\textrm{(3.a)}}=&
(\lambda\ox A)\lambda^{13}(v\ox E(a\ox b))\\
&=&(\lambda\ox A)\lambda^{13}(v\ox T_1(p^i\ox r^i)(q^i\ox 1))\nonumber\\
&\stackrel{\textrm{part}\ (1)}=&
((\lambda\ox A)\lambda^{13}(V\ox T_1)(v\ox p^i\ox r^i))(1\ox q^i\ox 1)
\nonumber\\
&\stackrel{\eqref{eq:coass-l-prep}}=&
((V\ox T_1)(\lambda\ox A)(v\ox p^i\ox r^i))(1\ox q^i\ox 1).
\nonumber
\end{eqnarray}
Using this identity in the first and the last equalities, the right $A$-module
map property of $E_1$ in the second equality and $E_1T_1=T_1$ in the third
one, 
\begin{eqnarray*}
(V\ox E_1)(\lambda\ox A)\lambda^{13}(v\ox a\ox b)&=&
(V\ox E_1)[((V\ox T_1)(\lambda\ox A)(v\ox p^i\ox r^i))(1\ox q^i\ox 1)]\\
&=&((V\ox E_1)(V\ox T_1)(\lambda\ox A)(v\ox p^i\ox r^i))(1\ox q^i\ox 1)\\
&=&((V\ox T_1)(\lambda\ox A)(v\ox p^i\ox r^i))(1\ox q^i\ox 1)\\
&=&(\lambda\ox A)\lambda^{13}(v\ox a\ox b).
\end{eqnarray*}
Finally --- using the same notation as above and the implicit summation index
notation $\varrho(v\ox a)=v^\varrho\ox a^\varrho$, $\lambda(v\ox
p)=v^\lambda\ox p^\lambda$ --- we show that (4.a) implies
\eqref{eq:coass-r-prep}. Take any $v\in V$ and $a,b,c,d\in A$. Then
\begin{eqnarray*}
((V\ox T_3)(\varrho\ox A)(v\ox c\ox d))\hspace{-1cm}&&(1\ox a\ox b)\\
&=&((V\ox E_2)(V\ox T_3)(\varrho\ox A)(v\ox c\ox d))(1\ox a\ox b)\\
&=&((V\ox T_3)(\varrho\ox A)(v\ox c\ox d))(1\ox E(a\ox b))\\
&=&v^\varrho\ox (1\ox d)\Delta(c^\varrho p^i)(q^i\ox r^i)\\
&=&v^\lambda\ox (1\ox d)\Delta(cp^{i\lambda} )(q^i\ox r^i)\\
&=&(1\ox T_3(c\ox d))((V\ox T_1)(\lambda\ox A)(v\ox p^i\ox r^i))(1\ox q^i\ox 1).
\end{eqnarray*}
In the first equality we used $E_2T_3=T_3$ and in the second equality we used
axiom (i) of weak multiplier bialgebra in Definition \ref{def:wmba}.
In the third and the last equalities we used the multiplicativity of $\Delta$
and in the penultimate one we applied \eqref{eq:prep}. On the other hand,
\begin{eqnarray*}
((\varrho\ox A)\varrho^{13}(V\ox T_3)\hspace{-1cm}&&(v\ox c\ox d))
(1\ox a\ox b)\\
&\stackrel{\textrm{(3.b)}}=&
((V\ox E_2)(\varrho\ox A)\varrho^{13}(V\ox T_3)(v\ox c\ox d))(1\ox a\ox
b)\\
&\stackrel{\textrm{(i)}}=&
((\varrho\ox A)\varrho^{13}(V\ox T_3)(v\ox c\ox d))(1\ox E(a\ox b))\\
&=&((\varrho\ox A)\varrho^{13}(V\ox T_3)(v\ox c\ox d))(1\ox T_1(p^i\ox r^i)
(q^i\ox 1))\\
&\stackrel{\eqref{eq:l-r-2}}=&
(1\ox T_3(c\ox d))((\lambda\ox A)\lambda^{13}(V\ox T_1)(v\ox p^i\ox r^i))
(1\ox q^i\ox 1).
\end{eqnarray*}
The expressions we obtained are equal by \eqref{eq:coass-l-prep} so that
\eqref{eq:coass-r-prep} holds by the non-degeneracy of the right $A\ox
A$-module $V\ox A\ox A$. 
The converse implication (4.d)$\Rightarrow$(4.a) follows symmetrically. 
\end{proof}

We define a right comodule over a regular weak multiplier bialgebra $A$ as a
triple $(V,\lambda,\varrho)$ satisfying \eqref{eq:prep} and the equivalent
conditions in Proposition \ref{prop:prep_comod}~(4). That is, we propose the
following.

\begin{definition}\label{def:comod}
A {\em right comodule} over a regular weak multiplier bialgebra $A$ is a
triple consisting of a vector space $V$ and linear maps $\lambda,\varrho:V\ox
A \to V\ox A$ satisfying 
\begin{equation}\label{eq:comp}
(1\ox a)\lambda(v\ox b)=
\varrho(v\ox a)(1\ox b),
\qquad \forall v\in V,\ a,b\in A
\end{equation}
together with
\begin{eqnarray}
\label{eq:l-i-norm}
&&(\lambda\ox A)\lambda^{13}(V \ox E_1)=(\lambda\ox A)\lambda^{13}\\
\label{eq:l-coass}
&&(\lambda\ox A)\lambda^{13}(V \ox T_1)=(V \ox T_1)(\lambda\ox A).
\end{eqnarray}
Equivalently, $\lambda$ and $\varrho$ satisfy \eqref{eq:comp} and 
\begin{eqnarray}
\label{eq:r-i-norm}
&&(\varrho \ox A)\varrho^{13}(V \ox E_2)=(\varrho \ox A)\varrho^{13}\\
\label{eq:r-coass}
&&(\varrho\ox A)\varrho^{13}(V \ox T_3)=(V \ox T_3)(\varrho\ox A).
\end{eqnarray}
\end{definition}

\begin{remark}\label{rem:unital_com}
If $A$ is a usual weak bialgebra as in \cite{WHAI,Nill} possessing an
algebraic unit $1$, and $V$ is a vector space, then a pair of linear maps
$\lambda,\varrho:V\ox A \to V\ox A$ satisfying \eqref{eq:comp} is equivalent
to a linear map
$$
\tau:V\to V\ox A,\qquad 
v\mapsto \lambda(v\ox 1)\stackrel{\eqref{eq:comp}}=\varrho(v\ox 1). 
$$
For this, the implicit summation index notation $\tau(v)=v^0\ox v^1$ is widely
used in the literature. In this situation the normalization condition
\eqref{eq:l-i-norm} translates to
$$
((\tau\ox A)\tau(v))(1\ox \Delta(1))=(\tau\ox A)\tau(v),\qquad \forall v\in V
$$
and the coassociativity condition \eqref{eq:l-coass} translates to
$$
((\tau\ox A)\tau(v))(1\ox \Delta(1))=(V\ox \Delta)\tau(v),\qquad \forall v\in
V. 
$$
These conditions together are equivalent to the usual coassociativity
condition $(\tau\ox A)\tau=(V\ox \Delta)\tau$.

Comodules over coalgebras (so in particular over (weak) bialgebras) are
usually required to be counital as well. Later we will require an additional
condition replacing it; see Definition \ref{def:full} and Remark
\ref{rem:counital}. 
\end{remark}

\begin{example}\label{ex:A_com}
For any regular weak multiplier bialgebra $A$, there is a right $A$-comodule
$(A,T_1,T_3)$. 
\end{example}

\begin{proof}
Condition \eqref{eq:comp} holds by the first equality in axiom (ix) in
Definition \ref{def:wmba}. For any $b,c\in A$, introduce the notation
$T_1(b\ox c)=:b^1\ox c^1$, where implicit summation is understood. Then for
any $a,b,c,d\in A$,
\begin{eqnarray*}
(d\ox 1\ox 1)((T_1\ox A)T_1^{13}(A\ox T_1)\hspace{-1cm}&&(a\ox b\ox c))\\
&=&((T_2\ox A)(A\ox T_1)(d\ox a\ox c^1))(1\ox b^1\ox 1)\\
&=&((A\ox T_1)(T_2\ox A)(d\ox a\ox c^1))(1\ox b^1\ox 1)\\
&=&(A\ox T_1)(T_2(d\ox a)(1\ox b)\ox c)\\
&=&(d\ox 1\ox 1)((A\ox T_1)(T_1\ox A)(a\ox b\ox c)).
\end{eqnarray*}
In the first and the last equalities we used \eqref{eq:iii}, in the second
equality we used the coassociativity axiom (iii), and in the penultimate
equality we used axiom (v) of weak multiplier bialgebra in Definition
\ref{def:wmba}.  
By non-degeneracy of the left $A$-module $A\ox A\ox A$, this proves the
coassociativity condition \eqref{eq:l-coass}. It remains to check the
normalization condition \eqref{eq:l-i-norm}. To this end, note that for any
$a,b,d\in A$, 
\begin{equation}\label{eq:T_1_balanced}
T_1(a\ox \overline\sqcap^R(b)d)=
\Delta(a)(1\ox \overline\sqcap^R(b)d)=
\Delta(a\overline\sqcap^R(b))(1\ox d)=
T_1(a \overline\sqcap^R(b)\ox d)
\end{equation}
and
\begin{equation}\label{eq:T_1_L_mod}
T_1(a\ox d)(\overline \sqcap^R(b)\ox 1)=
\Delta(a)(\overline \sqcap^R(b)\ox d)=
\Delta(a)(1 \ox \sqcap^L(b) d)=
T_1(a\ox \sqcap^L(b) d),
\end{equation}
where the second equalities follow by Lemma 3.3 and Lemma 3.9 in
\cite{BoGTLC:wmba}, respectively. Combining these
equalities, for any $a,b,c,d\in A$
\begin{equation}\label{eq:A_com}
(T_1\ox A)T_1^{13}(a\ox \overline \sqcap^R(b)d\ox c)=
(T_1\ox A)T_1^{13}(a\ox d\ox \sqcap^L(b)c).
\end{equation}
With this identity at hand, and using again the notation $T_1(b\ox c)=:b^1\ox
c^1$ (with implicit summation understood),
\begin{eqnarray*}
(T_1\ox A)T_1^{13}(A\ox E_1)(a\ox d\ox bc)&=&
(T_1\ox A)T_1^{13}(a\ox \overline \sqcap^R(b^1)d\ox c^1)\\
&\stackrel{\eqref{eq:A_com}}=&
(T_1\ox A)T_1^{13}(a\ox d\ox \sqcap^L(b^1)c^1)\\
&=&(T_1\ox A)T_1^{13}(a\ox d\ox bc).
\end{eqnarray*}
The first equality follows by identity (2.3) in \cite{BoGTLC:wmba}, and the
last equality follows by Lemma 3.7~(3) in \cite{BoGTLC:wmba}. Using that the
algebra $A$ is idempotent, this proves the normalization condition
\eqref{eq:l-i-norm}.
\end{proof}

\begin{example}\label{ex:R_com}
For any regular weak multiplier bialgebra $A$ with a right full
comultiplication, denote by $R$ the coinciding range of the maps
$\sqcap^R:A\to \M(A)$ and $\overline \sqcap^R:A\to \M(A)$ (cf. \cite[Theorem
3.13]{BoGTLC:wmba}). Then there is a right $A$-comodule $(R,r\ox a \mapsto
E(1\ox ra),r\ox a \mapsto (1\ox ar)E)$. 
\end{example}

\begin{proof}
Both maps $\lambda:r\ox a \mapsto E(1\ox ra)$ and $\varrho:r\ox a \mapsto
(1\ox ar)E$ have their range in $R\ox A$ by identities (2.3) and (3.4) in 
\cite{BoGTLC:wmba}, respectively. By \cite[Lemma 3.3]{BoGTLC:wmba}, we can
write equivalently
$$
\lambda(r\ox a) =(1\ox r)E(1\ox a)
\qquad \textrm{and}\qquad 
\varrho(r\ox a)= (1\ox a)E(1\ox r)
$$
so the compatibility condition \eqref{eq:comp} evidently holds. For any $r\in
R$ and $a,b\in A$, 
\begin{equation}\label{eq:R_com}
(\lambda\ox A)\lambda^{13}(r\ox a\ox b)=
(\lambda\ox A)(E^{13}(1\ox a\ox rb))=
(E\ox 1)(1\ox E)(1\ox a\ox rb).
\end{equation}
Thus
\begin{eqnarray*}
(R\ox T_1)(\lambda \ox A)(r\ox a\ox b)&=&
(R\ox T_1)((E\ox 1)(1\ox ra\ox b))\\
&=&(E\ox 1)(1\ox E)(1\ox T_1(ra\ox b))\\
&=&(E\ox 1)(1\ox E)(1\ox1\ox r)(1\ox T_1(a\ox b))\\
&\stackrel{\eqref{eq:R_com}}=&
(\lambda\ox A)\lambda^{13}(R\ox T_1)(r\ox a\ox b).
\end{eqnarray*}
The second equality follows by axiom (vii) of weak multiplier bialgebra in
Definition \ref{def:wmba} and $E_1T_1=T_1$. The third equality follows by
\cite[Lemma 3.3]{BoGTLC:wmba}. This proves the coassociativity condition
\eqref{eq:l-coass}. Finally,  
\begin{eqnarray*}
(\lambda\ox A)\lambda^{13}(R\ox E_1)(r\ox a\ox b)
&\stackrel{\eqref{eq:R_com}}=&
(E\ox 1)(1\ox E)(1\ox1\ox r)(1\ox E(a\ox b))\\
&=&(E\ox 1)(1\ox E)(1\ox a\ox rb)\\
&\stackrel{\eqref{eq:R_com}}=&
(\lambda\ox A)\lambda^{13}(r\ox a\ox b).
\end{eqnarray*}
The second equality follows by \cite[Lemma 3.3]{BoGTLC:wmba} and $E^2=E$.
This proves the normalization condition \eqref{eq:l-i-norm}.
\end{proof}

In the following proposition the compatibility of comodules with the counit is
studied. 

\begin{proposition}\label{prop:old_cu}
Let $A$ be a regular weak multiplier bialgebra and let $(V,\lambda,\varrho)$
be a right $A$-comodule. Then the following equalities hold.
\begin{itemize}
\item[{(1)}] $(V\ox \epsilon\ox A)(\lambda\ox A)\lambda^{13}=
\lambda(V\ox \mu(\sqcap^L\ox A))$.
\item[{(2)}] $(V\ox \epsilon\ox A)\lambda^{13}(\lambda\ox A)=
\lambda(V\ox \mu(\overline \sqcap^R\ox A))$.
\item[{(3)}] $(V\ox A \ox \epsilon)(\varrho\ox A)\varrho^{13}=
\varrho(V\ox \mu(A\ox \sqcap^R))$.
\item[{(4)}] $(V\ox A\ox \epsilon)\varrho^{13}(\varrho\ox A)=
\varrho(V\ox \mu(A \ox \overline \sqcap^L))$.
\end{itemize}
\end{proposition}

\begin{proof}
By axiom (vi) of weak multiplier bialgebra in Definition \ref{def:wmba}, for
any $a,b\in A$ there exist finitely many elements $p^i,q^i,r^i\in A$ such that
$E(a\ox b)=\sum_i\Delta(p^i)(q^i\ox r^i)=\sum_iT_1(p^i\ox r^i)(q^i\ox 1)$. In
terms of these elements, omitting the summation symbols for brevity, for any
$v\in V$ the equality \eqref{eq:4a>b} holds. Equivalently, by
\eqref{eq:T_14}, 
\begin{equation}\label{eq:4a>b_v}
(\lambda\ox A)\lambda^{13}(v\ox a\ox b)=
((V\ox T_4)\lambda^{13}(v\ox q^i\ox p^i))(1\ox 1\ox r^i).
\end{equation}
The identities \eqref{eq:4a>b_v} and \eqref{eq:4a>b} are used to prove parts
(1) and (2), respectively. 

(1) Using the above elements $p^i,q^i,r^i\in A$ associated to $a,b\in A$, for
any $v\in V$ 
\begin{eqnarray*}
(V\ox \epsilon\ox A)(\lambda\ox A)\lambda^{13}(v\ox a\ox b)
&\stackrel{\eqref{eq:4a>b_v}}=&
((V\ox \epsilon\ox A)(V\ox T_4)\lambda^{13}(v\ox q^i\ox p^i))(1\ox r^i)\\
&\stackrel{\eqref{eq:Pi}}=&
\lambda(v\ox p^i)(1\ox \sqcap^L(q^i)r^i)\\
&=&\lambda(v\ox p^i\sqcap^L(q^i)r^i)=
\lambda(v\ox \sqcap^L(a)b).
\end{eqnarray*}
The penultimate equality follows by Proposition \ref{prop:prep_comod}~(1) and
the last one follows by  
$$
p^i\sqcap^L(q^i)r^i\stackrel{\eqref{eq:Pi}}=
(\epsilon \ox A)(\Delta(p^i)(q^i\ox r^i))=
(\epsilon \ox A)(E(a\ox b))\stackrel{\eqref{eq:Pi}}=
\sqcap^L(a)b.
$$

(2) Using the same notation as above,
\begin{eqnarray*}
(V\ox A\ox \epsilon)(\lambda\ox A)\lambda^{13}(v\ox a\ox b)
&\stackrel{\eqref{eq:4a>b}}=&
((V\ox A \ox \epsilon)(V\ox T_1)(\lambda(v\ox p^i)\ox r^i))
(1\ox q^i)\\
&\stackrel{\eqref{eq:Pibar}}=& 
\lambda(v\ox p^i)(1\ox \overline\sqcap^R(r^i)q^i)\\
&=&\lambda(v\ox p^i\overline \sqcap^R(r^i)q^i)=
\lambda(v\ox\overline \sqcap^R(b)a).
\end{eqnarray*}
The penultimate equality follows by Proposition \ref{prop:prep_comod}~(1) and
the last equality follows by 
$$
p^i\overline \sqcap^R(r^i)q^i\stackrel{\eqref{eq:Pibar}}=
(A\ox \epsilon)(\Delta(p^i)(q^i\ox r^i))=
(A\ox \epsilon)(E(a\ox b))\stackrel{\eqref{eq:Pibar}}=
\overline\sqcap^R(b)a.
$$

The remaining assertions follow symmetrically.
\end{proof}

\begin{lemma}\label{lem:bim}
For any regular weak multiplier bialgebra $A$, any right $A$-comodule
$(V,\lambda,\varrho)$ and any $v\in V$ and $a,b\in A$, the following
statements hold.
\begin{itemize}
\item[{(1)}] $(V\ox A\ox \epsilon)\lambda^{13}(\varrho\ox A)(v\ox a \ox b)=
\varrho(v\ox a)(1\ox \sqcap^L(b))$.
\item[{(2)}] $(V\ox A\ox \epsilon)\varrho^{13}(\lambda\ox A)(v\ox a \ox b)=
(1\ox \overline\sqcap^L(b))\lambda(v\ox a)$.
\end{itemize}
\end{lemma}

\begin{proof}
We only prove part (1), part (2) is proven symmetrically. For any $v\in V$ and
$a,b,c\in A$, 
\begin{eqnarray*}
((V\ox A\ox \epsilon)\lambda^{13}(\varrho\ox A)\hspace{-1.5cm}&&
(v\ox a \ox b))(1\ox c)\\
&\stackrel{\eqref{eq:comp}}=&
(1\ox a)((V\ox A\ox \epsilon)\lambda^{13}(\lambda\ox A)(v\ox c \ox b))\\
&=&(1\ox a)\lambda(v\ox \sqcap^L(b)c)
\stackrel{\eqref{eq:comp}}=
\varrho(v\ox a)(1\ox \sqcap^L(b)c),
\end{eqnarray*}
from which we conclude by the non-degeneracy of the right $A$-module $V\ox
A$. The second equality follows by Proposition \ref{prop:old_cu}~(1).
\end{proof}

In order to define morphisms between comodules, we need the following
lemma. Since for any vector space $V$, the left and right modules $V\ox
A$ over a weak multiplier bialgebra $A$ are non-degenerate, it follows
immediately from \eqref{eq:comp}. 

\begin{lemma}\label{lem:comod_mor}
Let $A$ be a regular weak multiplier bialgebra and let $(V,\lambda,\varrho)$
and $(V',\lambda',\varrho')$ be right $A$-comodules. For any linear map
$f:V\to V'$, the following assertions are equivalent.
\begin{itemize}
\item[{(a)}] $\lambda'(f\ox A)=(f\ox A)\lambda$.
\item[{(b)}] $\varrho'(f\ox A)=(f\ox A)\varrho$.
\end{itemize}
\end{lemma}

We define the morphisms between right comodules over a regular weak multiplier
bialgebra as linear maps satisfying the equivalent conditions in Lemma
\ref{lem:comod_mor}. That is, we put the following.

\begin{definition}\label{def:comod_mor}
A {\em morphism of right comodules} over a regular weak multiplier bialgebra
$A$ is a linear map $f:V \to V'$ such that any (hence both) of the following
diagrams commutes.
$$
\xymatrix{
V\ox A \ar[r]^-\lambda\ar[d]_-{f\ox A}&
V\ox A\ar[d]^-{f\ox A}&&
V\ox A \ar[r]^-\varrho\ar[d]_-{f\ox A}&
V\ox A\ar[d]^-{f\ox A}\\
V'\ox A \ar[r]_-{\lambda'}&
V'\ox A&&
V'\ox A \ar[r]_-{\varrho'}&
V'\ox A}
$$
\end{definition}

Left comodules are defined and treated symmetrically. One of several
equivalent definitions is the following.

\begin{definition}\label{def:left_comod}
A {\em left comodule} over a regular weak multiplier bialgebra $A$ is a triple
consisting of a vector space $V$ and linear maps $\lambda,\varrho:A\ox V \to
A\ox V$ obeying the following conditions. Compatibility, in the sense that 
\begin{equation}\label{eq:left_comp}
(a\ox 1)\lambda(b\ox v)=\varrho(a\ox v)(b\ox 1)\qquad
\textrm{for any $v\in V$ and $a,b\in A$};
\end{equation}
normalization in the sense 
\begin{equation}\label{eq:left_norm}
(A\ox \lambda)\lambda^{13}(E_1\ox V)=(A\ox \lambda)\lambda^{13};
\end{equation}
and coassociativity in the sense 
\begin{equation}\label{eq:left_coass}
(A\ox \lambda)\lambda^{13}(T_1\ox V)=(T_1\ox V)\lambda^{13}.
\end{equation}
\end{definition}

\section{Integrals on weak multiplier bialgebras}\label{sec:int}

Integrals on unital (weak) bialgebras can be interpreted as comodule maps from
the regular comodule to the comodule defined on the base object. Our aim in
this section is to give a similar interpretation of integrals on regular weak
multiplier bialgebras with a right full comultiplication (when the comodule on
the base object is available). A very general notion of {\em integral} is
studied in \cite{VDa:int}. Applying it to a regular weak multiplier bialgebra
with right full comultiplication, it yields the notion in forthcoming
Proposition \ref{prop:int_def}.

Consider a regular weak multiplier bialgebra $A$ over a field $k$. By the
first axiom in (ix) in Definition \ref{def:wmba}, any linear map
$\psi:A\to k$ determines a linear map $(-) \leftharpoonup \psi: A\to \M(A)$ by
the prescriptions
$$
(a \leftharpoonup \psi)b=(\psi \ox A)T_1(a\ox b)
\qquad \textrm{and} \qquad
b(a \leftharpoonup \psi)=(\psi \ox A)T_3(a\ox b),
$$
see \cite{VDaWa_III}. For this map $(-)\leftharpoonup \psi$, the following
holds. 

\begin{lemma}\label{lem:harpoon}
Let $A$ be a regular weak multiplier bialgebra over a field $k$.
For any linear map $\psi:A\to k$ and for any $a,b\in A$,
$$
((-)\leftharpoonup \psi \ox A)T_3(a\ox b)=
(1\ox b)\overline\Delta(a \leftharpoonup \psi).
$$
\end{lemma}

\begin{proof}
By axiom (vi) of weak multiplier bialgebra in Definition \ref{def:wmba}, for
any fixed elements $b,c\in A$ there exist finitely many elements
$p^i,q^i,r^i\in A$ such that $(c\ox b)E=\sum_i(p^i\ox q^i)\Delta(r^i)=
\sum_i(p^i\ox 1)T_3(r^i\ox q^i)$. In terms of these elements (omitting the
summation symbols for brevity), for any $a\in A$, 
\begin{eqnarray*}
(c\ox 1)[((-)\leftharpoonup \psi \ox A)&&\hspace{-1cm}T_3(a\ox b)]=
(\psi\ox A \ox A)(T_3\ox A)T_3^{13}(a\ox c\ox b)\\
&=&(\psi\ox A \ox A)(T_3\ox A)T_3^{13}(A\ox E_2)(a\ox c\ox b)\\
&=&(\psi\ox A \ox A)(T_3\ox A)T_3^{13}(a\ox (p^i\ox 1)T_3(r^i\ox q^i))\\
&=&(p^i\ox 1)[(\psi\ox A \ox A)(T_3\ox A)T_3^{13}(A\ox T_3)
(a\ox r^i\ox q^i)]\\
&=&(p^i\ox 1)[(\psi\ox A \ox A)(A\ox T_3)(T_3\ox A)(a\ox r^i\ox q^i)]\\
&=&(p^i\ox 1)T_3(r^i(a \leftharpoonup \psi) \ox q^i)=
(p^i\ox 1)T_3(r^i\ox q^i)\overline\Delta(a \leftharpoonup \psi)\\
&=&(c\ox b)E\overline\Delta(a \leftharpoonup \psi)
=(c\ox b)\overline\Delta(a \leftharpoonup \psi),
\end{eqnarray*}
from which we conclude by simplifying with $c$. The second equality follows by
part (2.b), and the fifth equality follows by part (4.d) of Proposition
\ref{prop:prep_comod}, applied to the $A$-comodule $(A,T_1,T_3)$ in Example
\ref{ex:A_com}. 
\end{proof}

\begin{proposition}\label{prop:int_def}
For a regular weak multiplier bialgebra $A$ over a field $k$ with right full
comultiplication; and for a linear map $\psi:A \to k$, the following
assertions are equivalent. 
\begin{itemize}
\item[{(a)}] $(\psi\ox A)T_1=(\psi\ox A)G_1$ (where $G_1$ is the idempotent
map in \eqref{eq:G_1}). 
\item[{(b)}] $(\psi\ox A)T_3=(\psi\ox A)[(-)F]$ (where $F\in \M(A\ox A)$ is 
as in \eqref{eq:F}).
\item[{(c)}] $\overline \Delta(a \leftharpoonup \psi)=
(1\ox (a \leftharpoonup \psi))E$, for all $a\in A$.
\item[{(d)}] $a \leftharpoonup \psi\in \sqcap^R(A)$, for all $a\in A$.
\end{itemize}
If these equivalent assertions hold then we term $\psi$ a {\em right
integral} on $A$.
\end{proposition}

\begin{proof}
By the form of $G_1$ in (6.1) of \cite{BoGTLC:wmba} (cf. \eqref{eq:G_1})
and by the non-degeneracy of $A\ox A$ as a left and right $A$-module via
multiplication in the second factor, both assertions (a) and (b) are
equivalent to
$$
(\psi\ox A)[(1\ox c)\Delta(a)(1\ox b)]=
(\psi\ox A)[(a\ox c)F(1\ox b)]\qquad
\forall a,b,c\in A,
$$
hence they are equivalent also to each other.

Assertion (d) implies (c) by (3.7) in \cite{BoGTLC:wmba}. 

If (a) holds then 
$a \leftharpoonup \psi=
(\psi\ox A)[(a\ox 1)F]$
is an element of $\sqcap^R(A)$ for any $a\in A$ by \cite[Proposition
4.3~(1)]{BoGTLC:wmba} (cf. \eqref{eq:F}) and the idempotency of $A$, proving
that (d) holds. 

The implication (c)$\Rightarrow$(b) can be found in \cite[Proposition
2.7]{VDa:int}: For any $a,b,c\in A$,
\begin{eqnarray*}
(\psi \ox A\ox A)(T_3\ox A)T_3^{13}(a\ox c\ox b)
&=&(c\ox 1)[((-)\leftharpoonup \psi \ox A)T_3(a\ox b)]\\
&=&(c\ox b)\overline\Delta(a \leftharpoonup \psi)\\
&\stackrel{\textrm{(c)}}=&(c\ox b(a \leftharpoonup \psi))E\\
&=&(\psi \ox A\ox A)[T_3^{13}(a\ox c\ox b)(1\ox E)]\\
&=&(\psi \ox A\ox A)[T_3^{13}(a\ox c\ox b)](F\ox 1)].
\end{eqnarray*}
From this we conclude by using that the comultiplication is right full; that
is, any element of $A$ is a linear combination of elements of the form $(A\ox
\omega)T_3(a\ox b)$, for $a,b\in A$ and $\omega\in \mathsf{Lin}(A,k)$. 
The second equality follows by Lemma \ref{lem:harpoon} and the last one
follows since for any $a,p,q,r\in A$,
\begin{eqnarray*}
E^{13}(1\ox E)(r\ox pq\ox 1)
&=&E^{13}[r\ox (A\ox \sqcap^L)T_4(q\ox p)]\\
&=&E^{13}[((\overline \sqcap^R\ox A)\mathsf{tw} 
T_4(q\ox p))(r \ox 1)\ox 1]\\
&=&E^{13}(F\ox 1)(r\ox pq \ox 1),
\end{eqnarray*}
--- where in the first equality we applied (3.4) in \cite{BoGTLC:wmba}, in the
second equality we used \cite[Lemma 3.9]{BoGTLC:wmba}, and in the last one we
made use of \cite[Proposition 4.3~(1)]{BoGTLC:wmba} (cf. \eqref{eq:F}) ---
hence by $T_3=E_2T_3$, $T_3^{13}(a\ox c\ox b)(1\ox E)= T_3^{13}(a\ox c\ox
b)(F\ox 1)$, for any $a,b,c\in A$. 
\end{proof}

The aim of this section is to prove an isomorphism between the vector space of
right integrals on $A$; and the space of homomorphisms from the comodule $A$
in Example \ref{ex:A_com} to the comodule $R$ in Example \ref{ex:R_com} ---
assuming that the comultiplication of $A$ is right full.

Let $A$ be a regular weak multiplier bialgebra over a field $k$ 
with a right full comultiplication. Recall that, by Theorem 3.13 and Theorem
4.6 in \cite{BoGTLC:wmba}, the coinciding range of the maps $\sqcap^R$ and
$\overline \sqcap^R:A\to \M(A)$ --- the so-called base algebra $R$ --- is a
co-separable co-Frobenius coalgebra over $k$ and it is a firm $k$-subalgebra
of $\M(A)$. Applying this fact, we can relate the $k$-dual and the $R$-dual of
any firm right $R$-module. 

\begin{proposition}\label{prop:R/k-dual} 
Let $A$ be a regular weak multiplier bialgebra over a field $k$ with right
full comultiplication, and let $R$ be the coinciding range of the maps
$\sqcap^R$ and $\overline \sqcap^R:A\to \M(A)$. For any firm right $R$-module
$M$, the vector space $\mathsf{Lin}(M,k)$ of linear maps $M\to k$ is
isomorphic to the vector space $\mathsf{Hom}_R(M,R)$ of $R$-module maps $M\to
R$. 
\end{proposition}

\begin{proof}
In terms of the comultiplication $\delta:R\to R\ox R$ in \cite[Proposition
4.3~(3-4)]{BoGTLC:wmba} (cf. \eqref{eq:delta}) and the counit
$\varepsilon:R\to k$ in \cite[Proposition 4.1]{BoGTLC:wmba}
(cf. \eqref{eq:varepsilon}), consider the maps
\begin{eqnarray}\label{eq:com>int}
&\mathsf{Hom}_R(M,R) \to \mathsf{Lin}(M,k), \qquad 
&\Psi\mapsto \varepsilon \Psi\quad \textrm{and}\\
\label{eq:int>com}
&\mathsf{Lin}(M,k)\to \mathsf{Hom}_R(M,R), \qquad
&\psi \mapsto [m\br r\mapsto (\psi\ox R)((m\ox 1)\delta(r))].
\end{eqnarray}
Since $M$ is a firm right $R$-module by assumption, any element of $M$ can be
written as a linear combination of elements of the form $m\br r$ for
$m\in M$ and $r\in R$. So it is enough to define the image of $\psi$ in
\eqref{eq:int>com} on such elements. In order to see that it is a well-defined
linear map, we need to show that it maps a zero element to zero. So assume that
for some elements $m^i\in M$ and $r^i\in R$, $\sum_i m^i\br r^i=0$. Then
omitting the summation symbol for brevity, and using in the second equality
that $\delta$ is a morphism of $R$-bimodules, it follows for any $s\in R$ that 
$$
0=
(\psi \ox R)[(m^ir^i \ox 1)\delta(s)]=
(\psi \ox R)[(m^i\ox 1)\delta(r^i)(1\ox s)]=
(\psi \ox R)[(m^i\ox 1)\delta(r^i)]s.
$$
Since the multiplication in $R$ is non-degenerate by \cite[Theorem
4.6~(2)]{BoGTLC:wmba}, this proves that the image of $\psi$ in
\eqref{eq:int>com} is a well-defined linear map. It is a homomorphism of right
$R$-modules since $\delta$ is so. We claim that the maps \eqref{eq:com>int}
and \eqref{eq:int>com} establish the stated isomorphism.

Take a linear map $\psi:M \to k$ and apply to it the constructions in
\eqref{eq:int>com} and in \eqref{eq:com>int}. For any $m\in M$ and $r\in R$,
the resulting map takes $m\br r \in M$ to 
$$
\varepsilon(\psi\ox R)[(m\ox 1)\delta(r)]=
\psi[m\br ((R\ox \varepsilon)\delta(r))]=\psi(m\br r).
$$
We conclude by the surjectivity of the $R$-action on $M$ that we re-obtained
$\psi$. In the last equality we used that by \cite[Theorem 4.4]{BoGTLC:wmba},
$\varepsilon$ is the counit of the comultiplication $\delta$. 
 
Take next an $R$-module map $\Psi:M\to R$ and apply to it the constructions in
\eqref{eq:com>int} and in \eqref{eq:int>com}. For any $m\in M$ and $r\in R$,
the resulting $R$-module map takes $m\br r\in M$ to
\begin{eqnarray*}
(\varepsilon \Psi \ox R)[(m\ox 1)\delta(r)]&=&
(\varepsilon \ox R)[(\Psi(m)\ox 1)\delta(r)]\\
&=&(\varepsilon \ox R)\delta(\Psi(m)r)=
\Psi(m)r=
\Psi(m\br r).
\end{eqnarray*}
So by surjectivity of the $R$-action on $M$ we re-obtained $\Psi$. 
In the first and the last equalities we used the right $R$-module map
property of $\Psi$, the second and the third equalities follow by left
$R$-linearity and the counitality of $\delta$, respectively.
\end{proof}

Proposition \ref{prop:R/k-dual} provides in particular a vector space
isomorphism between $\mathsf{Lin}(A,k)$ and $\mathsf{Hom}_R(A,R)$, for any
regular weak multiplier bialgebra $A$ over a field $k$ with a right full
comultiplication, and its base algebra $R$. In this particular case, using the
multiplier $F$ on $A\ox A$ in \cite[Proposition 4.3~(1)]{BoGTLC:wmba}
(cf. \eqref{eq:F}), the map in \eqref{eq:int>com} can be rewritten in the
equivalent form 
$$
\mathsf{Lin}(A,k)\to \mathsf{Hom}_R(A,R), \qquad
\psi \mapsto (\psi \ox R)[(-\ox 1)F],
$$
see \cite[Proposition 4.3]{BoGTLC:wmba} (or \eqref{eq:delta}).

\begin{proposition}\label{prop:com>int}
Let $A$ be a regular weak multiplier bialgebra with right full
comultiplication. Any comodule homomorphism $\Psi$ from the comodule $A$
in Example \ref{ex:A_com} to the comodule $R$ in Example \ref{ex:R_com}, is a
homomorphism of right $R$-modules. Moreover, for the counit $\varepsilon$ of
$R$ in \cite[Proposition 4.1]{BoGTLC:wmba} (cf. \eqref{eq:varepsilon}), the
composite $\varepsilon \Psi$ is a right integral on $A$.
\end{proposition}

\begin{proof}
First we show that $\Psi$ is a homomorphism of right $R$-modules. Since $\Psi$
is a comodule map, it renders commutative the diagram
$$
\xymatrix{
A\ox A\ar[r]^-{\Psi\ox A}\ar[d]_-{T_1}&
R\ox A\ar[d]^-{r\ox a\ \mapsto \ (1\ox r)E(1\ox a)}\\
A\ox A \ar[r]_-{\Psi\ox A}&
R\ox A.}
$$
That is, $(\Psi\ox A)T_1(a\ox b)=
(1\ox \Psi(a))E(1\ox b)$
for any $a,b\in A$. Applying $\varepsilon \ox A$ to both sides of this
equality, we conclude that 
\begin{equation}\label{eq:PsiT1}
(\varepsilon\Psi\ox A)T_1(a\ox b)=
\Psi(a)((\varepsilon \ox A)(E(1\ox b)))=
\Psi(a)b.
\end{equation}
The last equality follows by the idempotency of $A$ and 
$$
(\varepsilon \ox A)(E(1\ox bc))=
(\varepsilon\overline\sqcap^R \ox A)T_1(b\ox c)=
(\epsilon \ox A)T_1(b\ox c)=
bc,
$$
for any $b,c\in A$, where we applied identity (2.3) in \cite{BoGTLC:wmba},
Lemma 4.5 (3) in \cite{BoGTLC:wmba} and the counitality axiom (iv) 
of weak multiplier bialgebra in Definition \ref{def:wmba}. Therefore,
$$
\Psi(a\overline\sqcap^R(c))b\stackrel{\eqref{eq:PsiT1}}=
(\varepsilon\Psi\ox A)T_1(a\overline\sqcap^R(c)\ox b)
\stackrel{\eqref{eq:T_1_balanced}}=
(\varepsilon\Psi\ox A)T_1(a\ox \overline\sqcap^R(c) b)
\stackrel{\eqref{eq:PsiT1}}=
\Psi(a)\overline\sqcap^R(c)b
$$
for any $a,b,c\in A$, so that $\Psi$ is a right $R$-module map. 

Applying \eqref{eq:PsiT1} in the second equality, it follows for any $a,b\in
A$ that
$$
(a\leftharpoonup \varepsilon \Psi)b=
(\varepsilon\Psi\ox A)T_1(a\ox b)=
\Psi(a)b
$$
so that $a\leftharpoonup \varepsilon \Psi=\Psi(a)$ is an element of $R$. 
Thus $\varepsilon\Psi$ is a right integral on $A$ by Proposition
\ref{prop:int_def}.
\end{proof}

\begin{proposition}\label{prop:int>com}
For a regular weak multiplier bialgebra $A$ with right full comultiplication,
consider the multiplier $F$ on $A\ox A$ in \cite[Proposition
4.3~(1)]{BoGTLC:wmba} (cf. \eqref{eq:F}). For any right integral $\psi$ on
$A$, the map \begin{equation}\label{eq:Psi}
A\to R,\qquad
a\mapsto (\psi \ox R)((a\ox 1)F)
\end{equation}
is a homomorphism from the comodule in Example \ref{ex:A_com} to the
comodule in Example \ref{ex:R_com}. 
\end{proposition}

\begin{proof}
By Lemma \ref{lem:Pi_spanned}~(4), the right $R$-action on $A$ is
surjective. Since $R$ has local units by \cite[Theorem 4.6~(2)]{BoGTLC:wmba},
this proves that $A$ is a firm right $R$-module and thus \eqref{eq:Psi} ---
being a particular instance of \eqref{eq:int>com} --- is a well-defined (right
$R$-module) map by (the proof of) Proposition \ref{prop:R/k-dual}.

In light of the form of the structure map $\varrho$ in Example \ref{ex:R_com},
\eqref{eq:Psi} is a comodule map provided that for all 
$a,b,c \in A$, 
$$
[(\psi \ox A)((-\ox c)F)\ox A]T_3(a\ox b)=
[c\ox (\psi\ox A)((a\ox b)F)]E.
$$ 
By Proposition \ref{prop:int_def}~(b), this is equivalent to
$$
[(\psi \ox A)T_3(-\ox c)\ox A]T_3(a\ox b)=
[c\ox (\psi\ox A)T_3(a\ox b)]E,
$$
hence to
$$
[c((-)\leftharpoonup \psi)\ox A]T_3(a\ox b)=
[c\ox b(a\leftharpoonup \psi)]E.
$$
This equality holds by Proposition \ref{prop:int_def}~(c) and Lemma
\ref{lem:harpoon}. 
\end{proof}

Combining Proposition \ref{prop:R/k-dual} with Proposition \ref{prop:com>int}
and Proposition \ref{prop:int>com}, we proved the following.

\begin{theorem}\label{thm:int}
Consider a regular weak multiplier bialgebra $A$ with right full
comultiplication. The constructions in Proposition \ref{prop:com>int} and in
Proposition \ref{prop:int>com} establish a vector space isomorphism between
the right integrals on $A$ on one hand, and the homomorphisms from the
comodule in Example \ref{ex:A_com} to the comodule in Example \ref{ex:R_com}
on the other hand.
\end{theorem}

\section{The bimodule structure of full comodules}\label{sec:bim}

In this section we consider a particular class of comodules over a regular
weak multiplier bialgebra $A$ that we call {\em full} comodules. Assuming that
also the comultiplication of $A$ is right full, full right comodules are shown
to carry natural bimodule structures over the base algebra $R$ of $A$. This
results in a faithful functor from the full subcategory of full right
$A$-comodules to the category of firm $R$-bimodules.

The definition of full comodule will be based on the following lemma. 

\begin{lemma}\label{lem:full}
For a regular weak multiplier bialgebra $A$ over a field $k$, and a right
$A$-comodule $(V,\lambda,\varrho)$, the following statements are equivalent.
\begin{itemize}
\item[{(a)}] $V=\langle (V\ox \omega)\lambda(v\ox a)\ | \
v\in V,\ a\in A,\ \omega\in \mathsf{Lin}(A,k) \rangle$.
\item[{(b)}] $V=\langle (V\ox \omega)\varrho(v\ox a)\ | \
v\in V,\ a\in A,\ \omega\in \mathsf{Lin}(A,k) \rangle$.
\end{itemize}
\end{lemma}

\begin{proof}
We only prove (a)$\Rightarrow$(b), the converse implication follows
symmetrically. The idea of the proof is the same as in \cite[Lemma
1.11]{VDaWa:Banach}. Take any linear decomposition of $V$ of the form
$$
\langle (V\ox \omega)\varrho(v\ox a)\ | \
v\in V,\ a\in A,\ \omega\in \mathsf{Lin}(A,k) \rangle \oplus W
$$
and any linear functional $\varphi:V\to k$ vanishing on 
$\langle (V\ox \omega)\varrho(v\ox a)\ | \
v\in V,\ a\in A,\ \omega\in \mathsf{Lin}(A,k) \rangle$.
Then
$$
0=\varphi(V\ox \omega)\varrho(v\ox a)=
\omega(\varphi \ox A)\varrho(v\ox a),
$$
for all $v\in V$, $a\in A$, and $\omega\in \mathsf{Lin}(A,k)$. Hence
$0=(\varphi \ox A)\varrho(v\ox a)$ for all $v\in V$ and $a\in A$ and therefore 
$$
0=((\varphi \ox A)\varrho(v\ox a))b
\stackrel{\eqref{eq:comp}}=
a((\varphi \ox A)\lambda(v\ox b))
$$
for all $v\in V$ and $a,b\in A$. Thus by the non-degeneracy of $A$,
$0=(\varphi \ox A)\lambda(v\ox b)$ for all $v\in V$ and $b\in A$ and so 
$$
0=\omega(\varphi \ox A)\lambda(v\ox b)=
\varphi(V\ox \omega)\lambda(v\ox b)
$$
for all $v\in V$, $b\in A$, and $\omega\in \mathsf{Lin}(A,k)$. If (a) holds
then this implies $\varphi=0$ and hence $W=\emptyset$, proving that also (b)
holds.
\end{proof}

\begin{definition}\label{def:full}
A right comodule $(V,\lambda,\varrho)$ for a regular weak multiplier bialgebra
$A$ over a field $k$ is said to be {\em full} if the equivalent conditions in
Lemma \ref{lem:full} hold. That is, $V$ is equal to any (hence both) of the
subspaces 
\begin{eqnarray*}
&&\langle (V\ox \omega)\lambda(v\ox a)\ | \
v\in V,\ a\in A,\ \omega\in \mathsf{Lin}(A,k) \rangle
\qquad \textrm{and}\\
&&\langle (V\ox \omega)\varrho(v\ox a)\ | \
v\in V,\ a\in A,\ \omega\in \mathsf{Lin}(A,k) \rangle.
\end{eqnarray*}
\end{definition}

For a regular weak multiplier bialgebra $A$, we denote by $M^{(A)}$ the full
subcategory in the category of right $A$-comodules whose objects are the full
comodules.

As discussed in Remark \ref{rem:unital_com}, for a usual, unital weak
bialgebra $A$ as in \cite{WHAI,Nill}, a right $A$-comodule in the sense of
Definition \ref{def:comod} is given by a coassociative coaction $\tau:V\to
V\ox A$. Its fullness in the sense of Definition \ref{def:full} is in fact
equivalent to counitality in the usual sense: $(V\ox \epsilon)\tau=V$. This
will be proven in Remark \ref{rem:counital}. 

\begin{example}\label{ex:A_full}
Consider a regular weak multiplier bialgebra $A$. The comultiplication of
$A$ is right full if and only if the $A$-comodule in Example \ref{ex:A_com} is
full. 
\end{example}

\begin{example}\label{ex:R_full}
For a regular weak multiplier bialgebra $A$ with right full comultiplication,
the $A$-comodule in Example \ref{ex:R_com} is full. 
\end{example}

\begin{proof}
By equality (3.2) in \cite{BoGTLC:wmba}, any element of $R$ can be written as
$(R\ox \epsilon)((1\ox a)E)$ for an appropriate element $a$ of $A$. By 
Lemma \ref{lem:Pi_spanned}~(4), any element of $A$ can be written as a linear
combination of elements of the form $ar$ for $a\in A$ and $r\in R$. Hence
\begin{eqnarray*}
R&\subseteq&
\langle (R\ox \epsilon)((1\ox a)E)\ | \ a\in A \rangle \\
&\subseteq&
\langle (R\ox \omega)((1\ox ar)E)=(R\ox \omega)\varrho(r\ox a)
\ | \ r\in R,\ a\in A,\ \omega \in \mathsf{Lin}(A,k) \rangle 
\subseteq R,
\end{eqnarray*}
so that the comodule in Example \ref{ex:R_com} is full.
\end{proof}

\begin{theorem}\label{thm:ff}
For any regular weak multiplier bialgebra $A$ with right full comultiplication,
there is a faithful functor $M^{(A)}\to {}_R M_R$.
\end{theorem}

\begin{proof}
First we equip any full right $A$-comodule $V$ with the structure of
$R$-bimodule. As right and left actions, we put
$$
v\br \overline \sqcap^R(a):=(V\ox \epsilon)\lambda(v\ox a)
\qquad \textrm{and}\qquad
\sqcap^R(a)\bl v:=(V\ox \epsilon)\varrho(v\ox a).
$$
In order to see that the right action is well defined, we need to check that
whenever $\overline \sqcap^R(a)=0$, also $(V\ox \epsilon)\lambda(v\ox
a)=0$ for any $v\in V$. So let us assume that $\overline \sqcap^R(a)=0$ for
some $a\in A$. Then also $\sqcap^L(a)=0$ by \cite[Lemma 4.8~(4)]{BoGTLC:wmba},
so that for any $v\in V$ and $b\in A$, 
$$
0=\lambda(v\ox \sqcap^L(a)b)=
(V\ox \epsilon\ox A)(\lambda \ox A)\lambda^{13}(v\ox a\ox b),
$$
where the equality follows by Proposition \ref{prop:old_cu}~(1). Applying any
$\omega\in \mathsf{Lin}(A,k)$ to the second factor, we obtain 
$$
0=(V\ox \epsilon)\lambda((V\ox \omega)\lambda(v\ox b)\ox a),
$$
for any $v\in V$, $b\in A$ and $\omega\in \mathsf{Lin}(A,k)$. By the
assumption that $V$ is a full comodule, this implies $(V\ox
\epsilon)\lambda(v\ox a)=0$ for all $v\in V$. Thus the right $R$-action on
$V$ is well-defined. It is proven symmetrically that the left action is
well-defined.

The right action is associative by 
\begin{eqnarray*}
(v\br \overline \sqcap^R(a))\br \overline \sqcap^R(b)&=&
(V\ox \epsilon\ox \epsilon)(\lambda \ox A)\lambda^{13}(v\ox b\ox a)=
(V\ox \epsilon)\lambda(v\ox \sqcap^L(b)a)\\
&=&v\br \overline \sqcap^R(\sqcap^L(b)a)=
v\br (\overline \sqcap^R(a)\overline \sqcap^R(b)),
\end{eqnarray*}
for any $v\in V$ and $a,b\in A$.
The second equality follows by Proposition \ref{prop:old_cu}~(1) and the last
one follows by \cite[Lemma 3.12]{BoGTLC:wmba}. Associativity of the left
action is proven symmetrically.

In order to see that the left and right actions commute, compute for any $v\in
V$ and $a,b\in A$
\begin{eqnarray*}
(\sqcap^R(a)\bl v)\br \overline \sqcap^R(b)&=&
(V\ox \epsilon\ox \epsilon)\lambda^{13}(\varrho\ox A)(v\ox a\ox b)\\
&=&(V\ox \epsilon)(\varrho(v\ox a)(1\ox \sqcap^L(b)))\\
&=&(V\ox \epsilon)(\varrho(v\ox a)(1\ox b)).
\end{eqnarray*}
The second equality follows by Lemma \ref{lem:bim}~(1) and the last one
follows by identity (3.5) in \cite{BoGTLC:wmba}. Symmetrically, compute 
\begin{eqnarray*}
\sqcap^R(a)\bl (v \br \overline \sqcap^R(b))&=&
(V\ox \epsilon\ox \epsilon)\varrho^{13}(\lambda\ox A)(v\ox b\ox a)\\
&=&(V\ox \epsilon)((1\ox \overline\sqcap^L(a))\lambda(v\ox b))\\
&=&(V\ox \epsilon)((1\ox a)\lambda(v\ox b)).
\end{eqnarray*}
The second equality follows by Lemma \ref{lem:bim}~(2) and the last one
follows by Lemma 3.1 in \cite{BoGTLC:wmba}. These expressions are equal by
\eqref{eq:comp}, proving that $V$ is an $R$-bimodule.

Since $R$ has local units by \cite[Theorem 4.6~(2)]{BoGTLC:wmba}, an
$R$-module is firm if and only if the $R$-action on it is surjective.
So we only need to show that the above $R$-actions are surjective. In the case
of the right action it can be seen by the following considerations; for the
left action symmetric reasoning applies.
By Lemma \ref{lem:Pi_spanned}~(3), $A$ is spanned by elements of the form
$\sqcap^L(b)a$ for $a,b\in A$. Since $V$ is a full comodule by assumption, we
conclude that it is spanned by elements of the form 
\begin{eqnarray*}
(V\ox \omega)\lambda(v\ox \sqcap^L(b)a)=
(V\ox \epsilon)\lambda((V\ox \omega)\lambda(v\ox a)\ox b),
\end{eqnarray*}
for $v\in V$, $a,b\in A$ and $\omega\in \mathsf{Lin}(A,k)$, where we applied
Proposition \ref{prop:old_cu}~(1). Hence $V$ is spanned by elements of the
form 
$(V\ox \epsilon)\lambda(v \ox b)=v\br \overline \sqcap^R(b),$
for $v\in V$ and $b\in A$.

The above construction gives the object map of the stated functor $M^{(A)}\to
{}_R M_R$. On the morphisms it acts as the identity map. Any morphism $f$ in
$M^{(A)}$ is indeed a morphism of right $R$-modules by
\begin{eqnarray*}
f(v\br \overline \sqcap^R(a))&=&
f(V\ox \epsilon)\lambda(v\ox a)=
(V'\ox \epsilon)(f\ox A)\lambda(v\ox a)\\
&=&(V'\ox \epsilon)\lambda'(f(v)\ox a)=
f(v)\br \overline \sqcap^R(a)
\end{eqnarray*}
for any $v\in V$ and $a\in A$; and it is a morphism of left $R$-modules by a
symmetric reasoning.
\end{proof}

\begin{remark}\label{rem:counital}
Consider a usual weak bialgebra $A$ as in \cite{WHAI,Nill} possessing an
algebraic unit $1$. Let $(V,\lambda,\varrho)$ be a right $A$-comodule in the
sense of Definition \ref{def:comod}. As discussed in Remark
\ref{rem:unital_com}, this means the existence of a coassociative coaction 
$$
\tau:=\lambda(-\ox 1)=\varrho(-\ox 1):V\to V\ox A,\qquad
v \mapsto v^0\ox v^1.
$$
Here we claim that $\tau$ is counital --- i.e. $(V\ox \epsilon)\tau=V$ --- if
and only if $V$ is a full comodule in the sense of Definition \ref{def:full}.

Assume first that $\tau$ is counital. Then
\begin{eqnarray*}
V&=&
\langle v=(V\ox \epsilon)\tau(v)=(V\ox \epsilon)\lambda(v\ox 1)\ |\ v\in V
\rangle\\
&\subseteq& \langle (V\ox \omega)\lambda(v\ox a)\ |\ v\in V,\ a\in A,\ \omega
\in \mathsf{Lin}(A,k) \rangle \subseteq V
\end{eqnarray*}
so that $V$ is a full comodule. 

Conversely, assume that the comodule $(V,\lambda,\varrho)$ is full. Then the
$R$-actions on $V$ are surjective by Theorem \ref{thm:ff}. Because in this
case $R$ is a unital algebra, this implies that the $R$-actions on $V$ are
unital as well. Then for any $v\in V$,
$$
v=
v\br \overline \sqcap^R(1)=
(V\ox \epsilon)\lambda(v\ox 1)=
(V\ox \epsilon)\tau(v),
$$
so that $\tau$ is counital. 
\end{remark}

\begin{example}\label{ex:A_bim}
Consider a regular weak multiplier bialgebra $A$ with right full
comultiplication. Applying the functor in Theorem \ref{thm:ff} to the full
comodule in Example \ref{ex:A_full}, we obtain the $R$-actions 
\begin{eqnarray*}
&&a\ox \overline \sqcap^R(b)\mapsto(A\ox \epsilon)T_1(a\ox b)=
a\overline \sqcap^R(b)\\
&&\sqcap^R(b)\ox a\mapsto(A\ox \epsilon)T_3(a\ox b)=\sqcap^R(b)a,
\end{eqnarray*}
for $a,b\in A$. In the last equalities we used the identities (2.2) and (3.3)
in \cite{BoGTLC:wmba} (see \eqref{eq:Pibar} and \eqref{eq:Pi}), respectively.
\end{example}

\begin{example}\label{ex:R_bim}
Consider a regular weak multiplier bialgebra $A$ with right full
comultiplication. Applying the functor in Theorem \ref{thm:ff} to the full
comodule in Example \ref{ex:R_full}, we obtain the $R$-actions 
\begin{eqnarray*}
&&r\ox \overline \sqcap^R(a)\mapsto
(R\ox \epsilon)(E(1\ox ra))=
\overline \sqcap^R(ra)=
r\overline\sqcap^R(a)\\
&&\sqcap^R(a)\ox r\mapsto
(R\ox \epsilon)((1\ox ar)E)=
\sqcap^R(ar)=
\sqcap^R(a)r,
\end{eqnarray*}
for any $r\in R$ and $a\in A$. In the first line we used identity (2.2) and
Lemma 3.4 in \cite{BoGTLC:wmba}, and in the second line we used identities
(3.2) and (3.8) in \cite{BoGTLC:wmba}.
\end{example}

In terms of the $R$-actions in Theorem \ref{thm:ff}, we can reformulate
Proposition \ref{prop:old_cu} in the following form.

\begin{lemma}\label{lem:l-r-bim}
For a regular weak multiplier bialgebra $A$ with right full comultiplication,
and for a full right $A$-comodule $(V,\lambda,\varrho)$, the following
statements hold for any $v\in V$ and $a,b\in A$.
\begin{itemize} 
\item[{(1)}] $\lambda(v\ox \sqcap^L(a)b)=
\lambda(v\ox b)(\overline \sqcap^R(a)\ox 1)$.
\item[{(2)}] $\lambda(v\ox \overline \sqcap^R(a)b)=
\lambda(v\br \overline \sqcap^R(a)\ox b)$.
\item[{(3)}] $\varrho(v\ox a\sqcap^R(b))=
\varrho(\sqcap^R(b)\bl v\ox a)$.
\item[{(4)}] $\varrho(v\ox a\overline\sqcap^L(b))=
(\sqcap^R(b)\ox 1)\varrho(v\ox a)$.
\item[{(5)}] $\varrho(v\ox b)(1\ox \sqcap^L(a))=
\varrho(v\ox b)(\overline\sqcap^R(a)\ox 1)$.
\item[{(6)}] $\varrho(v\ox b)(1\ox \overline \sqcap^R(a))=
\varrho(v\br \overline\sqcap^R(a)\ox b)$.
\item[{(7)}] $(1\ox \sqcap^R(a))\lambda(v\ox b)=
\lambda(\sqcap^R(a)\bl v\ox b)$.
\item[{(8)}] $(1\ox \overline \sqcap^L(a))\lambda(v\ox b)=
(\sqcap^R(a)\ox 1)\lambda(v\ox b)$.
\end{itemize}
\end{lemma}

\begin{proof}
Assertions (1)-(4) are evident reformulations of the statements in Proposition
\ref{prop:old_cu}. The remaining parts are obtained from them applying
\eqref{eq:comp}. For example, (5) follows by 
\begin{eqnarray*}
\varrho(v\ox b)(1\ox \sqcap^L(a)c)
&\stackrel{\eqref{eq:comp}}=&
(1\ox b)\lambda(v\ox \sqcap^L(a)c)\\
&\stackrel{(1)}=&
(1\ox b)\lambda(v\ox c)(\overline \sqcap^R(a)\ox 1)\\
&\stackrel{\eqref{eq:comp}}=&
\varrho(v\ox b)(\overline \sqcap^R(a)\ox c)
\end{eqnarray*}
--- for any $v\in V$ and $a,b,c\in A$ --- and non-degeneracy of the right
$A$-module $V\ox A$.
\end{proof}

In the next section we shall need the following compatibilities between full
comodules and the idempotent $E$. Recall that for a regular weak multiplier
bialgebra $A$ with right full comultiplication, by identities (2.3) and (3.4)
in \cite{BoGTLC:wmba}, both $E(1\ox a)$ and $(1\ox a)E$ are elements of $R\ox
A$, for any $a\in A$. Hence the maps $E_1,E_2:A\ox A\to A \ox A$ allow for the
generalizations
\begin{eqnarray*}
&E_1:V\ox A \to V\ox A,\qquad
&v\ox a \mapsto ((-)\bl v\ox A)[E(1\ox a)]\equiv E(v\ox a)\\
&E_2:V\ox A \to V\ox A,\qquad
&v\ox a \mapsto (v\br (-)\ox A)[(1\ox a)E]\equiv (v\ox a)E
\end{eqnarray*}
for any full right $A$-comodule $V$. 

\begin{lemma}\label{lem:l-r-E}
For a regular weak multiplier bialgebra $A$ with right full comultiplication,
and for a full right $A$-comodule $(V,\lambda,\varrho)$, the following
statements hold.
\begin{itemize}
\item[{(1)}] $E_1\lambda=\lambda$.
\item[{(2)}] $(V\ox E_1)(\lambda \ox A)=(\lambda \ox A)E_1^{13}$.
\item[{(3)}] $E_2\varrho=\varrho$.
\item[{(4)}] $(V\ox E_2)(\varrho \ox A)=(\varrho \ox A)E_2^{13}$.
\end{itemize}
\end{lemma}

\begin{proof} (1) For any $v\in V$ and $a,b,c\in A$,
\begin{eqnarray*}
(1\ox bc)(E_1\lambda(v\ox a))&\stackrel{\mathsf{(i)}}=&
E_2(1\ox bc)\lambda(v\ox a)\\
&=&((\sqcap^R\ox A)T_3(c\ox b))\lambda(v\ox a)\\
&=&(V\ox \mu^{\mathsf{op}}(\overline \sqcap^L\ox A)T_3(c\ox b))
\lambda(v\ox a)\\
&=&(1\ox bc)\lambda(v\ox a),
\end{eqnarray*}
from which we conclude by the non-degeneracy of the left $A$-module $V\ox
A$. In the second equality we used identity (3.4) in \cite{BoGTLC:wmba}, in
the third equality we used Lemma \ref{lem:l-r-bim}~(8) and in the last
equality we used \cite[Lemma 3.7~(1)]{BoGTLC:wmba}.

(2) For any $v\in V$ and $a,b,c\in A$,
\begin{eqnarray*}
(V\ox E_1)(\lambda \ox A)(v\ox a \ox bc)&=&
(1\ox (\overline \sqcap^R \ox A)T_1(b\ox c))(\lambda(v\ox a)\ox 1)\\
&=&(\lambda\ox A)((\overline \sqcap^R \ox A)T_1(b\ox c))^{13}(v\ox a\ox 1)\\
&=&(\lambda\ox A)E_1^{13}(v\ox a\ox bc),
\end{eqnarray*}
from which we conclude by the idempotency of $A$. In the first and the last
equalities we used identity (2.3) in \cite{BoGTLC:wmba} and in the second
equality we used Lemma \ref{lem:l-r-bim}~(7). 

The remaining assertions follow symmetrically.
\end{proof}

The following technical lemmata will be applied in Section \ref{sec:Hopf_mod}.

\begin{lemma}\label{lem:l-r-Pi_ac}
Consider a regular weak multiplier bialgebra $A$ with right full
comultiplication, and a full right $A$-comodule $(V,\lambda,\varrho)$. For
any $v\in V$ and $a\in A$, introduce the index notation $\varrho(v\ox
a)=:v^\varrho\ox a^\varrho$ and $\lambda(v\ox a)=:v^\lambda\ox a^\lambda$
(where in both cases implicit summation is understood). Then for any $v\in V$
and $a\in A$, the following identities hold.
\begin{itemize}
\item[{(1)}] $v^\varrho\br \sqcap^R(a^\varrho)=\sqcap^R(a)\bl v$.
\item[{(2)}] $\overline \sqcap^R(a^\lambda)\bl v^\lambda=
v\br \overline \sqcap^R(a)$.
\end{itemize}
\end{lemma}

\begin{proof}
(1) For any $a,b\in A$ and any $v\in V$,
$$
v\br \sqcap^R(ab)=
v\br \sqcap^R(\sqcap^R(a)b)=
(V\ox \epsilon)[(v\ox 1)((\sqcap^R\ox A)T_3(b\ox a))]=
(V\ox \epsilon)E_2(v\ox ab).
$$
The first equality follows by identity (3.6), the second one follows by (3.3),
and the last one follows by (3.4) in \cite{BoGTLC:wmba}. By the idempotency of
$A$, this proves $v\br \sqcap^R(a)=(V\ox \epsilon)E_2(v\ox a)$. With this
identity at hand, 
$$
v^\varrho\br \sqcap^R(a^\varrho)=
(V\ox \epsilon)E_2\varrho(v\ox a)=
(V\ox \epsilon)\varrho(v\ox a)=
\sqcap^R(a)\bl v.
$$
In the second equality we used Lemma \ref{lem:l-r-E}~(3). 
Assertion (2) follows symmetrically.
\end{proof}

\begin{lemma}\label{lem:prep_dual_comod}
Let $A$ be a regular weak multiplier bialgebra with right full
comultiplication and let $(V,\lambda,\varrho)$ be a full right $A$-comodule.
For any $v\in V$, $a,b\in A$ and $r\in R$, the following identities hold. 
\begin{itemize}
\item[{(1)}] $(V\ox \sqcap^R)\varrho(v\br r\ox b)=
(\sqcap^R(b)\bl v \ox 1)\delta(r)$ (where $\delta$ denotes the
 comultiplication on the coalgebra $R$ in \cite[Theorem 4.4]{BoGTLC:wmba};
 cf. \eqref{eq:delta}).
\item[{(2)}] $(1\ox b)((V\ox \sqcap^L)\varrho(v\ox a))=
((\sqcap^R\ox A)T_4(a\ox b))(v\ox 1)$.
\end{itemize}
\end{lemma}

\begin{proof}
(1) For any $v\in V$, $r\in R$ and $a,b,c\in A$, denote $\varrho(v\br r\ox
b)=(v\br r)^\varrho\ox b^\varrho$ and $T_2(c\ox a)=c^2\ox a^2$, where in
both cases implicit summation is understood. Then 
\begin{eqnarray*}
(1\ox c)((V\ox \sqcap^R)\varrho(v\br r\ox b))(1\ox a)&=&
(v\br r)^\varrho\ox (A\ox \epsilon)((c\ox 1)T_3(a\ox b^\varrho))\\
&=&(v\br r)^\varrho\ox (A\ox \epsilon)((1\ox b^\varrho)T_2(c\ox a))\\
&=&(v\br r)^\varrho\ox c^2\epsilon(b^\varrho\overline\sqcap^R(a^2))\\
&=&(V\ox \epsilon)\varrho(v\br r\overline\sqcap^R(a^2) \ox b)\ox c^2\\
&=&\sqcap^R(b)\bl v\br r\overline\sqcap^R(a^2)\ox c^2\\
&=&(1\ox c)(\sqcap^R(b)\bl v\ox 1)\delta(r)(1\ox a),
\end{eqnarray*}
from which we conclude simplifying by $a$ and $c$.
The first equality follows by (3.3) in \cite{BoGTLC:wmba}
(cf. \eqref{eq:Pi}). In the second equality we applied \eqref{eq:T_23}.
The third equality holds by
\cite[Lemma 3.1]{BoGTLC:wmba} and the fourth one does by Lemma
\ref{lem:l-r-bim}~(6). In the last equality we used that by 
\cite[Proposition 4.3~(1)]{BoGTLC:wmba} (cf. \eqref{eq:F}) and by 
the second equality in axiom (ix) in Definition \ref{def:wmba},
the multiplier $F$ in \cite[Proposition 4.3~(1)]{BoGTLC:wmba} satisfies 
$$
(1\ox c)F(1\ox ab)=
(1\ox c)((\overline\sqcap^R\ox A)\mathsf{tw}T_4(b\ox a))=
((\overline\sqcap^R\ox A)\mathsf{tw}T_2(c\ox a))(1\ox b).
$$
Hence $\delta(r)=(r\ox 1)F$ satisfies $r\overline\sqcap^R(a^2)\ox c^2=(1\ox
c)\delta(r)(1\ox a)$. 

(2) For any $v\in V$ and $a,b,c\in A$,
\begin{eqnarray*}
(1\ox c\ox 1)((V\ox T_4)(\varrho\ox A)(v\ox a\ox b))&=&
(1\ox T_2(c\ox b))(\varrho(v\ox a)\ox 1)\\
&=&(\varrho\ox A)(v\ox T_2(c\ox b)(a\ox 1))\\
&=&(\varrho\ox A)(v\ox (c\ox 1)T_4(a\ox b))\\
&=&(1\ox c\ox 1)((\varrho\ox A)(V\ox T_4)(v\ox a\ox b)).
\end{eqnarray*}
In the first and the third equalities we used the second equality in axiom
(ix) in Definition \ref{def:wmba}. 
In the second and the last equalities we used that $\varrho$ is a left
$A$-module map, cf. Proposition \ref{prop:prep_comod}~(1). Simplifying by $c$,
we obtain 
\begin{equation}\label{eq:T4-rho}
(V\ox T_4)(\varrho\ox A)=
(\varrho\ox A)(V\ox T_4).
\end{equation}
Using this identity in the second equality and (3.3) in \cite{BoGTLC:wmba}
(cf. \eqref{eq:Pi}) in the first one, it follows for any $v\in V$ and $a,b\in
A$ that
\begin{eqnarray*}
(1\ox b)((V\ox \sqcap^L)\varrho(v\ox a))&=&
(V\ox \epsilon\ox A)(V\ox T_4)(\varrho\ox A)(v\ox a\ox b)\\
&=&(V\ox \epsilon\ox A)(\varrho\ox A)(V\ox T_4)(v\ox a\ox b)\\
&=&((\sqcap^R\ox A)T_4(a\ox b))(v\ox 1).
\end{eqnarray*}
\end{proof}

\section{The monoidal category of full comodules}\label{sec:mon_cat}

For a unital (weak) bialgebra, both the category of modules and the category
of comodules are known to be monoidal with respect to the monoidal product
provided by the module tensor product over the base algebra. For a regular
(weak) multiplier bialgebra $A$ with a full comultiplication, it was shown in 
\cite[Section 5]{BoGTLC:wmba} that the category of idempotent non-degenerate
$A$-modules is monoidal with respect to the same monoidal product provided by
the module tensor product over the base algebra. The aim of this section is to
prove a similar result about $A$-comodules: We show that the category
$M^{(A)}$ of full $A$-comodules is monoidal such that the functor in Theorem
\ref{thm:ff} is strict monoidal. 

\begin{lemma}\label{lem:k_product}
For a regular weak multiplier bialgebra $A$ and any right $A$-comodules
$(V,\lambda_V,$ $\varrho_V)$ and $(W,\lambda_W,\varrho_W)$, there is an
$A$-comodule 
$$
(V\ox W,\lambda_V^{13}(V\ox \lambda_W),(V\ox \varrho_W)\varrho_V^{13}).
$$
\end{lemma}

\begin{proof}
Since \eqref{eq:comp} holds both for $V$ and $W$, it clearly holds for the
stated comodule $V\ox W$. Since the normalization condition
\eqref{eq:l-i-norm} holds for $W$, the top left region in 
$$
\xymatrix{
V\ox W \ox A\ox A\ar[r]^-{
\raisebox{7pt}{${}_{V\ox W\ox E_1}$}}\ar[d]_-{\lambda_W^{24}}&
V\ox W \ox A\ox A\ar[r]^-{
\raisebox{7pt}{${}_{\lambda_W^{24}}$}}&
V\ox W \ox A\ox A\ar[d]^-{\lambda_W^{23}}\ar[r]^-{
\raisebox{7pt}{${}_{\lambda_V^{14}}$}}&
V\ox W \ox A\ox A\ar[dd]^-{\lambda_W^{23}}\\
V\ox W \ox A\ox A\ar[rr]^-{\lambda_W^{23}}\ar[d]_-{\lambda_V^{14}}&&
V\ox W \ox A\ox A\ar[rd]^-{\lambda_V^{14}}&\\
V\ox W \ox A\ox A\ar[rrr]_-{\lambda_W^{23}}&&&
V\ox W \ox A\ox A}
$$
commutes. Postcomposing both paths around this diagram by $\lambda_V^{13}$, we
conclude that the normalization condition \eqref{eq:l-i-norm} holds for the
stated comodule $V\ox W$. Finally, since the coassociativity condition
\eqref{eq:l-coass} holds both for $V$ and $W$, 
$$
\xymatrix{
V\ox W \ox A\ox A\ar[r]^-{
\raisebox{7pt}{${}_{V\ox W\ox T_1}$}}\ar[d]_-{\lambda_W^{23}}&
V\ox W \ox A\ox A\ar[r]^-{
\raisebox{7pt}{${}_{\lambda_W^{24}}$}}&
V\ox W \ox A\ox A\ar[d]^-{\lambda_W^{23}}\ar[r]^-{
\raisebox{7pt}{${}_{\lambda_V^{14}}$}}&
V\ox W \ox A\ox A\ar[d]^-{\lambda_W^{23}}\\
V\ox W \ox A\ox A\ar[rr]^-{V\ox W\ox T_1}\ar[d]_-{\lambda_V^{13}}&&
V\ox W \ox A\ox A\ar[r]^-{\lambda_V^{14}}&
V\ox W \ox A\ox A\ar[d]^-{\lambda_V^{13}}\\
V\ox W \ox A\ox A\ar[rrr]_-{V\ox W\ox T_1}&&&
V\ox W \ox A\ox A}
$$
commutes proving that \eqref{eq:l-coass} holds also for the
stated comodule $V\ox W$. 
\end{proof}

The comodule in Lemma \ref{lem:k_product} is not yet the appropriate monoidal
product. 
With the aim of lifting the monoidal structure of ${}_R M_R$ to $M^{(A)}$, the
candidate monoidal unit in $M^{(A)}$ is the full $A$-comodule $R$ in Example
\ref{ex:R_full}. The candidate monoidal product of full $A$-comodules $V$ and
$W$ is their $R$-module tensor product $V\ox_R W$. 

\begin{proposition}\label{prop:ox_R}
Let $A$ be a regular weak multiplier bialgebra with right full
comultiplication. For any full right $A$-comodules $(V,\lambda_V,\varrho_V)$
and $(W, \lambda_W,\varrho_W)$, $V\ox_R W$ is isomorphic to the range of the
idempotent map
$$
G_1:V\ox W \to V \ox W,\qquad
\sqcap^R(a)\act v\ox w \act \overline \sqcap^R(b) \mapsto 
v^\varrho \ox w^\lambda \epsilon(a^\varrho b^\lambda),
$$
where the implicit summation index notation $\varrho_V(v\ox a)=:v^\varrho \ox
a^\varrho$ and $\lambda_W(w\ox b)=:w^\lambda\ox b^\lambda$ is used.
\end{proposition}

\begin{proof}
By \cite[Theorem 4.6~(1)]{BoGTLC:wmba} $R$ is a coseparable coalgebra. Thus
by \cite[Proposition 2.17]{BoVe}, for any firm $R$-bimodules $V$ and $W$,
$V\ox_R W$ is isomorphic to the range of the idempotent map
\begin{equation}\label{eq:Theta}
V\ox W \to V \ox W,\qquad 
v\act r \ox w \mapsto (v\ox 1)\act \delta(r) \act (1\ox w),
\end{equation}
where $\delta$ denotes the comultiplication in the coalgebra $R$
(cf. \cite[Theorem 4.4]{BoGTLC:wmba} and \eqref{eq:delta}) and we use that $V$
is spanned by elements of the form $v\act r$ in terms of $v\in V$ and $r\in
R$.

In view of Theorem \ref{thm:ff} and the idempotency of $A$, the tensor product
vector space $V\ox W$ is spanned by elements of the form $\sqcap^R(a)\act
v\act \sqcap^R(cd) \ox w \act \overline \sqcap^R(b)$, for $v\in V$, $w\in W$
and $a,b,c,d\in A$. Evaluate \eqref{eq:Theta} on this element. Introducing the
implicit summation index notation $T_2(c\ox d)=c^2\ox d^2$ and applying (4.1)
in \cite{BoGTLC:wmba}, we obtain 
\begin{eqnarray*}
\sqcap^R(a)\act v\act \sqcap^R(c^2)\hspace{-.4cm}&&\hspace{-.5cm}\ox 
\sqcap^R(d^2)\act w \act \overline \sqcap^R(b)\\
&=&
(V\ox \epsilon)\varrho_V(v\act \sqcap^R(c^2) \ox a)\ox
(W\ox \epsilon)\lambda_W(\sqcap^R(d^2)\act w \ox b)\\
&=&
(V\ox \epsilon)(\varrho_V(v \ox a)(1\ox \sqcap^R(c^2)))\ox
(W\ox \epsilon)((1\ox \sqcap^R(d^2))\lambda_W(w \ox b))\\
&=&
(V\ox \epsilon)(\varrho_V(v \ox a)(1\ox \sqcap^R(c^2)))\ox
(W\ox \epsilon)((1\ox d^2)\lambda_W(w \ox b))\\
&=&
v^\varrho\ox w^\lambda 
\epsilon(a^\varrho(\sqcap^R\ox \epsilon)(T_2(c\ox d)(1\ox b^\lambda)))\\
&=&
v^\varrho\ox w^\lambda 
\epsilon(a^\varrho(\sqcap^R\ox \epsilon)
(T_2(c\ox d)(1\ox \overline \sqcap^R(b^\lambda))))\\
&=&
v^\varrho\ox w^\lambda 
\epsilon(a^\varrho(\sqcap^R\ox \epsilon)T_2(c\ox d\overline\sqcap^R(b^\lambda)))\\
&=&
v^\varrho\ox w^\lambda 
\epsilon(a^\varrho\sqcap^R(cd\overline\sqcap^R(b^\lambda)))\\
&=&
v^\varrho\ox w^\lambda 
\epsilon(a^\varrho\sqcap^R(cd)\overline\sqcap^R(b^\lambda))\\
&=&
v^\varrho\ox w^\lambda 
\epsilon(a^\varrho\sqcap^R(cd)b^\lambda)\\
&=&
(v\act\sqcap^R(cd)) ^\varrho\ox w^\lambda 
\epsilon(a^\varrho b^\lambda).
\end{eqnarray*}
This proves that the idempotent map \eqref{eq:Theta} is equal to the stated
map $G_1$.
In the first equality we used the $R$-actions in Theorem \ref{thm:ff}. In the
second and the last equalities we used parts (6)-(7) and (6) of Lemma
\ref{lem:l-r-bim}, respectively. The third equality follows by one of the
identities in (3.5) in \cite{BoGTLC:wmba}. The fifth and the penultimate
equalities follow by \cite[Lemma 3.1]{BoGTLC:wmba}. In the sixth equality we
used that by \cite[Lemma 3.3]{BoGTLC:wmba}, $T_2(c\ox d)(1\ox \overline
\sqcap^R(b))=T_2(c\ox d\overline \sqcap^R(b))$. The seventh equality follows
by the counitality axiom (iv) of weak multiplier bialgebra in Definition
\ref{def:wmba} and the eighth one follows by \cite[Lemma3.11]{BoGTLC:wmba}.
\end{proof}

\begin{example}
It is immediate from the proof of Proposition \ref{prop:ox_R} that for the
full $A$-comodule $A$ in Example \ref{ex:A_full}, $G_1:A\ox A \to A\ox A$ is
equal to the idempotent map in \eqref{eq:G_1}. More generally, for any full
right $A$-comodule $V$, we can write $G_1:V\ox A \to V\ox A$ as $v\ox a
\mapsto (v(-)\ox A)[F(1\ox a)]$ (which is meaningful by \cite[Proposition
4.3~(1)]{BoGTLC:wmba}; cf. \eqref{eq:F}).
\end{example}

In order to see that the $R$-module tensor product of any full right comodules
is again a full right comodule, we need some compatibilities between the
coactions and the map $G_1$ in Proposition \ref{prop:ox_R}.

\begin{lemma}\label{lem:l-r-Theta}
Consider a regular weak multiplier bialgebra $A$ with right full
comultiplication. For any full right $A$-comodules $V$ and $W$, the map
$G_1$ in Proposition \ref{prop:ox_R} satisfies the following identities. 
\begin{itemize}
\item[{(1)}] $\lambda_V G_1=\lambda_V$.
\item[{(2)}] $(V\ox \lambda_W)(G_1\ox A)=
G_1^{13}(V\ox \lambda_W)$.
\item[{(3)}] $\lambda_V^{13}(V\ox \lambda_W)(G_1\ox A)=
\lambda_V^{13}(V\ox \lambda_W)$.
\item[{(4)}] $(G_1\ox A)\lambda_V^{13}=\lambda_V^{13}(V\ox E_1)$.
\item[{(5)}] $(G_1\ox A)\lambda_V^{13}(V\ox \lambda_W)=
\lambda_V^{13}(V\ox \lambda_W)$.
\item[{(6)}] $\varrho_V G_1^{21}=\varrho_V$.
\item[{(7)}] $\varrho_V^{13}(G_1\ox A)=
(V\ox G_1^{21})\varrho_V^{13}$. 
\item[{(8)}] $(V\ox \varrho_W)\varrho_V^{13}(G_1\ox A)=
(V\ox \varrho_W)\varrho_V^{13}$. 
\item[{(9)}] $(G_1\ox A)(V\ox \varrho_W)=(V\ox \varrho_W)E_2^{13}$.
\item[{(10)}] $(G_1\ox A)(V\ox \varrho_W)\varrho_V^{13}=
(V\ox \varrho_W)\varrho_V^{13}$.
\end{itemize}
\end{lemma}

\begin{proof}
(1) For any $v\in V$, $r\in R$ and $a\in A$, 
$$
\lambda_V G_1(v\ox ra)=
\lambda_V((v\ox 1)\act \delta(r)(1\ox a))=
\lambda_V(v\ox \mu \delta(r) a)=
\lambda_V(v\ox ra).
$$
The second equality follows by Lemma \ref{lem:l-r-bim}~(2) 
and in the last equality we used that the comultiplication $\delta$ of the
coseparable coalgebra $R$ is a section of the multiplication $\mu$, see
\cite[Proposition 4.3~(3)]{BoGTLC:wmba}. Since $A$ is spanned by elements of
the form $ra$, for $r\in R$ and $a\in A$ (cf. Lemma \ref{lem:Pi_spanned}), this
proves the claim.

(2) Let us use the implicit summation index notation $T_2(a\ox
b)=:a^2\ox b^2$ and $\lambda_W(w\ox c)=:w^\lambda\ox c^\lambda$ for any
$a,b,c\in A$ and $w\in W$. For any $v\in V$,
\begin{eqnarray*}
(V\ox \lambda_W)(G_1\ox A)\hspace{-1cm}&&
(v\ox \sqcap^R(ab)\act w\ox c)\\
&=&
(V\ox \lambda_W)(v\act\sqcap^R(a^2)\ox \sqcap^R(b^2)\act w \ox c)\\
&=&
v\act \sqcap^R(a^2)\ox w^\lambda \ox \sqcap^R(b^2)c^\lambda\\
&=&
G_1^{13}(v\ox w^\lambda \ox \sqcap^R(ab)c^\lambda)\\
&=&
G_1^{13}(V\ox \lambda_W)(v\ox \sqcap^R(ab)\act w\ox c).
\end{eqnarray*}
The first and the third equalities follow by (4.1) in \cite{BoGTLC:wmba} and
in the second and the last equalities we applied Lemma
\ref{lem:l-r-bim}~(7). Since $R$ is idempotent and the left $R$-action on $W$
is surjective by Theorem \ref{thm:ff}, this proves the claim.

(3) follows immediately from (1) and (2).

(4) For any $v\in V$, $w\in W$ and $a,b,c\in A$, 
\begin{eqnarray*}
(G_1\ox A)\lambda_V^{13}(v\ox \overline\sqcap^R(ab)\act w\ox c)&=&
v^\lambda\act \overline\sqcap^R(b^1) \ox 
\overline\sqcap^R(a^1)\act w \ox c^\lambda\\
&=&
\lambda_V^{13}(v\ox \overline\sqcap^R(a^1)\act w \ox \sqcap^L(b^1)c)\\
&=&
\lambda_V^{13}(V\ox E_1)(v\ox \overline\sqcap^R(ab)\act w\ox c).
\end{eqnarray*}
where we used the notation $\lambda_V(v\ox c)=v^\lambda\ox c^\lambda$ and
$T_1(a\ox b)=a^1\ox b^1$. The first equality follows by \cite[Lemma
4.5~(2)]{BoGTLC:wmba}. The second one follows by Lemma \ref{lem:l-r-bim}~(1)
and the last one does by Lemma \ref{lem:prep_d}~(2). Since $R$ is idempotent
and the left $R$-action on $W$ is surjective by Theorem \ref{thm:ff}, this
proves the claim.

(5) is immediate by (4) and Lemma \ref{lem:l-r-E}~(1).

\noindent
The remaining assertions are proven symmetrically.
\end{proof}

Since the module tensor product is a coequalizer (of linear maps), we study
next comodule structures on such coequalizers.

\begin{lemma}\label{lem:equalizer}
Consider a regular weak multiplier bialgebra $A$. For some right $A$-comodules
$(V,\lambda,\varrho)$ and $(V',\lambda',\varrho')$, comodule maps $f,g:V\to
V'$, a vector space $P$ and a linear map $\pi:V'\to P$, assume that 
\begin{equation}\label{eq:split_eq}
\xymatrix@C=35pt{
V  \ar@<2pt>[r]^-f\ar@<-2pt>[r]_-g&
V' \ar[r]^-\pi&
P}
\end{equation}
is a coequalizer of linear maps. Then $P$ carries a unique right
$A$-comodule structure such that $\pi$ is a comodule map.
\end{lemma}

\begin{proof}
Since $f$ and $g$ are comodule maps, the diagrams
$$
\xymatrix@C=30pt{
V\ox A \ar@<2pt>[r]^-{f \ox A}\ar@<-2pt>[r]_-{g \ox A}
\ar[d]_-{\lambda}&
V' \ox A \ar[d]^-{\lambda'} \ar[r]^-{\pi \ox A}&
P \ox A\ar@{-->}[d]^-{\lambda_P}&
V\ox A \ar@<2pt>[r]^-{f \ox A}\ar@<-2pt>[r]_-{g \ox A}
\ar[d]_-{\varrho}&
V' \ox A \ar[d]^-{\varrho'}\ar[r]^-{\pi \ox A}&
P \ox A\ar@{-->}[d]^-{\varrho_P} \\
V \ox A \ar@<2pt>[r]^-{f \ox A}\ar@<-2pt>[r]_-{g \ox A}&
V' \ox A \ar[r]_-{\pi \ox A}&
P\ox A &
V \ox A \ar@<2pt>[r]^-{f \ox A}\ar@<-2pt>[r]_-{g \ox A}&
V' \ox A \ar[r]_-{\pi \ox A}&
P\ox A}
$$
serially commute (meaning that they commute with either
simultaneous choice of the upper or lower ones of the parallel arrows). 
Since \eqref{eq:split_eq} is a coequalizer by assumption, 
also the top rows are coequalizers. Hence by their universality, there exist
unique morphisms $\lambda_P$ and $\varrho_P$ making the respective diagrams
commutative. Let us see that they define a comodule
$(P,\lambda_P,\varrho_P)$. For any $v'\in V'$ and $a,b\in A$, 
\begin{eqnarray*}
(1\ox a)\lambda_P(\pi(v') \ox b)&=&
(\pi\ox A)[(1\ox a)\lambda'(v'\ox b)]\\
&=&
(\pi\ox A)[\varrho'(v'\ox a)(1\ox b)]=
\varrho_P(\pi(v') \ox a)(1\ox b).
\end{eqnarray*}
In the second equality we used that \eqref{eq:comp} holds for $V'$. Since
$\pi\ox A$ is surjective, this proves that \eqref{eq:comp} holds for
$P$. Since also $\pi \ox A \ox A$ is surjective, the normalization condition
\eqref{eq:l-i-norm} on $P$ follows by 
\begin{eqnarray*}
(\lambda_P\ox A)\lambda_P^{13}(P\ox E_1)(\pi\ox A\ox A)&=&
(\pi\ox A\ox A)(\lambda'\ox A)\lambda^{\prime 13}(V'\ox E_1)\\
&=&(\pi\ox A\ox A)(\lambda'\ox A)\lambda^{\prime 13}\\
&=&(\lambda_P\ox A)\lambda_P^{13}(\pi\ox A\ox A),
\end{eqnarray*}
where in the second equality we used that \eqref{eq:l-i-norm} holds for $V'$. 
Similarly, the coassociativity condition \eqref{eq:l-coass} follows by
\begin{eqnarray*}
(\lambda_P\ox A)\lambda_P^{13}(P\ox T_1)(\pi\ox A\ox A)&=&
(\pi\ox A\ox A)(\lambda'\ox A)\lambda^{\prime 13}(V'\ox T_1)\\
&=&(\pi\ox A\ox A)(V'\ox T_1)(\lambda'\ox A)\\
&=&(P\ox T_1)(\lambda_P\ox A)(\pi\ox A\ox A),
\end{eqnarray*}
where in the second equality we used that \eqref{eq:l-coass} holds for $V'$.
\end{proof}

\begin{proposition}\label{prop:product}
Consider a regular weak multiplier bialgebra $A$ with right full
comultiplication. For any full right $A$-comodules $V$ and $W$, there is a
unique right $A$-comodule structure on $V\ox_R W$ such that the canonical
epimorphism $\pi_{V,W}:V\ox W \to V \ox_R W$ is a morphism of comodules from
the comodule in Lemma \ref{lem:k_product}.
\end{proposition}

\begin{proof}
By Proposition \ref{prop:ox_R}, $\pi_{V,W}$ is split by the $R$-bimodule map 
$$
V\ox_R W\cong \mathsf{Im}(G_1)\to V\ox W,\qquad
v\ox_R w \mapsto G_1(v\ox w).
$$
Since $G_1$ is idempotent,
$$
\xymatrix@C=35pt{
V\ox W  \ar@<2pt>[r]^-{G_1}\ar@<-2pt>[r]_-{V\ox W}&
V\ox W \ar[r]^-{\pi_{V,W}}&
V\ox_R W}
$$
is a (split) coequalizer of linear maps. Moreover, by Lemma
\ref{lem:l-r-Theta}~(3) and (5), $G_1$ is a comodule map $V\ox W\to V\ox W$ 
for the comodule $V\ox W$ in Lemma \ref{lem:k_product}.
Then we conclude by Lemma \ref{lem:equalizer}. 
\end{proof}

Our next aim is to prove that the product of full comodules in Proposition
\ref{prop:product} is a full comodule again. This starts with the following. 

\begin{lemma}\label{lem:l-r-eps}
Let $A$ be a regular weak multiplier bialgebra with right full
comultiplication. For any full right $A$-comodules $V$ and $W$, consider the
product $A$-comodule $(V\ox_R W,\lambda_{V\ox_R W},\varrho_{V\ox_R W})$ in
Proposition \ref{prop:product}. For all $v\ox_R w\in V\ox_R W$ and $a\in A$,
the following assertions hold. 
\begin{itemize}
\item[{(1)}] $(V\ox_R W \ox \epsilon)\varrho_{V\ox_R W}(v\ox_R w\ox a)=
\sqcap^R(a)\act v\ox_R w$.
\item[{(2)}] $(V\ox_R W \ox \epsilon)\lambda_{V\ox_R W}(v\ox_R w\ox a)=
v\ox_R w\act \overline \sqcap^R(a)$.
\end{itemize}
\end{lemma}

\begin{proof}
(1) The canonical epimorphism $\pi_{V,W}:V\ox W\to V\ox_R W$ is a
morphism of comodules by Proposition \ref{prop:product}. Hence for any $v\in V$,
$w\in W$, $a\in A$ and $v\ox_R w:=\pi_{V,W}(v\ox w)$, 
\begin{eqnarray*}
(V\ox_R W \ox \epsilon)\hspace{-1cm}&&
\varrho_{V\ox_R W}(v\ox_R w\ox a)\\
&=&
\pi_{V,W}(V\ox W \ox \epsilon)(V\ox \varrho_W)\varrho_V^{13}
(v\ox w\ox a)\\
&=&
v^\varrho \ox_R \sqcap^R(a^\varrho)\act w=
v^\varrho \act \sqcap^R(a^\varrho) \ox_R w=
\sqcap^R(a)\act v\ox_R w,
\end{eqnarray*}
where the implicit summation index notation $\varrho_V(v\ox a)=:v^\varrho \ox
a^\varrho$ is used and the last equality holds by Lemma
\ref{lem:l-r-Pi_ac}~(1). Part (2) is proven symmetrically.
\end{proof}

\begin{proposition}\label{prop:product_full}
Consider a regular weak multiplier bialgebra $A$ with right full
comultiplication. For any full right $A$-comodules $V$ and $W$, the product
$A$-comodule $(V\ox_R W,\lambda_{V\ox_R W},\varrho_{V\ox_R W})$ in
Proposition \ref{prop:product} is full.
\end{proposition}

\begin{proof}
Since $V$ is a full $A$-comodule, its left $R$-action is surjective by
Theorem \ref{thm:ff}. So we conclude by Lemma \ref{lem:l-r-eps}~(1) that
\begin{eqnarray*}
V\ox_R W &\subseteq&
\langle 
(V\ox_R W \ox \epsilon)\varrho_{V\ox_R W}(v\ox_R w\ox a)\ |\ 
v\ox_R w\in V\ox_R W,\ a\in A \rangle\\
&\subseteq&
\langle 
(V \raisebox{-6pt}{$\,\stackrel{\displaystyle\ox} {_R}\,$} W \ox \omega)
\varrho_{V\ox_R W}(
v \raisebox{-6pt}{$\,\stackrel{\displaystyle\ox} {_R}\,$}w \ox a)\ |\ 
v\raisebox{-6pt}{$\,\stackrel{\displaystyle\ox} {_R}\,$}w\in 
V \raisebox{-6pt}{$\,\stackrel{\displaystyle\ox} {_R}\,$} W,\ a\in A,\ 
\omega \in \mathsf{Lin}(A,k) \rangle\\
&\subseteq& V\ox_R W
\end{eqnarray*}
so that $V\ox_R W$ is a full $A$-comodule.
\end{proof}

It follows immediately from Lemma \ref{lem:l-r-eps} that the left and right
$R$-actions on the full $A$-comodule $V\ox_R W$ in Proposition
\ref{prop:product_full} are
$$
\sqcap^R(a)\act (v\ox_R w)=
\sqcap^R(a)\act v\ox_R w
\qquad \textrm{and}\qquad
(v\ox_R w)\act \overline \sqcap^R(a)=
v\ox_R w\act \overline \sqcap^R(a).
$$

\begin{remark}\label{rem:mono-epi}
Consider right comodules $(V,\lambda_V,\varrho_V)$, $(P,\lambda,\varrho)$
and $(P',\lambda',\varrho')$ over a regular weak multiplier bialgebra
$A$. Note that for a surjective homomorphism $\pi:V\to P$ of comodules and a 
linear map $f:P\to P'$, $f\pi$ is a morphism of comodules if and only if
$f$ is so. Indeed, the composite of comodule maps $\pi$ and $f$ is
evidently a morphism of comodules and the converse follows by
$$
\lambda'(f\ox A)(\pi \ox A)=
(f\ox A)(\pi \ox A)\lambda_V=
(f\ox A)\lambda(\pi \ox A)
$$
and the surjectivity of $\pi \ox A$.
\end{remark}

We are ready to prove the main result of this section.

\begin{theorem}\label{thm:monoidal}
For a regular weak multiplier bialgebra $A$ with right full
comultiplication, the category $M^{(A)}$ of full right $A$-comodules is
monoidal and the functor $M^{(A)}\to {}_R M_R$ in Theorem \ref{thm:ff} is
strict monoidal. 
\end{theorem}

\begin{proof}
The tensor product $f\ox g:V\ox W\to V'\ox W'$ of any $A$-comodule maps
$f:V\to V'$ and $g:W\to W'$ is evidently an $A$-comodule map between the
comodules as in Lemma \ref{lem:k_product}. Since $\pi_{V,W}$ and $\pi_{V',W'}$
are surjective comodule maps by Proposition \ref{prop:product}, also
$\pi_{V',W'}(f\ox g)=(f\ox_R g) \pi_{V,W}$ is a comodule map. Thus $f\ox_R
g:V\ox_R W\to V'\ox_R W'$ is a comodule map by Remark
\ref{rem:mono-epi}. 

It remains to show that the unitors and the associator of ${}_R M_R$ --- if
evaluated at objects of $M^{(A)}$ --- are morphisms of $A$-comodules. 
For any full right $A$-comodule $V$ the composition of the canonical
epimorphism $\pi_{V,R}:V\ox R \to V\ox_R R$ with the right unitor
$\mathbf{r}_V:V\ox_R R\to V$ yields the right $R$-action $\nu:V\ox R \to V$ on
$V$ in Theorem \ref{thm:ff}. Let us show that $\nu=\mathbf{r}_V\pi_{V,R}$ is a
morphism of $A$-comodules. This is proven by the following computation
for any $v\in V$ and $a,b,c\in A$. (We use again the implicit summation index
notation $T_1(b\ox c)=b^1\ox c^1$.)
\begin{eqnarray*}
(\nu \ox A)&&\hspace{-1cm}\lambda_V^{13}(V\ox \lambda_R)
(v\ox \overline \sqcap^R(a)\ox bc)\\
&=&(\nu \ox A)\lambda_V^{13}(v\ox (1\ox \overline\sqcap^R(a))E(1\ox bc))
=\lambda_V(v\ox \overline\sqcap^R(a)c^1)\act (\overline\sqcap^R(b^1)\ox 1)\\
&=&\lambda_V(v\ox \sqcap^L(b^1)\overline\sqcap^R(a)c^1)
=\lambda_V(v\ox \overline\sqcap^R(a)\sqcap^L(b^1)c^1)\\
&=&\lambda_V(v\ox \overline\sqcap^R(a)bc)
=\lambda_V(\nu\ox A)(v\ox \overline\sqcap^R(a)\ox bc).
\end{eqnarray*}
The second equality follows by identity (2.3) in \cite{BoGTLC:wmba}. The
third and the last equalities hold by parts (1) and (2) of Lemma
\ref{lem:l-r-bim}, respectively. In the fourth equality we used (3.9) in
\cite{BoGTLC:wmba} and in the fifth one we applied Lemma 3.7~(3) in
\cite{BoGTLC:wmba}. We know from Proposition \ref{prop:product} that
$\pi_{V,R}$ is a surjective morphism of $A$-comodules; hence 
$\mathbf{r}_V$ is a comodule map by Remark \ref{rem:mono-epi}. The left
unitor is treated symmetrically.

For any full $A$-comodules $V,W$ and $Z$, the associator
$\mathbf{a}_{V,W,Z}:(V\ox_R W)\ox_R Z\to V\ox_R (W\ox_R Z)$ satisfies by
construction 
\begin{equation}\label{eq:a_pi}
\mathbf{a}_{V,W,Z}\pi_{V\ox_R W,Z}  (\pi_{V,W}\ox Z)=
\pi_{V,W\ox_R Z} (V\ox \pi_{W,Z})
\end{equation}
(where we omit explicitly denoting the associator in the monoidal category of
vector spaces). 
By Proposition \ref{prop:product} both $\pi_{V\ox_R W,Z}  (\pi_{V,W}\ox Z)$
and $\pi_{V,W\ox_R Z} (V\ox \pi_{W,Z})$ are surjective morphisms of
$A$-comodules. Hence $\mathbf{a}_{V,W,Z}$ is a morphism of $A$-comodules by
Remark \ref{rem:mono-epi}. 
\end{proof}

\section{Antipode and dual comodules}
\label{sec:dual}

In this section we deal with a regular weak multiplier bialgebra with a left
and right full comultiplication. Assuming that it possesses an antipode (in
the sense of \cite[Theorem 6.8]{BoGTLC:wmba}), we show that finite dimensional
full right $A$-comodules possess duals in $M^{(A)}$.

To begin with, let $A$ be a regular weak multiplier bialgebra over a field
$k$. As a first step, we equip the $k$-dual $V^*:=\mathsf{Lin}(V,k)$ of any
finite dimensional right $A$-comodule $V$ with the structure of a left
$A$-comodule.

\begin{proposition}\label{prop:V^*_left}
Let $A$ be a regular weak multiplier bialgebra over a field $k$ and let
$(V,\lambda,\varrho)$ be a finite dimensional right $A$-comodule. Using a dual
basis basis $\{a_i\}$ in $A$ and $\{\alpha_i\}$ in $A^*:=\mathsf{Lin}(A,k)$,
the dual vector space $V^*:=\mathsf{Lin}(V,k)$ carries a left $A$-comodule
structure 
\begin{eqnarray*}
&\lambda^*:A\ox V^* \to A\ox V^*,\qquad
&b\ox \varphi\mapsto \sum_i a_i\ox (\varphi\ox \alpha_i)\lambda(-\ox b)\\
&\varrho^*:A\ox V^* \to A\ox V^*,\qquad
&b\ox \varphi\mapsto \sum_i a_i\ox 
(\varphi\ox \alpha_i)\varrho(-\ox b).
\end{eqnarray*}
\end{proposition}

\begin{proof}
For any given elements $v\in V$ and $b\in A$, there are only finitely many
values of the index $i$ such that $(V\ox \alpha_i)\lambda(v\ox b)\neq 0$. 
Thus since $V$ is a finite dimensional vector space, for any given element
$b\in A$ there are only finitely many indices $i$ such that $(V\ox
\alpha_i)\lambda(-\ox b)$ is a non-zero map $V\to V$. Symmetrically, there
are only finitely many indices $i$ such that $(V\ox \alpha_i)\varrho(-\ox
b)\neq 0$. This proves that the sums defining the maps $\lambda^*$ and
$\varrho^*$ contain only finitely many non-zero terms; hence they are
meaningful. 

Introduce the implicit summation index notation $\lambda(v\ox a)=v^\lambda\ox
a^\lambda$, $\varrho(v\ox a)=v^\varrho\ox a^\varrho$, $\lambda^*(a\ox
\varphi)= a^{\lambda^*}\ox \varphi^{\lambda^*}$ and $\varrho^*(a\ox
\varphi)= a^{\varrho^*}\ox \varphi^{\varrho^*}$ for any $a\in A$, $v\in V$ and
$\varphi\in V^*$. By construction,
\begin{equation}\label{eq:dual_coaction}
a^{\lambda^*}\varphi^{\lambda^*}(v)=
\varphi(v^\lambda)a^\lambda
\qquad \textrm{and}\qquad 
a^{\varrho^*}\varphi^{\varrho^*}(v)=
\varphi(v^\varrho)a^\varrho,
\end{equation}
hence 
$$
ab^{\lambda^*}\varphi^{\lambda^*}(v)=
\varphi(v^\lambda)ab^\lambda=
 \varphi(v^\varrho)a^\varrho b=
a^{\varrho^*}\varphi^{\varrho^*}(v)b.
$$
In the second equality we used the compatibility condition \eqref{eq:comp}
between $\lambda$ and $\varrho$. This proves the compatibility condition
\eqref{eq:left_comp}, that is, 
$$
(a\ox 1)\lambda^*(b\ox \varphi)=
\varrho^*(a\ox \varphi)(b\ox 1)\qquad
\forall a,b\in A,\ \varphi\in V^*.
$$
In view of \eqref{eq:dual_coaction}, the normalization condition
$$
(A\ox \lambda^*)\lambda^{*13}(E_1\ox V^*)=(A\ox \lambda^*)\lambda^{*13}
$$
on $\lambda^*$ (see \eqref{eq:left_norm}) is immediate from the normalization
condition \eqref{eq:l-i-norm} on $\lambda$ and the coassociativity condition
$$
(A\ox \lambda^*)\lambda^{*13}(T_1\ox V^*)=(T_1\ox V^*)\lambda^{*13}
$$
on $\lambda^*$ (see \eqref{eq:left_coass}) is immediate from the
coassociativity condition \eqref{eq:l-coass} on $\lambda$.
\end{proof}

\begin{proposition}\label{prop:V^*_full}
Let $A$ be a regular weak multiplier bialgebra with a right full
comultiplication. For a finite dimensional full right $A$-comodule $V$, the 
(obviously finite dimensional) left $A$-comodule $V^*$ in Proposition
\ref{prop:V^*_left} is also full. 
\end{proposition}

\begin{proof}
By Theorem \ref{thm:ff}, $V$ is a firm right $R$-module so in particular the
action $V\ox R\to V$ is surjective. By \cite[Theorem 4.6~(2)]{BoGTLC:wmba} $R$
has local units; so for any finite set of elements in $R$ there is a common
right unit. Thus by the finite dimensionality of $V$ there is an element $e\in
A$ such that $v\act \overline\sqcap^R(e)=v$ for all $v\in V$. Then for any
$v\in V$ and $\varphi\in V^*$, 
$$
\epsilon(e^{\lambda^*})\varphi^{\lambda^*}(v)\stackrel{\eqref{eq:dual_coaction}}=
(\varphi \ox \epsilon)\lambda(v\ox e)=
\varphi(v\act \overline \sqcap^R(e))=
\varphi(v),
$$
so that $(\epsilon\ox V^*)\lambda^*(e\ox \varphi)=\varphi$. This proves 
$$
V^*\subseteq 
\langle (\epsilon\ox V^*)\lambda^*(e\ox \varphi) | \varphi\in V^* \rangle
\subseteq 
\langle (\omega \ox V^*)\lambda^*(a\ox \varphi) | 
\varphi\in V^*,\ \omega\in A^*,\ a\in A \rangle
\subseteq 
V^*
$$
hence the fullness of $V^*$.
\end{proof}

The notion of {\em antipode} $S:A\to \M(A)$ on a regular weak multiplier
bialgebra $A$ was introduced in \cite[Section 6]{BoGTLC:wmba}, see Section
\ref{sec:antipode}. In what follows, we study the bearing of the existence of
an antipode on the relation between left and right comodules. 

\begin{theorem}\label{thm:S_functor}
Let $A$ be a regular weak multiplier bialgebra with a left and right full
comultiplication, possessing an antipode $S$. Any left $A$-comodule $V$ (with
structure maps $\lambda,\varrho:A\ox V \to A\ox V$) is a right $A$-comodule as
well with the structure maps 
$$
\begin{array}{llll}
\lambda^S:
&V\ox A \to V\ox A,\quad
&v\ox S(b)a
&\mapsto 
((V\ox S)\varrho^{21}(v\ox b))(1\ox a)\\
\varrho^S:
&V\ox A \to V\ox A,\quad
&v\ox aS(b)
&\mapsto 
(1\ox a)((V\ox S)\lambda^{21}(v\ox b)).
\end{array}
$$
This is the object map of a functor --- acting on the morphisms as the
identity map --- from the category of left $A$-comodules to the category of
right $A$-comodules.
\end{theorem}

\begin{proof}
We need to show first that $\lambda^S$ and $\varrho^S$ are well-defined linear
maps. By \cite[Proposition 6.13]{BoGTLC:wmba}, any element of $A$ is a linear
combination of elements of the form $S(b)a$ or,
alternatively, a linear combination of elements of the form $aS(b)$. 
So $\lambda^S$ and $\varrho^S$ can be defined by giving their values on the
elements above. For any $a,b,c\in A$ and $v\in V$,
\begin{eqnarray}\label{eq:S_1}
((V\ox S)\lambda^{21}(v \ox c))(1\ox S(b)a)&=&
((V\ox S)\mathsf{tw}((b\ox 1)\lambda(c\ox v))) (1\ox a)\\
&\stackrel{\eqref{eq:left_comp}}=&
((V\ox S)\mathsf{tw}(\varrho(b\ox v)(c\ox 1))) (1\ox a)\nonumber\\
&=&(1 \ox S(c))((V\ox S)\varrho^{21}(v\ox b))(1\ox a).
\nonumber
\end{eqnarray}
The first and the last equalities follow by the anti-multiplicativity of $S$,
see \cite[Theorem 6.12]{BoGTLC:wmba}. So if for some elements $v^i\in V$ and
$a^i,b^i\in A$ the (finite) sum $\sum_i v^i\ox S(b^i)a^i$ is equal to zero,
then we conclude from \eqref{eq:S_1} by Lemma \ref{lem:S_nd} that
$\sum_i((V\ox S)\varrho^{21}(v^i\ox b^i))(1\ox a^i)=0$; hence
$\lambda^S$ is a well-defined linear map. A symmetric reasoning applies to
$\varrho^S$.

Multiplying both sides of \eqref{eq:S_1} on the left by $1\ox d$ for any $d\in
A$, we obtain
$$
\varrho^S(v \ox dS(c))(1\ox S(b)a)=
(1 \ox dS(c))\lambda^S(v\ox S(b)a).
$$ 
In light of \cite[Proposition 6.13]{BoGTLC:wmba}, this proves the
compatibility condition \eqref{eq:comp} between $\lambda^S$ and $\varrho^S$. 

By Proposition \ref{prop:prep_comod}, $(V,\lambda^S,\varrho^S)$ is a right
$A$-comodule provided that it satisfies the conditions in Proposition
\ref{prop:prep_comod}~(4.f). First we prove that the identity 
in Proposition \ref{prop:prep_comod}~(2.a) holds. For any $v\in V$ and
$a,b,c,d\in A$,
\begin{eqnarray*}
(V\ox E_1)(\lambda^S\!\!\!\!&\!\!\!\!\ox\!\!\!\!&\!\!\!\! A)\lambda^{S13}
(v\ox S(b)a\ox S(d)c)\\
&=&v^{\varrho\varrho'}\ox E(S(b^{\varrho'})a\ox S(d^\varrho)c)\\
&=&v^{\varrho\varrho'}\ox 
(S\ox S)((b^{\varrho'}\ox d^\varrho)E^{21})(a\ox c)\\
&=&v^{\varrho\varrho'}\ox S(b^{\varrho'})a\ox S(d^{\varrho})c\\
&=&(\lambda^S\ox A)\lambda^{S13}(v\ox S(b)a\ox S(d)c).
\end{eqnarray*}
where we used the implicit summation index notation 
$\varrho(b\ox v)=b^\varrho\ox v^\varrho=b^{\varrho'}\ox v^{\varrho'}$.
In the second equality we used that 
by \cite[Theorem 6.12 and Proposition 6.15]{BoGTLC:wmba},
$E_1(S \ox S)=(S\ox S)E_2^{21}$.
The third equality holds since
$(E_2\ox V)(A\ox \varrho)\varrho^{13}=(A\ox \varrho)\varrho^{13}$ by a
symmetric form of Proposition \ref{prop:prep_comod} on left $A$-comodules. By
\cite[Proposition 6.13]{BoGTLC:wmba}, this proves the normalization condition
in the form $(V\ox E_1)(\lambda^S\ox A)\lambda^{S13}=(\lambda^S\ox
A)\lambda^{S13}$. 

It remains to prove the coassociativity condition 
\eqref{eq:coass-r-prep-f}. For any $v\in V$ and $a,b,c,d\in A$, introduce the
implicit summation index notation $\varrho(d\ox v)=d^\varrho\ox v^\varrho$,
$\lambda(b\ox v)=b^\lambda\ox v^\lambda$ and $T_2(d\ox b)=d^2\ox b^2$.
Then
\begin{eqnarray*}
(\varrho^S\ox A)(V\ox T_1)\lambda^{S13}(v\ox aS(b)\ox S(d)c)\!\!&=&\!\!
(\varrho^S\ox A)(V\ox T_1)(v^\varrho\ox aS(b)\ox S(d^\varrho)c)\\
\!\!&=&\!\!(\varrho^S\ox A)(v^\varrho\ox 
T_1(a\ox S(d^{\varrho 2})c)(S(b^2)\ox 1))\\
\!\!&=&\!\!v^{\varrho\lambda}\ox 
T_1(a\ox S(d^{\varrho 2})c)(S(b^{2\lambda})\ox 1)\\
\!\!&=&\!\!v^\lambda\ox 
T_1(a\ox S(d^2)c)(S(b^{\lambda 2})\ox 1)\\
\!\!&=&\!\!(V\ox T_1)(v^\lambda\ox aS(b^\lambda)\ox S(d)c)\\
\!\!&=&\!\!(V\ox T_1)(\varrho^S\ox A)(v\ox aS(b)\ox S(d)c).
\end{eqnarray*}
From this we conclude by \cite[Proposition 6.13]{BoGTLC:wmba} that the
coassociativity condition \eqref{eq:coass-r-prep-f} holds. The second 
and the penultimate equalities hold by Lemma \ref{lem:T_1_S} and the fourth
equality holds since
$(A\ox \lambda)(T_2\ox V)\varrho^{13}=(T_2\ox V)(A\ox \lambda)$ by a
symmetric form of Proposition \ref{prop:prep_comod} on left $A$-comodules.
 
Left $A$-comodule maps are evidently maps between the induced right
$A$-comodules above. 
\end{proof}

Clearly, there is a symmetric functor from the category of right comodules to
the category of left comodules: It takes a right $A$-comodule
$(V,\lambda,\varrho)$ to the left $A$-comodule $(V,\lambda^S,\varrho^S)$,
where for any $v\in V$ and $a,b\in A$, 
\begin{eqnarray}\label{eq:rightS}
&&\lambda^S(S(b)a\ox v)=((S\ox V)\varrho^{21}(b\ox v))(a\ox 1)\\
&&\varrho^S(aS(b)\ox v)=(a\ox 1)((S\ox V)\lambda^{21}(b\ox v)).\nonumber
\end{eqnarray}
It acts on the morphisms as the identity map.

Let $A$ be a regular weak multiplier bialgebra with left and right full
comultiplication and base algebra $R$. Consider the linear maps 
\begin{eqnarray*}
&E_1:V\ox A\to V\ox A,\qquad
&v\ox a\mapsto ((-)v\ox A)[E(1\ox a)]\equiv E(v\ox a)\\
&G_1:V\ox A\to V\ox A,\qquad
&v\ox a\mapsto (v(-)\ox A)[F(1\ox a)]\equiv (v\ox 1)F(1\ox a)
\end{eqnarray*}
as before (they are meaningful since for any $a\in A$, both $E(1\ox a)$ and
$F(1\ox a)$ are elements of $R\ox A$, see (2.3) and Proposition 4.3~(1) in
\cite{BoGTLC:wmba}).

\begin{lemma}\label{lem:lambda_tilde}
Let $A$ be a regular weak multiplier bialgebra with left and right full
comultiplication possessing an antipode $S$. For any right $A$-comodule
$(V,\lambda,\varrho)$, the linear map $\lambda^S:V\ox A \to V\ox A$ in
\eqref{eq:rightS} obeys the following identities. 
\begin{itemize}
\item[{(1)}] $\lambda^{S21} E_1=\lambda^{S21}$.
\item[{(2)}] $\lambda \lambda^{S21}=E_1$.
\item[{(3)}] $\lambda^{S21}\lambda=G_1$.
\end{itemize}
\end{lemma}

\begin{proof}
From \eqref{eq:rightS} we have 
\begin{equation}\label{eq:lambda_tilde}
\lambda^{S21}:V\ox A \to V\ox A,\qquad
v\ox S(b)a\mapsto 
((V\ox S)\varrho(v\ox b))(1\ox a).
\end{equation}
Since $(V,\lambda^S,\varrho^S)$ in \eqref{eq:rightS} is a left $A$-comodule,
it follows by the compatibility axiom \eqref{eq:left_comp} that for any $v\in
V$ and $a,b,c\in A$,
$$
(1\ox bS(c))\lambda^{S21}(v\ox a)=
\varrho^{S21}(v\ox bS(c)))(1 \ox a)=
(1\ox b)((V\ox S)\lambda(v\ox c))(1\ox a).
$$
Thus simplifying by $b$,
\begin{equation}\label{eq:lt_1}
(1\ox S(c))\lambda^{S21}(v\ox a)=
((V\ox S)\lambda(v\ox c))(1\ox a)\qquad 
\forall v\in V,\ a,c\in A.
\end{equation}

(1) Using \eqref{eq:lt_1} in the first and the last equalities, Lemma
\ref{lem:l-r-E}~(2) in the second equality, and an identity in (6.14) of
\cite{BoGTLC:wmba} (cf. \eqref{eq:S_id}) in the third one, we obtain
\begin{eqnarray*}
(1\ox S(c))(\lambda^{S21} E_1(v\ox a))&=&
(V\ox \mu(S\ox A))(\lambda\ox A)E_1^{13}(v\ox c\ox a)\\
&=&(V\ox \mu(S\ox A)E_1)(\lambda\ox A)(v\ox c\ox a)\\
&=&(V\ox \mu(S\ox A))(\lambda\ox A)(v\ox c\ox a)\\
&=&(1\ox S(c))\lambda^{S21} (v\ox a).
\end{eqnarray*}
In light of Lemma \ref{lem:S_nd}, this proves assertion (1).

(2) For any $v\in V$ and $a,b,c\in A$,
\begin{eqnarray*}
(1\ox c)\!\!\!\!\!\!\!&&\!\!\!\!\!\!\!\!
(\lambda\lambda^{S21}(v\ox S(b)a))\\
&\stackrel{\eqref{eq:lambda_tilde}}=&
(1\ox c)\lambda((V\ox S)\varrho(v\ox b)(1\ox a))\\
&=&((V\ox \mu(A\ox S))(\varrho\ox A)\varrho^{13}(v\ox c\ox b))(1\ox a)\\
&=&((V\ox \mu(A\ox S)T_2)\varrho^{13}(V\ox R_2)(v\ox c\ox b))(1\ox a)\\
&=&((V\ox \epsilon \ox A)(\varrho\ox A)(V\ox T_4 \mathsf{tw}R_2)
(v\ox c\ox b))(1\ox a)\\
&=&(V\ox \epsilon \ox A)(\varrho\ox A)(v\ox T_1(c\ox S(b)a))\\
&=&((\sqcap^R\ox A)T_1(c\ox S(b)a))(v\ox 1)\\
&=&(1\ox c)E(v\ox S(b)a).
\end{eqnarray*}
From this we conclude simplifying by $c$ and using that $A$ is spanned by
elements of the form $S(b)a$, cf. \cite[Proposition 6.13]{BoGTLC:wmba}.
The second equality follows by \eqref{eq:comp} and the third equality follows
by 
$$
(\varrho\ox A)\varrho^{13}\stackrel{\eqref{eq:r-i-norm}}=
(\varrho\ox A)\varrho^{13}(V\ox E_2)=
(\varrho\ox A)\varrho^{13}(V\ox T_2R_2)
\stackrel{\eqref{eq:coass-r-prep-e}}=
(V\ox T_2)\varrho^{13}(V\ox R_2).
$$
In the fourth equality we used that by an identity in (6.14) in
\cite{BoGTLC:wmba} (cf. \eqref{eq:S_id}); by (3.3) in \cite{BoGTLC:wmba}
(cf. \eqref{eq:Pi}); and by \eqref{eq:T4-rho}, 
\begin{eqnarray*}
(V\ox \mu(A\ox S)T_2)\varrho^{13}&=&
(V\ox \mu(A\ox \sqcap^L))\varrho^{13}\\
&=&(V\ox \epsilon \ox A)(V\ox T_4\mathsf{tw})\varrho^{13}\\
&=&(V\ox \epsilon \ox A)(\varrho\ox A)(V\ox T_4 \mathsf{tw}).
\end{eqnarray*}
The fifth equality is a consequence of Lemma \ref{lem:T_tw_R}~(2) and the
last equality holds since by the first axiom in (ix) in Definition
\ref{def:wmba} and by (3.4) in \cite{BoGTLC:wmba},
$$
(1\ox d)((\sqcap^R\ox A)T_1(c\ox a))=
((\sqcap^R \ox A)T_3(c\ox d))(1\ox a)=
(1\ox dc)E(1\ox a),
$$
for all $a,c,d\in A$, so that $(\sqcap^R\ox A)T_1(c\ox a)=(1\ox c)E(1\ox a)$.

(3) The proof is similar to part (2). For any $v\in V$ and $a,b,c\in A$,
\begin{eqnarray*}
(1\ox bS(c))\hspace{-.5cm}&&\hspace{-.5cm}
(\lambda^{S21}\lambda(v\ox a))\\
&=&(1\ox b)((V\ox \mu(S\ox A))(\lambda\ox A)\lambda^{13}(v\ox c \ox a))\\
&=&(1 \ox b)((V\ox \mu(S\ox A)T_1)(\lambda\ox A)(V\ox R_1)(v\ox c \ox a))\\
&=&(1 \ox b)((V\ox A\ox \epsilon)(V \ox T_3)\lambda^{13}(V\ox \mathsf{tw}R_1)
(v\ox c \ox a))\\
&=&(1 \ox b)((V\ox A\ox \epsilon)\lambda^{13}(V\ox T_3\mathsf{tw}R_1)
(v\ox c \ox a))\\
&=&(V\ox A\ox \epsilon)\lambda^{13}(v\ox T_2(bS(c)\ox a))\\
&=&(v\ox 1)((\overline \sqcap^R\ox A)\mathsf{tw}T_2(bS(c)\ox a))\\
&=&(v\ox bS(c))F(1\ox a).
\end{eqnarray*}
From this we conclude simplifying by $bS(c)$ (cf. \cite[Proposition
6.13]{BoGTLC:wmba}). 
The first equality follows by \eqref{eq:lt_1}.
The second one holds since 
$$
(\lambda\ox A)\lambda^{13}\stackrel{\eqref{eq:l-i-norm}}=
(\lambda\ox A)\lambda^{13}(V\ox E_1)=
(\lambda\ox A)\lambda^{13}(V\ox T_1R_1)\stackrel{\eqref{eq:l-coass}}=
(V\ox T_1)(\lambda\ox A)(V\ox R_1). 
$$
In the third equality we used that by an identity in (6.14) in
\cite{BoGTLC:wmba} (cf. \eqref{eq:S_id}) and by (3.3) in \cite{BoGTLC:wmba}
(cf. \eqref{eq:Pi}),
$$
V\ox \mu(S\ox A)T_1=
V\ox \mu(\sqcap^R \ox A)=
V\ox (A\ox \epsilon)T_3 \mathsf{tw}.
$$
In the fourth equality we used that since $\lambda$ is a right $A$-module map
(cf. Proposition \ref{prop:prep_comod}~(1)) and $T_3$ is a left $A$-module map,
it follows for any $v\in V$ and $a,b,c\in A$ that
$$
(V\ox T_3)\lambda^{13}(v\ox a\ox bc)=
\lambda(v \ox b)^{13}(1\ox T_3(a\ox c))=
\lambda^{13}(V\ox T_3)(v\ox a\ox bc).
$$
Hence by the idempotency of $A$, $(V\ox T_3)\lambda^{13}=\lambda^{13}(V\ox
T_3)$. 
The fifth equality follows by Lemma \ref{lem:T_tw_R}~(1) and the last one
holds since by the second axiom in (ix) in Definition \ref{def:wmba} and by
\cite[Proposition 4.3~(1)]{BoGTLC:wmba} (cf. \eqref{eq:F}),
\begin{eqnarray*}
((\overline \sqcap^R\ox A)\mathsf{tw}T_2(b\ox a))(1\ox c)&=&
(\overline \sqcap^R\ox A)\mathsf{tw}(T_2(b\ox a)(c\ox 1))\\
&=&(\overline \sqcap^R\ox A)\mathsf{tw}((b\ox 1)T_4(c\ox a))\\
&=&(1\ox b)((\overline \sqcap^R\ox A)\mathsf{tw}T_4(c\ox a))=
(1\ox b)F(1\ox ac),
\end{eqnarray*}
for all $a,b,c\in A$; hence $(\overline \sqcap^R\ox A)\mathsf{tw}T_2(b\ox
a)=(1\ox b)F(1\ox a)$. 
\end{proof}

\begin{proposition}\label{prop:left_right_full}
Let $A$ be a regular weak multiplier bialgebra with left and right full
comultiplication, possessing an antipode $S$. For a full left $A$-comodule
$(V,\lambda,\varrho)$, also the right $A$-comodule $(V,\lambda^S,\varrho^S)$
in Theorem \ref{thm:S_functor} is full. 
\end{proposition}

\begin{proof}
By Lemma 6.11, (3.5) and Lemma 6.10 in \cite{BoGTLC:wmba},
\begin{equation}\label{eq:full_1}
\epsilon(aS(b))=
\epsilon\mu (A\ox \sqcap^L)R_2(a\ox b)=
\epsilon\mu R_2(a\ox b)=
\epsilon(a\sqcap^R(b)),\qquad \forall a,b\in A. 
\end{equation}
By Lemma 3.12 and (3.8) in \cite{BoGTLC:wmba},
$\sqcap^R(b\,\overline\sqcap^L(a))=\sqcap^R(a)\sqcap^R(b)=
\sqcap^R(a\sqcap^R(b))$ for any $a,b\in A$. Applying to this identity the
counit $\varepsilon$ of $R$ and using \cite[Proposition 4.1]{BoGTLC:wmba}, we
conclude that $\epsilon(b\overline\sqcap^L(a))=\epsilon(a\sqcap^R(b))$. 
Combining this with \eqref{eq:full_1}, we obtain
\begin{equation}\label{eq:counit_S}
\epsilon(aS(b))=\epsilon(b\overline\sqcap^L(a)),\qquad \forall a,b\in A. 
\end{equation}
Using this identity in the second equality, it follows for any $v\in V$ and
$a,b\in A$ that 
\begin{equation}\label{eq:full_2}
(V\ox \epsilon)\varrho^S(v\ox aS(b))=
\epsilon(aS(b^{\lambda}))v^\lambda=
\epsilon(b^{\lambda}\overline\sqcap^L(a))v^\lambda
=(\epsilon\ox V)\lambda(b\overline\sqcap^L(a)\ox v),
\end{equation}
where the implicit summation index notation $\lambda(b\ox v)=b^\lambda\ox
v^\lambda$ is used. In the third equality we used that $\lambda$ is a morphism
of right $\M(A)$-modules by Proposition \ref{prop:prep_comod}~(1) and the
idempotency of $A$. The last expression in \eqref{eq:full_2} is the right
action of $\sqcap^L(b\overline\sqcap^L(a))$ on $v$ obtained from the left
$A$-comodule structure of $V$ (cf. a symmetric counterpart of Theorem
\ref{thm:ff} on left comodules). Hence by Lemma \ref{lem:Pi_spanned}~(1) and a
symmetric counterpart of Theorem \ref{thm:ff} on left comodules, $V$ is
spanned by elements of the form in \eqref{eq:full_2}. Taking into account
\cite[Proposition 6.13]{BoGTLC:wmba}, this proves that
$(V,\lambda^S,\varrho^S)$ is a full right $A$-comodule. 
\end{proof}

Let $A$ be a regular weak multiplier bialgebra with left and right full
comultiplication, possessing an antipode $S$. Let $V$ be a finite dimensional
full right $A$-comodule. Combining the constructions in Proposition
\ref{prop:V^*_left} and in Theorem \ref{thm:S_functor}, we obtain a right
$A$-comodule structure on $V^*$ with the structure maps 
\begin{equation}\label{eq:V^*_right}
\begin{array}{llll}
\lambda^{*S}:
&V^*\ox A \to V^*\ox A,\quad
&\varphi \ox S(b)a
&\mapsto \ 
[v\mapsto \varphi(v^\varrho)]\ox S(b^{\varrho})a\\
\varrho^{*S}:
&V^*\ox A \to V^*\ox A,\quad 
&\varphi\ox aS(b)
&\mapsto \ 
[v\mapsto \varphi(v^\lambda)]\ox aS(b^{\lambda}),
\end{array}
\end{equation}
where the implicit summation index notation $\varrho(v\ox b)=v^\varrho\ox
b^\varrho$ and $\lambda(v\ox b)=v^\lambda\ox b^\lambda$ is used. This right
$A$-comodule $V^*$ is full by Proposition \ref{prop:V^*_full} and Proposition
\ref{prop:left_right_full}. 

\begin{proposition}\label{prop:V^*-R-actions}
Let $A$ be a regular weak multiplier bialgebra with left and right full
comultiplication, possessing an antipode $S$. Let $(V,\lambda,\varrho)$ be a
finite dimensional full right $A$-comodule and consider the $A$-comodule
$(V^*,\lambda^{*S},\varrho^{*S})$ in \eqref{eq:V^*_right}. In terms of the
Nakayama automorphism $\vartheta$ of $R$ (cf. \cite[Theorem
4.6~(3)]{BoGTLC:wmba}), the corresponding $R$-actions in Theorem \ref{thm:ff}
on $V^*$ come out as 
$$
r\act \varphi=\varphi(-\,\vartheta^{-1}(r))
\qquad \textrm{and}\qquad
\varphi\act r=\varphi(r\,-).
$$
In particular, the isomorphism $V^*\cong \mathsf{Hom}_R(V,R)$ in Proposition
\ref{prop:R/k-dual} is an isomorphism of $R$-bimodules.
\end{proposition}

\begin{proof}
By \cite[Proposition 6.13]{BoGTLC:wmba}, $R$ is spanned by elements of the form
$\overline \sqcap^R(S(b)a)$, for $a,b\in A$. So we only compute its
right action on an arbitrary element $\varphi\in V^*$. 
\begin{eqnarray*} 
\varphi\act \overline \sqcap^R(S(b)a)&=&
(V^*\ox \epsilon)\lambda^{*S}(\varphi \ox S(b)a)=
\varphi(-^\varrho) \epsilon(S(b^{\varrho})a)\\
&=&\varphi(-^\varrho) \epsilon(\overline \sqcap^R(a)b^{\varrho})=
(\varphi\ox \epsilon)\varrho(-\ox \overline \sqcap^R(a)b)\\
&=&\varphi(\sqcap^R(\overline \sqcap^R(a)b)\act -)=
\varphi(\overline \sqcap^R(S(b)a)\act -).
\end{eqnarray*}
The fourth equality holds because $\varrho$ is a left $A$-module map (see
Proposition \ref{prop:prep_comod}~(1)), hence by the idempotency of $A$ it is
also a left $R$-module map. In the last equality we used that by Lemma 6.11,
Lemma 3.11, and identity (2.2) in \cite{BoGTLC:wmba} and by the equality
$T_1R_1=E_1$, 
\begin{eqnarray*}
\overline \sqcap^R(S(b)a)&=&
\overline \sqcap^R\mu(\sqcap^R \ox A)R_1(b\ox a)=
\sqcap^R\mu(A\ox \overline \sqcap^R)R_1(b\ox a)\\
&=&\sqcap^R(A\ox \epsilon)T_1R_1(b\ox a)=
\sqcap^R(A\ox \epsilon)E_1(b\ox a)=
\sqcap^R(\overline \sqcap^R(a)b)
\end{eqnarray*}
for any $a,b\in A$. Applying the counit $\varepsilon$ of $R$ to both sides, by
\cite[Proposition 4.1]{BoGTLC:wmba} and \cite[Lemma 4.5~(3)]{BoGTLC:wmba} we
obtain $\epsilon(S(b)a)=\epsilon(\overline \sqcap^R(a)b)$ justifying the third
equality. 
For $\Phi\in \mathsf{Hom}_R(V,R)$, $r\in R$, and the counit $\varepsilon$ on
$R$, $\varepsilon(\Phi\act r)=\varepsilon\Phi(r\act -)$ so the isomorphism
in Proposition \ref{prop:R/k-dual} is an isomorphism of right $R$-modules.

Symmetrically, by \cite[Proposition 6.13]{BoGTLC:wmba}, $R$ is spanned by
elements of the form $\sqcap^R(aS(b))$, for $a,b\in A$. So we only
compute its left action on an arbitrary element $\varphi\in V^*$. 
\begin{eqnarray*} 
\sqcap^R(aS(b))\act \varphi&=&
(V^*\ox \epsilon)\varrho^{*S}(\varphi \ox aS(b))=
\varphi(-^\lambda)\epsilon(aS(b^{\lambda}))\\
&=&\varphi(-^\lambda)\epsilon(b^{\lambda}\overline\sqcap^L(a))=
(\varphi\ox \epsilon)\lambda(-\ox b\overline\sqcap^L(a))\\
&=&\varphi(-\act \overline\sqcap^R(b\overline\sqcap^L(a)))=
\varphi(-\act \vartheta^{-1}\sqcap^R(aS(b))).
\end{eqnarray*} 
In the third equality we applied \eqref{eq:counit_S} and in the last equality
we used that by Proposition 4.9, Lemma 6.11, Lemma 3.11 and identity (2.1) in
\cite{BoGTLC:wmba} and by the equality $T_2R_2=E_2$,
\begin{eqnarray*}
\vartheta^{-1} \sqcap^R(aS(b))&=&
\overline\sigma\,\overline\sqcap^L\mu(A\ox \sqcap^L)R_2(a\ox b)=
\overline\sqcap^R\mu(\overline\sqcap^L\ox A)R_2(a\ox b)\\
&=&\overline\sqcap^R(\epsilon\ox A)T_2R_2(a\ox b)=
\overline\sqcap^R(\epsilon\ox A)E_2(a\ox b)=
\overline\sqcap^R(b\overline\sqcap^L(a)).
\end{eqnarray*}
For $\Phi\in \mathsf{Hom}_R(V,R)$, $r\in R$, and the counit $\varepsilon$ on
$R$, 
$$
\varepsilon(r\act\Phi)=
\varepsilon(r\Phi(-))=
\varepsilon(\Phi(-)\vartheta^{-1}(r))=
\varepsilon\Phi(-\act \vartheta^{-1}(r))
$$ 
so the isomorphism in Proposition \ref{prop:R/k-dual} is an isomorphism of
left $R$-modules. 
\end{proof}

We are ready to prove the main result of the section:

\begin{theorem}\label{thm:dual}
Let $A$ be a regular weak multiplier bialgebra over a field $k$ with left
and right full comultiplication, possessing an antipode $S$. Then for any
finite dimensional full right $A$-comodule $V$, the right $A$-comodule
$V^*:=\mathsf{Lin}(V,k)$ in \eqref{eq:V^*_right} is the dual of $V$ in the
monoidal category $M^{(A)}$ of full right $A$-comodules in Theorem
\ref{thm:monoidal}. 
\end{theorem}

\begin{proof}
Let us choose a finite $k$-basis $\{v_i\}$ in $V$ with dual
basis $\{\varphi^i\}$ in $V^*$. The isomorphism in Proposition
\ref{prop:R/k-dual} takes $\{\varphi^i\}$ to the set $\{\Phi^i\in
\mathsf{Hom}_R(V,R)$\}, satisfying 
$$
\sum_i v_i \act \Phi^i(w\act r)=
\sum_i (\varphi_i \ox V)((w\ox v_i)\delta(r))=
w\act (\mu\delta(r))=
w\act r,
$$
for any $w\in V$ and $r\in R$. Since the $R$-action on $V$ is surjective by
Theorem \ref{thm:ff}, this means that $\sum_i v_i \act \Phi^i(-)=V$ and hence
$V\ox_R V^*\cong V\ox_R \mathsf{Hom}_R(V,R)\cong \mathsf{Hom}_R(V,V)$ via 
\begin{equation}\label{eq:omega}
\kappa:V\ox_R V^* \to \mathsf{Hom}_R(V,V),\qquad
v \ox_R \psi\mapsto [w\act s\mapsto (\psi \ox V)((w\ox v)\delta(s))].
\end{equation}

We need to show that the evaluation map
$$
\mathsf{Hom}_R(V,R)\ox_R V\to R,\qquad
\Psi\ox_R w\mapsto \Psi(w)
$$
and the coevaluation map
$$
R \to V\ox_R \mathsf{Hom}_R(V,R),\qquad
r\mapsto \sum_i r\act v_i \ox_R \Phi^i=\sum_i v_i \ox_R \Phi^i\act r
$$
are morphisms of right $A$-comodules. Equivalently, using the ($R$-bimodule)
isomorphism $\mathsf{Hom}_R(V,R)\cong V^*$ in Proposition \ref{prop:R/k-dual},
\begin{equation}\label{eq:ev}
\mathsf{ev}:V^* \ox_R V\to R,\qquad
\psi \ox w\act r \mapsto (\psi \ox R)((w\ox 1)\delta(r))
\end{equation}
and 
\begin{equation}\label{eq:coev}
\mathsf{coev}:R\to V \ox_R V^*,\qquad
r \mapsto \sum_i r\act v_i \ox_R \varphi^i=
\sum_i v_i \ox_R \varphi^i\act r
\end{equation}
are morphisms of right $A$-comodules. 

We know from Proposition \ref{prop:product} that the canonical
epimorphism $\pi:V^*\ox V\to V^*\ox_R V$ is a morphism of $A$-comodules. Thus
by Remark \ref{rem:mono-epi}, $\mathsf{ev}$ is a morphism of
$A$-comodules if and only if $\mathsf{ev}\pi$ is so, that is,
\begin{equation}\label{eq:ev_colin}
(\mathsf{ev}\pi\ox A)\lambda^{*S13}(V^*\ox \lambda)
=\lambda_R(\mathsf{ev}\pi \ox A).
\end{equation}
By the idempotency of $A$, by surjectivity of the $R$-actions on $V$
(cf. Theorem \ref{thm:ff}) and Lemma \ref{lem:Pi_spanned}~(3), the domain
$V^*\ox V \ox A$ of these maps in \eqref{eq:ev_colin} is spanned
by elements of the form $\psi\ox \sqcap^R(b)\act v \ox
\sqcap^L(cd)a$, for $v\in V$, $\psi\in V^*$ and $a,b,c,d\in A$. So it
is enough to show that both sides of \eqref{eq:ev_colin} are equal if
evaluated on such elements. 
By Lemma \ref{lem:l-r-bim}~(7) and (1), $\psi\ox \sqcap^R(b)\act v \ox
\sqcap^L(cd)a$ is taken by $V^*\ox \lambda$ to
$$
\psi
\ox v^\lambda \overline\sqcap^R(cd)\ox
\sqcap^R(b)a^\lambda=
\psi \ox v^\lambda \overline\sqcap^R(cd)\ox
S(b^1)a^{\lambda 1},
$$
where the implicit summation index notation $\lambda(v\ox a)=v^\lambda\ox
a^\lambda$ and $T_1(b\ox a)=b^1\ox a^1$ is used and the equality follows by
(6.14) in \cite{BoGTLC:wmba} (cf. \eqref{eq:S_id}). Applying to this
$\lambda^{*S13}$ and using the implicit summation index notation $\varrho(-\ox
b)=(-)^\varrho\ox b^\varrho$, we get
$$
\psi(-^\varrho)\ox v^\lambda \overline\sqcap^R(cd)\ox
S(b^{1\varrho})a^{\lambda 1}.
$$
In light of \cite[Lemma 4.5~(2)]{BoGTLC:wmba}, denoting $T_1(c\ox d)=c^{1'}\ox
d^{1'}$, this is taken by $\mathsf{ev}\pi\ox A$ to 
\begin{eqnarray*}
\psi
\!\!\!\!\!\!\!\!&&\!\!\!\!\!\!\!\!
((v^\lambda\act \overline\sqcap^R(d^{1'}))^\varrho)
\overline\sqcap^R(c^{1'})\ox S(b^{1\varrho})a^{\lambda 1}\\
&=&
\psi(v^{\lambda\varrho})\overline\sqcap^R(c^{1'})\ox 
S(b^{1\varrho}\overline\sqcap^R(d^{1'}))a^{\lambda 1}\\
&=&
\psi(v^{\lambda\varrho})\overline\sqcap^R(c^{1'})\ox 
\sqcap^L(d^{1'})S(b^{1\varrho})a^{\lambda 1}\\
&=&
\psi(v^{\varrho})\overline\sqcap^R(c^{1'})\ox 
\sqcap^L(d^{1'})S(b^{\varrho 1})a^{1}\\
&=&
\psi(v^{\varrho})\overline\sqcap^R(c^1)\ox 
\sqcap^L(d^1)\sqcap^R(b^{\varrho })a\\
&=&
\psi(v^{\varrho})
E(1\ox \sqcap^L(cd)\sqcap^R(b^{\varrho})a)\\
&=&
E(1\ox \sqcap^L(cd) \mathsf{ev}(\psi\ox_R \sqcap^R(b)\act v) a)\\
&=&
E(1\ox \mathsf{ev}(\psi\ox_R \sqcap^R(b)\act v)\sqcap^L(cd) a)\\
&=&
\lambda_R(\mathsf{ev}\pi\ox A)(\psi\ox \sqcap^R(b)\act v \ox \sqcap^L(cd) a).
\end{eqnarray*}
The first equality follows by Lemma \ref{lem:l-r-bim}~(6), the second one does
by \cite[Lemma 6.14]{BoGTLC:wmba}, the third one does by
\eqref{eq:coass-r-prep-f}, the fourth one does by (6.14) in
\cite{BoGTLC:wmba} (cf. \eqref{eq:S_id}) and the fifth one follows since by
(2.3) in \cite{BoGTLC:wmba} and Lemma \ref{lem:Pi_on_E}~(4),
$$
\overline \sqcap^R(c^1)\ox \sqcap^L(d^1)=
(\M(A)\ox \sqcap^L)(E(1\ox cd))=
E(1\ox \sqcap^L(cd)).
$$
In the sixth equality we used that by Lemma \ref{lem:prep_dual_comod}~(1)
and by surjectivity of the right $R$-action on $V$, $(\psi\ox
\sqcap^R)\varrho(v\ox b)= \mathsf{ev}(\psi\ox_R \sqcap^R(b)\act v)$, for any
$v\in V$, $\psi\in V^*$ and $b\in A$. 
In the penultimate equality we made use of (3.9) in \cite{BoGTLC:wmba}.
This proves that \eqref{eq:ev_colin} holds hence $\mathsf{ev}$ is a morphism
of $A$-comodules.

The coevaluation map $\mathsf{coev}:R \to V\ox_R V^*$ factorizes through the
canonical epimorphism $\pi:V\ox V^*\to V\ox_R V^*$ (which is a morphism of
$A$-comodules by Proposition \ref{prop:product}) via the map
$$
\mathsf{coev}': R \to V\ox V^*,\qquad 
r\mapsto  \sum_i r\act v_i \ox \varphi^i.
$$ 
Hence $\mathsf{coev}$ is a morphism of $A$-comodules if and only if 
\begin{equation}\label{eq:coev_colin}
(\kappa\,\pi\,\mathsf{coev}'\ox A)\lambda_R=
(\kappa\,\pi\ox A)\lambda^{13}(V\ox \lambda^{*S})(\mathsf{coev}'\ox A),
\end{equation}
where $\kappa$ is the isomorphism in \eqref{eq:omega}. By \cite[Proposition
6.13]{BoGTLC:wmba}, the domain $R\ox A$ of these maps in
\eqref{eq:coev_colin} is spanned by elements of the form $r \ox S(b)a$, for
$a,b\in A$ and $r\in R$. So it is enough to prove that both sides of
\eqref{eq:coev_colin} are equal if evaluated on such elements. Applying
$\mathsf{coev}'\ox A$ to $r \ox S(b)a$ and omitting the summation symbol for
brevity, we obtain $rv_i\ox \varphi^i\ox S(b)a$. This is taken by $V\ox
\lambda^{*S}$ to $rv_i \ox \varphi^i(-^\varrho)\ox S(b^\varrho)a$ (where the
implicit summation index notation $\varrho(-\ox b)=(-)^\varrho \ox b^\varrho$
is used) and then by $(\pi\ox A)\lambda^{13}$ to $(rv_i)^\lambda \ox_R
\varphi^i(-^\varrho)\ox (S(b^\varrho)a)^\lambda$ (where we introduced the
implicit summation index notation $\lambda(v\ox a)=v^\lambda\ox
a^\lambda$). Applying $\kappa\ox A$ (cf. \eqref{eq:omega}), and introducing
the implicit summation index notation $\delta(s)=s_1\ox s_2$ for the
comultiplication $\delta$ of $R$ and any $s\in R$, we get
$$
[w\act s\mapsto 
\varphi^i((w\act s_1)^{\varrho})
(rv_i)^\lambda \act s_2] \ox (S(b^{\varrho})a)^\lambda.
$$
Now 
\begin{eqnarray}\label{eq:rhs}
\varphi^i((w\act s_1)^{\varrho})
\lambda(rv_i\ox S(b^{\varrho})a)(s_2\ox 1)
&=&
\lambda(r (w\act s_1)^{\varrho}\ox 
S(b^{\varrho})a)(s_2\ox 1)\\
&=&
(1\ox r)(\lambda\lambda^{S21}(ws_1\ox S(b)a))(s_2\ox 1) \nonumber\\
&=&
(1\ox r)E(w\act s_1\ox S(b)a))(s_2\ox 1)\nonumber\\
&=&
E(w\act s\ox r S(b)a)).\nonumber
\end{eqnarray}
In the second equality we applied Lemma \ref{lem:l-r-bim}~(7) and in the
penultimate equality we applied Lemma \ref{lem:lambda_tilde}~(2) (the map
$\lambda^S$ appearing here is that in \eqref{eq:rightS}). The last equality
holds by identity (3.7) in \cite{BoGTLC:wmba} and the fact that the
comultiplication of $R$ splits its multiplication (i.e $s_1s_2=s$). 

On the other hand, applying $\lambda_R$ to $r \ox S(b)a$ we obtain $E(1\ox r
S(b)a)$. Applying to this $\mathsf{coev}\ox A=\pi\,\mathsf{coev}'\ox A$ we get
$(\pi\ox A)[(v_i\ox \varphi^i\ox 1)(1\ox E(1\ox rS(b)a))]$ which is taken by
the isomorphism $\kappa\ox A$ (cf. \eqref{eq:omega}) to   
\begin{equation}\label{eq:lhs}
[w\act s\mapsto v_i\act s_2]\ox 
(\varphi^i \ox A)E_1(ws_1 \ox rS(b)a)=
E(-\ox rS(b)a).
\end{equation}
The first expression is obtained applying Proposition \ref{prop:V^*-R-actions}
and the second one is obtained using again that $s_1s_2=s$. Comparing
\eqref{eq:rhs} and \eqref{eq:lhs}, we conclude that \eqref{eq:coev_colin}
holds hence $\mathsf{coev}$ is a morphism of $A$-comodules too. 
\end{proof}

\section{Hopf modules}\label{sec:Hopf_mod}

This section is devoted to the study of Hopf modules over a regular weak
multiplier bialgebra $A$. These are vector spaces carrying compatible
(non-degenerate idempotent) module and (full) comodule structures. Whenever
the comultiplication is left and right full and there exists an antipode, we
prove the Fundamental Theorem of Hopf Modules. That is, an equivalence between
the category of $A$-Hopf modules and the category of firm modules over the
base algebra $\sqcap^L(A)=\overline \sqcap^L(A)$.
 
\begin{definition}\label{def:Hopf_mod}
For a regular weak multiplier bialgebra $A$, a {\em right-right Hopf module}
is a vector space $V$ carrying the following structures. 
\begin{itemize}
\item $V$ is a non-degenerate right $A$-module with a surjective action
 $\cdot\ : V\ox A\to V$,
\item $(V,\lambda,\varrho)$ is a full right $A$-comodule, 
\end{itemize}
obeying the following equivalent conditions.
\begin{eqnarray}
&&\varrho(\cdot \ox A)=(\cdot \ox A)(V\ox T_3)\varrho^{13}\label{eq:Hr}\\
&&\lambda(\cdot \ox A)=(\cdot \ox A)\lambda^{13}(V\ox T_1).\label{eq:Hl}
\end{eqnarray}
A morphism of Hopf modules is a linear map which is both a morphism of modules
and a morphism of comodules.
The category of $A$-Hopf modules will be denoted by $M^{(A)}_{(A)}$.
\end{definition}

Conditions \eqref{eq:Hr} and \eqref{eq:Hl} are equivalent, indeed: any
one of them asserts precisely that $\cdot$ is a comodule map
(with respect to the comodule structure in Lemma \ref{lem:k_product} on the
domain).

\begin{example}\label{ex:reg_Hopf_mod}
Consider a regular weak multiplier bialgebra $A$ with right full
comultiplication. Via the multiplication $\mu:A\ox A \to A$ and the maps 
$T_1,T_3:A\ox A \to A\ox A$, $A$ is an $A$-Hopf module itself by the first
condition in axiom (v) in Definition \ref{def:wmba}.
\end{example}

\begin{remark}
A usual, unital (weak) bialgebra $A$ can be regarded as a monoid in the
monoidal category of $A$-comodules. An $A$-Hopf module is then precisely a
module over the monoid $A$ in the category of $A$-comodules. 

For a regular weak multiplier bialgebra $A$ with right full comultiplication a
similar interpretation is possible. It follows by Example \ref{ex:A_bim} that
the multiplication $\mu:A\ox A\to A$ factorizes as the canonical epimorphism
$\pi:A\ox A\to A\ox_R A$ composed with an associative $R$-bilinear
multiplication $\mu_R:A\ox_R A \to A$. 
By axiom (v) in Definition \ref{def:wmba} $\mu$ is a morphism of
$A$-comodules; hence so is $\mu_R$ by Remark \ref{rem:mono-epi}. That is to
say, via the multiplication $\mu_R$, $A$ is a (non-unital) monoid in $M^{(A)}$.

For any associative action $\cdot:V\ox A \to V$, it follows by \eqref{eq:G_1}
and by \cite[Lemma 3.7~(2)]{BoGTLC:wmba} that for any $v\in V$ and $a,b\in A$
\begin{equation}\label{eq:.G_1}
\cdot\ G_1(v\ox ab)=
v\cdot \mu^{\mathsf{op}}(A\ox \overline \sqcap^R)T_4(b\ox a)=
v\cdot ab.
\end{equation}
So by the idempotency of $A$ and Proposition \ref{prop:ox_R}, the $A$-actions
$\cdot:V\ox A \to V$ are in a bijective correspondence with the $A$-actions
$\star:V\ox_R A \to V$ such that $\cdot=\star \pi$. Moreover, if $V$ carries a
right $A$-comodule structure $(\lambda,\varrho)$, then \eqref{eq:Hr} 
(equivalently, \eqref{eq:Hl}) asserts precisely that $\cdot$ is a comodule
map. Hence by Remark \ref{rem:mono-epi}, any of \eqref{eq:Hr} and 
\eqref{eq:Hl} is equivalent to $\star$ being a morphism of
$A$-comodules; that is, to $(V,\star)$ being a module over $(A,\mu_R)$ in
$M^{(A)}$.
\end{remark}

\begin{proposition}\label{prop:ind_functor}
Let $A$ be a regular weak multiplier bialgebra with left and right full
comultiplication. Denote by $L$ the coinciding range of the maps $\sqcap^L$
and $\overline \sqcap^L:A\to \M(A)$ and regard $A$ as a left $L$-module via
the multiplication in $\M(A)$. There is a functor $(-)\ox_L A$ from the
category $M_L$ of firm right $L$-modules to the category $M^{(A)}_{(A)}$ of
$A$-Hopf modules.
\end{proposition}

\begin{proof}
Since both $A$ and $L$ are subalgebras of the associative algebra $\M(A)$, the
$A$-action provided by the multiplication of $A$ is a morphism of left
$L$-modules. The maps $T_1$ and $T_3$ are also left $L$-module maps by
\cite[Lemma 3.3]{BoGTLC:wmba}. These observations imply that the $A$-Hopf
module in Example \ref{ex:reg_Hopf_mod} induces an $A$-module structure $P\ox_L
\mu$ and an $A$-comodule structure $(P\ox_L T_1,P\ox_L T_3)$ on $P\ox_L A$,
for any firm right $L$-module $P$. By the surjectivity of $\mu$, this
$A$-action is surjective. Let us see that it is also non-degenerate. Denote by
$\delta:L\to L\ox L$ the comultiplication in the coseparable coalgebra $L$. If
for some $p\in P$, $l\in L$ and $a\in A$ we have $p\act l\ox_L ab=0$ for all
$b\in A$, then also $(p\ox 1)\delta(l)(1\ox ab)\in P\ox A$ is equal to zero
for all $b\in A$. By \cite[Lemma 1.11]{JaVe}, $P\ox A$ is a non-degenerate right
$A$-module hence $0=(p\ox 1)\delta(l)(1\ox a)$. Applying the epimorphism $P\ox
A\to P\ox_L A$ and using that $\delta$ is a section of the multiplication in
$L$, we conclude that $p\act l\ox_L a=0$ as needed. Since the comultiplication
of $A$ is right full by assumption, the comodule $P\ox_L A$ is also full so
that $P\ox_L A$ is an $A$-Hopf module.
For any right $L$-module map $f:P\to P'$, $f\ox_L A$ is evidently a morphism
of Hopf modules. 
\end{proof}

If $A$ is a regular non-weak multiplier bialgebra (i.e. it obeys the
equivalent conditions in \cite[Theorem 2.11]{BoGTLC:wmba}) with left and right
full comultiplication, then the functor in Proposition \ref{prop:ind_functor}
reduces to the functor $(-)\ox A$ from the category of vector
spaces to the category $M^{(A)}_{(A)}$ of right-right $A$-Hopf modules,
possessing a right adjoint $\mathsf{Hom}^{(A)}_{(A)}(A,-)$. An
analogous result fails to hold if $A$ is a weak multiplier bialgebra since then
$\mathsf{Hom}^{(A)}_{(A)}(A,V)$ may not be a firm right $L$-module for an
arbitrary Hopf module $V$. However, if the comultiplication is left and right
full and there exists an antipode for $A$, then the functor $(-)\ox_L
A$ in Proposition \ref{prop:ind_functor} turns out to be even an
equivalence. But its inverse is no longer of the form
$\mathsf{Hom}^{(A)}_{(A)}(A,-)$ as we shall see below. 

\begin{proposition}\label{prop:omega}
Let $A$ be a regular weak multiplier bialgebra with left and right full
comultiplication possessing an antipode $S$. For any right-right $A$-Hopf
module $(V,\cdot,\lambda,\varrho)$, there is a linear map 
$$
\varpi_V:V\to \mathsf{Hom}^{(A)}_{(A)}(A,V),\qquad
v\mapsto \cdot \ \lambda^{S21}(v\ox -),
$$
where $\lambda^S$ is the map from \eqref{eq:rightS}.
\end{proposition}

\begin{proof}
Since $\lambda^{S21}$ is a morphism of right $A$-modules, so is
$\varpi_V(v)$, for any $v\in V$. Let us see that $\varpi_V(v)$ is also a
morphism of $A$-comodules. The proof of this is analogous to \cite[Lemma
4.4]{VDaZha:corep_I} in the non-weak case. For any $v\in V$ and $a,b\in A$, 
\begin{eqnarray*}
\lambda(\varpi_V(v)\ox A)(a\ox b)&=&
\lambda(\cdot \ \lambda^{S21} \ox A)(v\ox a\ox b)\\
&\stackrel{\eqref{eq:Hl}}=&
(\cdot \ \ox A)\lambda^{13}(V\ox T_1)(\lambda^{S21} \ox A)(v\ox a\ox b)\\
&\stackrel{\eqref{eq:.G_1}}=&
(\cdot \ \ox A)(G_1\ox A)\lambda^{13}(V\ox T_1)(\lambda^{S21} \ox A)
(v\ox a\ox b)\\
&=&(\cdot \ \ox A)(\lambda^{S21}\lambda\ox A)\lambda^{13}(V\ox T_1)
(\lambda^{S21} \ox A)(v\ox a\ox b)\\
&\stackrel{\eqref{eq:l-coass}}=&
(\cdot \ \ox A)(\lambda^{S21} \ox A)(V\ox T_1)
(\lambda\lambda^{S21} \ox A)(v\ox a\ox b)\\
&=&(\cdot \ \ox A)(\lambda^{S21} \ox A)(V\ox T_1)(E_1\ox A)
(v\ox a\ox b)\\
&=&(\cdot \ \ox A)(\lambda^{S21} \ox A)(E_1\ox A)(V\ox T_1)
(v\ox a\ox b)\\
&=&(\cdot \ \ox A)(\lambda^{S21} \ox A)(V\ox T_1)(v\ox a\ox b)\\
&=&(\varpi_V(v)\ox A)T_1(a\ox b).
\end{eqnarray*} 
The fourth, sixth and eighth equalities follow by parts (3), (2) and (1) of
Lemma \ref{lem:lambda_tilde}, respectively. In the seventh equality we used
that by 
axiom (vii) in Definition \ref{def:wmba}, for any $v\in V$ and $a,b\in A$,   
\begin{eqnarray*}
(V\ox T_1)(E_1\ox A)(v\ox a\ox b)&=&
((-)\cdot v \ox A \ox A)(R\ox T_1)[E(1\ox a) \ox b]\\
&=&((-)\cdot v \ox A \ox A)[(E\ox 1)(1\ox T_1(a\ox b))]\\
&=&(E_1\ox A)(V\ox T_1)(v\ox a\ox b).
\end{eqnarray*}
\end{proof}

\begin{example}\label{ex:A_omega}
Let $A$ be a regular weak multiplier bialgebra with left and right full
comultiplication possessing an antipode $S$, and consider its right-right Hopf
module $A$ in Example \ref{ex:reg_Hopf_mod}. For any $a,b,c,d\in A$, it
follows form (6.3) in \cite{BoGTLC:wmba}, the anti-multiplicativity of $S$
(cf. \cite[Theorem 6.12]{BoGTLC:wmba}) and \eqref{eq:T_23} that
\begin{eqnarray*}
(d\ox 1)R_1(c\ox S(b)a)
&=&((A\ox S)T_2(d\ox c))(1\ox S(b)a)\\
&=&((A\ox S)((1\ox b)T_2(d\ox c)))(1\ox a)\\
&=&((A\ox S)((d\ox 1)T_3(c\ox b)))(1\ox a)\\
&=&(d\ox 1)((A\ox S)T_3(c\ox b))(1\ox a).
\end{eqnarray*}
Simplifying by $d$, we conclude that
$$
T_1^{S21}(c\ox S(b)a)
\stackrel{\eqref{eq:rightS}}=((A\ox S)T_3(c\ox b))(1\ox a)
=R_1(c\ox S(b)a);
$$
that is, (in view of \cite[Proposition 6.13]{BoGTLC:wmba})
$T_1^{S21}=R_1$. The corresponding map in Proposition \ref{prop:omega} comes
out as 
$$
\varpi_A(a)(b)=\mu R_1(a\ox b)=\sqcap^L(a)b
$$
for any $a,b\in A$, where the last equality follows by \cite[Lemma
6.10]{BoGTLC:wmba}. 
\end{example}

\begin{proposition}\label{prop:omega_surjective}
Let $A$ be a regular non-weak multiplier bialgebra (i.e. a regular weak
multiplier bialgebra satisfying the equivalent assertions in \cite[Theorem
2.11]{BoGTLC:wmba}) with left and right full comultiplication possessing an
antipode $S$. For any right-right $A$-Hopf module $(V,\cdot,\lambda,\varrho)$,
the map $\varpi_V$ in Proposition \ref{prop:omega} is surjective.
\end{proposition}

\begin{proof}
Under the hypotheses of the proposition, it follows by Lemma
\ref{lem:lambda_tilde} that $\lambda^{S21}$ is the inverse
of $\lambda$. In particular, $R_1$ is the inverse of $T_1$. 

Take a morphism of Hopf modules $f:A\to V$. It is a right $A$-module map
satisfying $\lambda(f\ox A)=(f\ox A)T_1$, equivalently, $(f\ox A)R_1=
\lambda^{S21}(f\ox A)$. Therefore for any $a,b\in A$,
$$
\varpi_V(f(a))(b)=
\cdot \ \lambda^{S21}(f(a)\ox b)
=\cdot\ (f\ox A)R_1(a\ox b)
=f\mu R_1(a\ox b)=
\epsilon(a)f(b).
$$
where in the penultimate equality we used that $f$ is a right $A$-module map
and the last equality follows by \cite[Lemma 6.10]{BoGTLC:wmba}. Thus if we
choose $a\in A$ such that $\epsilon(a)=1$, then $\varpi_V(f(a))=f$. 
\end{proof}

\begin{example}
The equivalent assertions in \cite[Theorem 2.11]{BoGTLC:wmba} are really
needed to prove Proposition \ref{prop:omega_surjective}. If they do not hold,
then the discussed map $\varpi_V$ may not be surjective for all Hopf modules
$V$. Consider e.g. an arbitrary groupoid $C$ and the weak multiplier bialgebra
$A:=kC$ in \cite[Example 2.12]{BoGTLC:wmba}, spanned by the morphisms in
$C$. For its right-right Hopf module $A=kC$ in Example \ref{ex:reg_Hopf_mod},
it follows by Example \ref{ex:A_omega} that the range of $\varpi_A$ is the
range of the map $\sqcap^L:A\to \M(A)$; which is spanned by the identity
morphisms in $C$. This is a proper non-unital subalgebra of
$\mathsf{Hom}^{(A)}_{(A)}(A,A)$ (which contains also infinite linear
combinations of the identity morphisms of $C$ regarded as multipliers on
$A=kC$).
\end{example}

\begin{definition}\label{def:coinv}
Consider a regular weak multiplier bialgebra $A$ with left and right full
comultiplication possessing an antipode $S$. For any right-right $A$-Hopf
module $V$, the {\em coinvariant space} is defined as the range
$V^c:=\varpi_V(V)\subseteq \mathsf{Hom}^{(A)}_{(A)}(A,V)$ of the map in
Proposition \ref{prop:omega}. 
\end{definition}

\begin{proposition}\label{prop:coinv_functor}
For any regular weak multiplier bialgebra $A$ with left and right full
comultiplication possessing an antipode $S$, there is a functor $(-)^c$ from
the category $M^{(A)}_{(A)}$ of right-right $A$-Hopf modules to the category
$M_L$ of firm right modules over the base algebra $L:=\sqcap^L(A)=\overline
\sqcap^L(A)$. 
\end{proposition}

\begin{proof}
Since $L$ is a (non-unital) subalgebra of $\mathsf{Hom}^{(A)}_{(A)}(A,A)$, for
any right-right $A$-Hopf module $V$ the vector space
$\mathsf{Hom}^{(A)}_{(A)}(A,V)$ is a right $L$-module via composition: 
$$
(f\act l)(a):=f(la), \qquad \textrm{for}\ f\in
\mathsf{Hom}^{(A)}_{(A)}(A,V),\ l\in L,\ a\in A.
$$
Let us see that the subspace $V^c\subseteq \mathrm{Hom}^{(A)}_{(A)}(A,V)$ is
closed under this action. Using \cite[Lemma 6.14]{BoGTLC:wmba} in the third
equality and Lemma \ref{lem:l-r-bim}~(3) in the penultimate one, for any $v\in
V$ and $a,b,c\in 
A$ we get
\begin{eqnarray*}
(\varpi_V(v)\act \sqcap^L(c))(S(b)a)&=&
\varpi_V(v)(\sqcap^L(c)S(b)a)=
\cdot\ \lambda^{S21}(v\ox \sqcap^L(c)S(b)a)\\
&=&\cdot\ \lambda^{S21}(v\ox S(b\overline\sqcap^R(c))a)
\stackrel{\eqref{eq:lambda_tilde}}=
\cdot[((V\ox S)\varrho(v\ox b\overline\sqcap^R(c)))(1\ox a)]\\
&=&\cdot[((V\ox S)\varrho(\overline\sqcap^R(c)\act v\ox b))(1\ox a)]=
\varpi_V(\overline\sqcap^R(c)\act v)(S(b)a).
\end{eqnarray*}
Thus regarding $V$ as a left $R:=\sqcap^R(A)=\overline \sqcap^R(A)$-module as
in Theorem \ref{thm:ff}, it follows by \cite[Proposition 6.13]{BoGTLC:wmba} that
\begin{equation}\label{eq:L_act}
\varpi_V(v)\act \sqcap^L(c)=\varpi_V(\overline\sqcap^R(c)\act v),\qquad
\forall v\in V, \ c\in A
\end{equation}
proving that $V^c$ is an $L$-submodule of $\mathsf{Hom}^{(A)}_{(A)}(A,V)$. 
Moreover, since the left $R$-action on $V$ is surjective by Theorem
\ref{thm:ff}, \eqref{eq:L_act} also implies that the right $L$-action on $V^c$
is surjective. Since $L$ has local units by a symmetric variant of
\cite[Theorem 4.6~(2)]{BoGTLC:wmba}, this proves that $V^c$ is a firm right
$L$-module. 

For a morphism of Hopf modules $f:V\to V'$, and for $v\in V$ and $a,b\in A$,
\begin{eqnarray*}
f\varpi_V(v)(S(b)a)&=&
f\cdot\ (((V\ox S)\varrho(v\ox b))(1\ox a))\\
&=&\cdot\ (((V'\ox S)\varrho'(f(v)\ox b))(1\ox a))
=\varpi_{V'}(f(v))(S(b)a).
\end{eqnarray*}
Hence (in view of \cite[Proposition 6.13]{BoGTLC:wmba}) there is a map
$$
f^c:V^c \to V^{\prime c}, \qquad
\varpi_V(v)\mapsto f\varpi_V(v)=\varpi_{V'}(f(v)).
$$
It is a morphism of $L$-modules since for any $v\in V$ and $c\in A$,
\begin{eqnarray*}
f(\varpi_V(v)\act \sqcap^L(c))&\stackrel{\eqref{eq:L_act}}=&
f\varpi_V(\overline\sqcap^R(c)\act v)=
\varpi_{V'}f(\overline\sqcap^R(c)\act v)\\
&=&\varpi_{V'}(\overline\sqcap^R(c)\act f(v))\stackrel{\eqref{eq:L_act}}=
(\varpi_{V'}f(v))\act \sqcap^L(c)=
(f\varpi_V(v))\act \sqcap^L(c).
\end{eqnarray*}
In the third equality we used that $f$ is a left $R$-module homomorphism by
Theorem \ref{thm:ff}. 
\end{proof}

Before we can prove the Fundamental Theorem of Hopf Modules, we need the
following technical lemma.

\begin{lemma}\label{lem:zeta}
Consider a regular weak multiplier bialgebra $A$ with left and right full
comultiplication possessing an antipode $S$. For any right-right $A$-Hopf
module $(V,\cdot,\lambda,\varrho)$, there is a linear map
$$
\zeta^0:V\to V^c\ox A,\qquad
v\cdot a \mapsto (\varpi_V\ox A)\lambda(v\ox a),
$$
where $V^c$ is the coinvariant space in Definition \ref{def:coinv} and
$\varpi_V$ is the map in Proposition \ref{prop:omega}. 
\end{lemma}

\begin{proof}
We only need to prove that $\zeta^0$ is a well-defined linear map; that is, it
takes zero to zero. So assume that $v\cdot a=0$. Then using the implicit
summation index notation $T_1(a\ox b)=a^1\ox b^1$ and $\lambda(v\ox
b)=v^\lambda\ox b^\lambda$, it follows for any $b\in A$ that
\begin{eqnarray}\label{eq:zeta_0}
0&=&
(\varpi_V\ox A)\lambda(v\cdot a\ox b)\stackrel{\eqref{eq:Hl}}=
\varpi_V(v^\lambda\cdot a^1)\ox b^{1 \lambda}
=\varpi_V(v^\lambda\overline\sqcap^R(a^1))\ox b^{1 \lambda}\\
&=&(\varpi_V\ox A)\lambda(V\ox \mu(\sqcap^L \ox A)T_1)(v\ox a\ox b)\nonumber\\
&=&(\varpi_V\ox A)\lambda(v\ox ab)=
((\varpi_V\ox A)\lambda(v\ox a))(1\ox b), \nonumber
\end{eqnarray}
so that by non-degeneracy of the right $A$-module
$\mathsf{Hom}^{(A)}_{(A)}(A,V)\ox A$, the expression $(\varpi_V\ox
A)\lambda(v\ox a)$ is equal to zero as needed. 

The third equality in \eqref{eq:zeta_0} follows by the following reasoning. By
\eqref{eq:T_23}, the anti-multiplicativity of $S$ (cf. \cite[Theorem
6.12]{BoGTLC:wmba} ) and (6.14) in \cite{BoGTLC:wmba} (cf. \ref{eq:S_id}), for
any $a,b,c\in A$ 
\begin{eqnarray*}
c(\mu(A\ox S)T_3(a\ox b))&=&
\mu(A\ox S)((c\ox 1)T_3(a\ox b))=
\mu(A\ox S)((1\ox b)T_2(c\ox a))\\
&=&(\mu(A\ox S)T_2(c\ox a))S(b)=
c\sqcap^L(a)S(b)
\end{eqnarray*}
so that $\mu(A\ox S)T_3=\mu(\sqcap^L\ox S)$. Using this identity in the third
equality, 
\cite[Lemma 6.14]{BoGTLC:wmba} in the fourth one and Lemma
\ref{lem:l-r-bim}~(6) in the penultimate one, we obtain 
\begin{eqnarray*}
\varpi_V(v\cdot a)(S(b)c)&\stackrel{\eqref{eq:lambda_tilde}}=&
(v\cdot a)^\varrho \cdot S(b^\varrho)c\stackrel{\eqref{eq:Hr}}=
v^\varrho \cdot (\mu(A\ox S)T_3(a\ox b^\varrho))c=
v^\varrho \cdot \sqcap^L(a)S(b^\varrho)c\\
&=&v^\varrho \cdot S(b^\varrho\overline\sqcap^R(a))c=
(v\act \overline\sqcap^R(a))^\varrho \cdot S(b^\varrho)c=
\varpi_V(v\act \overline\sqcap^R(a))(S(b)c),
\end{eqnarray*}
for any $v\in V$ and $a,b,c\in A$, where $V$ is regarded as a right
$R=\sqcap^R(A)=\overline \sqcap^R(A)$-module as in Theorem \ref{thm:ff} and
the implicit summation index notation $\varrho(v\ox b)=:v^\varrho\ox
b^\varrho$ is used. Applying \cite[Proposition 6.13]{BoGTLC:wmba}, this proves
$\varpi_V(v\cdot a)= \varpi_V(v\act \overline\sqcap^R(a))$ hence the third
equality in \eqref{eq:zeta_0}. The fourth equality in \eqref{eq:zeta_0}
follows by Lemma \ref{lem:l-r-bim}~(1), the penultimate equality follows by
\cite[Lemma 3.7~(3)]{BoGTLC:wmba} and the last one follows since $\lambda$ is
a right $A$-module map, see Proposition \ref{prop:prep_comod}~(1).
\end{proof}

\begin{theorem}\label{thm:FundThm}
Consider a regular weak multiplier bialgebra $A$ with left and right full
comultiplication possessing an antipode $S$. The functors in Proposition
\ref{prop:ind_functor} and Proposition \ref{prop:coinv_functor} are mutually
inverse equivalences. 
\end{theorem}

\begin{proof}
For any firm right $L$-module $P$, 
$$
(P\ox_L A)^c\cong P\ox_L A^c\cong P\ox_L L\cong P
$$
as $L$-modules.
In the second step we applied Example \ref{ex:A_omega} and in the last step we
used that $P$ is firm. This isomorphism is clearly natural in $P$.

For a right-right Hopf module $(V,\cdot,\lambda,\varrho)$, we claim that
$$
\xi:V^c\ox_L A \to V,\qquad 
\varpi_V(v)\ox_L a \mapsto \varpi_V(v)(a)=\cdot\ \lambda^{S21}(v\ox a)
$$
is a natural isomorphism of Hopf modules. It is a morphism of Hopf modules
since $\varpi_V(v):A\to V$ is so by Proposition \ref{prop:omega}. It is
natural in $V$ since for a morphism $f:V\to V'$ of Hopf modules and $a\in A$,
$$
\xi'(f^c\ox_L A)(\varpi_V(v)\ox_L a)= 
f\varpi_{V}(v)(a)=
f\xi(\varpi_V(v)\ox_L a).
$$
The candidate to be the inverse of $\xi$ is 
$$
\zeta:V \to V^c\ox_L A, \qquad 
v\cdot a \mapsto \pi(\varpi_V\ox A)\lambda(v\ox a),
$$
where $\pi$ is the canonical epimorphism $V^c\ox A\to V^c\ox_L A$. It is a
well-defined linear map by Lemma \ref{lem:zeta}. It is the inverse of $\xi$
since by Lemma \ref{lem:lambda_tilde}~(3) and \eqref{eq:.G_1}, for any $v\in
V$ and $a\in A$
$$
\xi\zeta(v\cdot a)=
\cdot \ \lambda^{S21}\lambda(v\ox a)=
\cdot \ G_1(v\ox a)=
v\cdot a
$$
and by Lemma \ref{lem:lambda_tilde}~(2), identity (2.3) in \cite{BoGTLC:wmba},
\eqref{eq:L_act} and by \cite[Lemma 3.7~(3)]{BoGTLC:wmba}, for any $v\in V$
and $a,b\in A$
\begin{eqnarray*}
\zeta\xi(\varpi_V(v)\ox_L ab)&=&
\pi(\varpi_V\ox A)\lambda\lambda^{S21}(v\ox ab)=
\pi(\varpi_V\ox A)E_1(v\ox ab)\\
&=&\pi(\varpi_V\ox A)(((\overline \sqcap^R\ox A)T_1(a\ox b))(v\ox 1))\\
&=&\pi((\varpi_V(v)\ox 1)((\sqcap^L\ox A)T_1(a\ox b)))\\
&=&\varpi_V(v)\ox_L \mu (\sqcap^L\ox A)T_1(a\ox b)=
\varpi_V(v)\ox_L ab.
\end{eqnarray*}
\end{proof}

\end{document}